\documentclass[12pt]{article}
\makeatletter
\let\@fnsymbol\@arabic
\makeatother
\usepackage{comment}
\usepackage{geometry}
\usepackage{times} 
\usepackage{booktabs}
\usepackage{tabularx}
\usepackage{threeparttable}
\usepackage{array}
\usepackage{url}
\usepackage{enumerate}
\usepackage{mathptmx}
\usepackage[T1]{fontenc}
\normalfont
\usepackage[final]{pdfpages}

\usepackage{setspace}

\usepackage{amsmath,amsthm,amssymb}

\usepackage{float}
\usepackage{graphicx}
\usepackage{subfig}
\usepackage{parskip}
\usepackage{natbib}
\usepackage{here,float}
\usepackage{color}
\usepackage[mathscr]{eucal}
 \usepackage{flafter}
\usepackage{titlesec}
\usepackage{titling}
\usepackage{bbm}
\theoremstyle{remark}
\newtheorem{remark}{Remark}
\theoremstyle{plain}
\newtheorem{prop}{Proposition}
\newtheorem{theo}{Theorem}
\newtheorem{exm}{Example}
\newtheorem{lemma}{Lemma}

\newtheorem{Ass}{Assumption}

\begin{titlepage}
\makeatletter 
\let\@fnsymbol\@arabic 
\makeatother

\title{
Near-Unit-Root Theory for Affine Processes
}
\author{Ga\"{e}l Anne\thanks{Department of Mathematics and Statistics, Concordia University, Montr\'eal, Canada. Email: gael.anne@concordia.ca} \, and Yang Lu\thanks{Department of Mathematics and Statistics, Concordia University, Montr\'eal, Canada. Email: yang.lu@concordia.ca} \, and Xuewen Yu\thanks{School of Management, Fudan University, Shanghai, China. Email: xuewenyu@fudan.edu.cn} \, and Xiaowen Zhou\thanks{Department of Mathematics and Statistics, Concordia University, Montr\'eal, Canada. Email: xiaowen.zhou@concordia.ca}}
\end{titlepage}
\begin{document}
\maketitle
\textbf{Abstract:} Discrete-time affine processes are widely used in finance and economics and encompass count, positive, and nonnegative-valued processes. This paper develops near-unit-root asymptotic theory for this class of models. Unlike linear AR(1) processes, affine processes exhibit time-varying conditional variance that remains asymptotically non-negligible near unity, leading to qualitatively different scaling limits and estimator behavior. We show that the local-to-unity regime suffers from the usual nuisance-parameter problem, whereas the mildly explosive regime, while free of it, still does not allow consistent estimation of the intercept. By contrast, the mildly stationary framework is more tractable: the OLS estimator is asymptotically normal, the resulting trajectories are more realistic than those of linear AR(1) models, and inference is possible through both a plug-in method or bootstrap. The theoretical results are supported by simulation evidence and illustrated through applications to insurance and financial data. 

%This paper studies the probabilistic and statistical properties for these processes, under either the local-to-unity, or the mildly integrated framework. The large sample theories we obtain are significantly different from the usual unit root theory for linear autoregressive processes. While the local-to-unity model often suffers from a nuisance parameter problem, the mildly integrated model exhibits better properties, including pivotal large sample distribution for the Ordinary Least Squares estimator, more realistic probabilistic properties compared to linear AR(1) models, and a valid bootstrap procedure. 
%for the local-to-unity case, only the OLS estimator of the autoregressive coefficient is consistent, but not the estimator of the intercept, and a random weighting bootstrap is also invalid, whereas in the mildly integrated case, both the autoregression coefficient and the intercept are consistent, and the bootstrap is also valid. 
%We also investigate alternative estimators such as Weighted Least Squares and Maximum Likelihood. The results are validated using simulations and applied to annual frequency of natural disasters, Bank of Canada rates, and the VIX index.

   \textbf{Acknowledgments:} Lu thanks  Universit\'e Laval/Autorit\'e des march\'es financiers (fonds AMF-GIRIF), NSERC (through grant RGPIN-2021-04144), FRQNT (STRATEGIA grant) and SSHRC (grant 430-2024-01325) and for financial support.  

\section{Introduction}

In several emerging applications, there is a need to model count, positive, or nonnegative-valued processes that may be close to unit root. For example, the annual number of disasters (count valued) may  be increasing due to climate change. Similarly, both long-run and short-run interest rates can exhibit strong persistence \citep{gourieroux2022long, gourieroux2022ultra, gourieroux2025long}.\footnote{
Depending on the application, the short-run interest rate may be positive or nonnegative if the zero lower bound is attainable. In some cases, it may even be count-valued. For instance, the Federal Reserve's policy rate (target range) is typically set in multiples of a minimum tick (25 basis points). In contrast, market-based interest rates, such as the effective federal funds rate, are usually continuous and positive-valued.}

%of winter storms in Canada (count valued), respectively. 
%\begin{figure}[H]
%    \centering
%    \includegraphics[width=0.7\linewidth]{FED.png}
%    \caption{Federal effective fund rates in the US between 1955 and 2025 (Source: Federal Reserve Economic Dataset). }
%\end{figure}

Discrete-time affine processes, also known as compound autoregressive processes \citep{darolles2006structural}, are Markov processes whose conditional Laplace transform (LT) is exponential-affine in the previous state variable. This class encompasses count, positive, and nonnegative-valued processes, and has been widely used in finance and economics. Prominent count-valued examples include the INARCH (or Poisson autoregressive) model \citep{agosto2016modeling} and the negative binomial autoregressive model \citep{gourieroux2017nbar, gourieroux2023noncausal}. A prime example of a positive affine process is the autoregressive gamma process, which arises as the exact discretization of the continuous-time Cox-Ingersoll-Ross process. This property makes it particularly suitable for financial applications such as credit risk, term structure modeling, and stochastic volatility \citep{overbeck1997estimation, gourieroux2006affine, le2010discrete, drea2017, gourieroux2023noncausal, mendes2025gamma}. 
More recently, \cite{monfort2016staying} introduced a nonnegative affine process that allows for a positive probability of reaching the zero lower bound (ZLB). %, a stylized fact frequently observed for interest rate set by central banks over the last few decades. 

Despite these developments, little attention has been paid to near-unit-root affine processes. However, such a framework could be quite useful from a macroprudential perspective. For instance, there is  a recent surge of interest in models capable of capturing long-run uncertainty in long-run interest rates \citep{gourieroux2022long,gourieroux2025long}, with near-unit-root linear models often proposed as candidates. However, these models are generally incompatible with positivity or nonnegative constraints, which are important in many regulatory contexts. %Such near-unit-root models are also very important for prediction purposes, both for the insurance industry in particular, and for the economy in general, in order to better evaluate the costs and benefits of the green transition. 
Such near-unit-root models are also relevant for long-horizon prediction: persistent changes in disaster frequencies and losses matter directly for insurance pricing and, more broadly, for evaluating the costs and benefits of the green transition.

%For instance, the annual count of natural disasters is count valued, the annual average loss of positively valued; the long-run interest rate is positively valued; depending on the application, the short-run interest rate could be either positive valued, or count valued. For instance, the policy rate of many central banks are usually set at a multiple of a minimal tick, such as 25 basis points. 

In this paper, we address this gap by developing a near-unit-root asymptotic theory for a class of Markov affine processes taking count, positive or nonnegative values. We show that these processes exhibit unit root behavior that differs substantially from that of the linear AR(1) model. We begin by analyzing the local-to-unity affine model, in which the autoregressive coefficient is $\alpha_n=1+\frac{\gamma_n}{n}$. The paper that is closest to this framework is \cite{barreto-souza2023}, who study a local-to-unity INARCH model, under the strong assumption that the intercept parameter is fixed \textit{a priori}. We extend their analysis to a class of Markov affine processes taking count, positive, or nonnegative values, including INARCH as a special case, and allow the intercept to be estimated jointly with the autoregressive parameter.  This extension reveals a fundamental nuisance parameter issue of the local-to-unity model: while the autoregressive coefficient remains estimable at rate $n$, neither the intercept parameter of the conditional expectation equation nor the localization parameter $\gamma_n$ is  consistently estimable, and the limiting distribution of the OLS estimator depends on these parameters, which makes inference difficult. Our framework is also distinct from two adjacent near-unit-root
literatures: the linear AR(1) \citep{phillips2007limit, fei2018limit,
yu2025inference} and the INAR(1) \citep{ispany2003asymptotic,
peng2024note}. In both of these existing literatures,  the conditional variance is bounded in expectation. In our setting, the conditional variance grows with the
state and is therefore unbounded, producing qualitatively different
asymptotics.

We then extend the unit root analysis to mildly integrated affine processes, which bridge the gap between the stationary and local-to-unity regimes. The main finding of the paper is that mildly integrated affine processes, particularly its mildly stationary version with $\alpha_n=1-\frac{1}{k_n}$, is more convenient for inference than the local-to-unity model for several reasons. % consistent with similar findings for the linear AR(1) \citep{phillips2023estimation}. 
%Mildly integrated affine processes
First, we show that they avoid the nuisance-parameter problem and we develop a block local-to-unity interpretation for mildly stationary processes to explain this property. Indeed, such a process behaves like $\lfloor n/k_n \rfloor \to \infty$ consecutive local-to-unity blocks, and although each block inherits the non-Gaussian, non-pivotal limit, averaging across the diverging number of asymptotically stationary blocks restores normality and solves the nuisance parameter problem. Second, unlike both local-to-unity and mildly explosive regimes, in which the OLS estimator of the intercept is inconsistent, the mildly stationary model admits feasible inference: the OLS estimator is consistent and jointly asymptotically normal, valid inference can be conducted without prior knowledge of $k_n$, and a
randomly weighted bootstrap is asymptotically valid. Third, it generates more plausible dynamics than the mildly integrated AR(1) models with intercept, as the affine structure preserves nonnegativity and supports richer stochastic variation after rescaling. Fourth, interestingly, this framework can also generate bubble-like paths, without producing a globally explosive pattern associated with the mildly explosive regime, a pattern not suitable for several types of data, such as interest rates. 
%The standard tests of bubbles \citep{phillips2011dating, phillips2015testing} fit local autoregressive models on moving windows. We show in the paper that a mildly stationary affine DGP can naturally produce locally explosive OLS estimates that are nonetheless fully compatible with global mean reversion.

%enjoy similar advantages compared to their local-to-unity alternatives. We also emphasize that these mildly integrated models have the block local-to-unity interpretation, relating their properties to their respective local-to-unity counterparts. Furthermore, we show that when compared to the mildly integrated AR(1) model with intercept, the proposed affine framework leads to more plausible stochastic properties of the process.
%Specifically, \begin{itemize}    \item in the mildly stationary case,  the AR(1) model converges to a deterministic constant after appropriate scaling, a feature that is rarely observed in practice. In contrast, the affine process converges, after appropriate rescaling, to a stationary diffusion.    \item in the mildly explosive case, both models exhibit an exponential trend.  However, the AR(1) process may display either positive or negative explosions. By contrast, the affine model preserves the positivity or nonnegativity of the process. \end{itemize}
 
The paper is organized as follows. Section~\ref{sec:2setting} introduces the affine process within the unit root framework. Section~\ref{sec:3localtounity} shows that the local-to-unity affine model converges to a CIR diffusion and inherits the nuisance-parameter problem. Section~\ref{sec:4probmild} studies the properties of mildly integrated affine processes, and Section~\ref{sec:5inferencemild} focuses on inference of such processes. %To facilitate comparison with the existing unit root literature, this paper mainly focuses on the ordinary least squares (OLS) estimator. 
Section~\ref{sec:6simulations} presents simulation results. Section~\ref{sec:7empirical} proposes applications to three economic datasets. Section~\ref{sec:conclusion} concludes. All proofs are gathered in Appendix A. We also include a short benchmark comparison with the INAR(1) analysis of \citet{peng2024note} using their crime data in Appendix B.

\section{The setting}
\label{sec:2setting}
In this section, we briefly review the class of Markov affine processes and introduce the local-to-unity and mildly integrated assumptions used throughout the paper.
\subsection{Affine processes}
%For a process of count, nonnegative, or positive values, its dynamics are characterized by the conditional LT. It is affine if and only if the LT is exponential-affine in the conditioning variable:
Let $(X_t)$ be a Markov chain taking values in the nonnegative integers, the nonnegative reals, or the positive reals. Its conditional distribution is fully characterized by the conditional Laplace transform (LT), and $(X_t)$ is said to be \emph{affine} when this conditional LT is exponential-affine in the conditioning variable:
\begin{equation}
\label{car}
\mathbb{E}[e^{-u X_t} \mid X_{t-1}]=e^{-a(u)X_{t-1}-b(u)}, \qquad \forall u \geq 0,
\end{equation}
where $e^{-a(u)}$ and $e^{-b(u)}$ are LT's of probability distributions %with the same support as $(X_t)$,
such that the conditional LT is well defined. More precisely, %Depending on the support of process $(X_t)$, they might also satisfy:
\begin{itemize}
%of an infinitely divisible distribution, and another Laplace transform. 
\item if $(X_t)$ is count-valued, then  $e^{-a(u)}$ and $e^{-b(u)}$ are the LT of count distributions, and equation \eqref{car} implies the following stochastic representation:
\begin{equation}
\label{branching}
X_t= \sum_{k=1}^{X_{t-1}} Y_{k,t} +\epsilon_t
\end{equation}
where the sequences $(Y_{k,t})_{k,t \geq 1}$ and $(\epsilon_t)_{t \geq 1}$ are mutually independent i.i.d. sequences with LT $\mathbb{E}[e^{-u Y_{k,t}}]=e^{-a(u)}$ and $\mathbb{E}[e^{-u \epsilon_{t}}]=e^{-b(u)}$, respectively \citep{lu2021predictive}. %In this case, $(X_t)$ is also referred to as a branching process with immigration, or Galton-Watson process with immigration \citep{wei1987unified}, and admits a tractable predictive distribution \citep{lu2021predictive}. 
\item if $(X_t)$ is positive (resp. nonnegative-valued), then the LT $e^{-a(u)}$ must be infinitely divisible in order for \eqref{car} to be well-defined for all values of $X_{t-1}$, whereas $e^{-b(u)}$ may be the LT of any positive (resp. nonnegative) variable.
\end{itemize}

Throughout the paper, we impose the following regularity condition:
\begin{Ass}
\label{ass1}
The first two derivatives of the LTs $e^{-a(u)}$ and $e^{-b(u)}$ exist and are finite at $u=0$.
\end{Ass}
%Because $e^{-a(u)}$ and $e^{-b(u)}$ are LTs,
This assumption ensures that the associated conditional distributions have finite means and variances. Indeed, by successive differentiation of eq.\eqref{car} at $u=0$, we get:
\begin{align}
\label{meancond}
    \mathbb{E}[X_t \mid X_{t-1}]&=a'(0) X_{t-1}+b'(0) ,\\
     \mathbb{V}[X_t \mid X_{t-1}]&=-a''(0) X_{t-1}-b''(0).
     \label{varcond}
\end{align}
%More generally, affine processes are such that all the conditional cumulants, whenever finite, are also affine in the conditioning variable:\begin{align}   \label{3rd}     \kappa_3[X_t \mid X_{t-1}]&=\mathbb{E}\Big[(X_t-\mathbb{E}[X_t \mid X_{t-1}])^3 \vert X_{t-1}\Big]=a^{(3)}(0) X_{t-1}+b^{(3)}(0) ,\\         \kappa_4[X_t \mid X_{t-1}]&=\mathbb{E}\Big[(X_t-\mathbb{E}[X_t \mid X_{t-1}])^4 \vert X_{t-1}\Big]-3 \mathbb{V}[X_t \mid X_{t-1}]=-a^{(4)}(0) X_{t-1}-b^{(4)}(0).\end{align}

%by successively differentiating eq.\eqref{car} once and twice at $u=0$. 
%These results have a simple interpretation in the case of the count-valued CaR process, that is, when  eq.\eqref{branching} holds. In this case, these derivatives of $a$ and $b$ are simply:$$a'(0)=\mathbb{E}[Y_{k,t}], \quad b'(0)=\mathbb{E}[\epsilon_t], \quad -a''(0)=\mathbb{V}[Y_{k,t}], \quad -b''(0)=\mathbb{V}[\epsilon_t].$$

\subsection{Unit root affine processes}
So far, the affine process has been assumed to be time-homogeneous. To accommodate the behavior near the unit root, we now allow functions $a$ and $b$ to depend on the sample size $n$ %which is often, but not necessarily interpreted as sample size $n$ at our disposal. 
as follows:
%In this paper, we assume that the coefficients in front of $X_{t-1}$ in both eqs \eqref{meancond} and \eqref{varcond} depend on the sample size $n$:
\begin{align}
\mathbb{E}[X_t \mid X_{t-1}]&= \alpha_n X_{t-1} +\mu_n, &t=1,\dots,n, \label{eq1}\\
\mathbb{V}[X_t \mid X_{t-1}]&= \beta_n X_{t-1} +\delta_n, &t=1,\dots,n, \label{eq2}
\end{align}
where the coefficients $\alpha_n$, $\mu_n$, $\beta_n$, $\delta_n$ are positive and satisfy the following: %Furthermore, we assume that the sequence $(\beta_n), (\mu_n), (\delta_n)$ are such that: 
\begin{Ass}
\label{Assbeta}
   The sequences $(\beta_n)$, $(\mu_n)$, and $(\delta_n)$ converge, as $n$ increases to infinity, to positive constants $\sigma^2$, $\mu$, and $\delta$, respectively. 
\end{Ass}
%Note that we denote by $\sigma^2$ the limit of $\beta_n$, instead of just $b$, as it will become clear in the next section that $\sigma$ will have the interpretation of volatility. 

We consider both the local-to-unity and mildly integrated regimes for $\alpha_n$: %In the local-to-unity case, parameter $\beta_n$ satisfies: %Assumption 1 below, and $\alpha_n$ satisfying one of Assumptions 2 and 3 below. 

\begin{Ass}[Local-to-unity, or nearly unstable case]
\label{assumptionnearlyunstable}
\begin{equation}
\label{nearlyunstable}
\alpha_n=1+ \frac{\gamma_n}{n},
\end{equation}
where $\gamma_n$ converges to a real constant $\gamma$. If $\gamma<0$, the process $(X_t)$ is  locally stationary, whereas if $\gamma>0$, it is locally explosive. 
\end{Ass}

Alternatively, $\alpha_n$ may follow a mildly integrated specification, in which Assumption \ref{Assbeta} is maintained, but Assumption~\ref{assumptionnearlyunstable} is replaced by the following:
%in the mildly integrated case, we maintain Assumption \ref{Assbeta} and assume that $\alpha_n$ satisfies:
\begin{Ass}[Mildly integrated case]
\label{assumptionmildly}
\begin{equation}
\label{mildly}\alpha_n=1+\frac{\gamma}{k_n},
\end{equation}
where the constant $\gamma \neq 0$, and the sequence $k_n$ is positive and diverges to infinity at a rate slower than $n$. That is,
\begin{equation}
\label{conditiononkn}
k_n \rightarrow \infty, \qquad \frac{k_n}{n} \rightarrow 0.
\end{equation}
The process is mildly stationary when $\gamma<0$, and mildly explosive when $\gamma>0$.
\end{Ass}
\subsection{Examples}
We now present three canonical examples. 
\begin{exm}[INARCH process, or Poisson autoregression]
If the conditional distribution of $X_t$ given $X_{t-1}$ is Poisson,
\begin{equation}
 	\label{inarch}
 X_t \mid X_{t-1} \sim \mathcal{P}ois(\alpha_n X_{t-1}+\mu), \qquad t=1,2,\dots,n,  
\end{equation}
then $a_n(u)=\alpha_n (1-e^{-u})$ and $b(u)= \mu (1-e^{-u})$. Equations \eqref{eq1}-\eqref{eq2} are satisfied with $\beta_n=\alpha_n \rightarrow 1,$ so that $\sigma^2=1$, and $\delta=\mu$. 
\end{exm}

\begin{exm}[ARG and NBAR processes]
%Other examples of affine processes include t
The count-valued negative binomial autoregressive (NBAR) process \citep{gourieroux2017nbar} and the positive-valued autoregressive gamma (ARG) process \citep{pitt2005constructing, gourieroux2006autoregressive} are both based on the Poisson-gamma conjugacy and can alternately be defined using the following Markov chain:
%Let us define processes $(X_t)$ and $(Z_t)$ through
\begin{itemize}
\item given $X_t$, count variable $Z_t$ follows a Poisson distribution with parameter $\theta_n X_t$,
\item given $Z_t$, positive variable $X_{t+1}$ follows a gamma distribution with shape parameter $\kappa+Z_t$ and scale parameter $c$,
\end{itemize}
where parameters $c$, $\theta_n$, and $\kappa$ are positive.

In this construction, the process $(X_t)$ is an ARG process and satisfies \citep[eq.(2.4)]{gourieroux2017nbar}:
$$
\mathbb{E}\left[e^{-u X_t}| X_{t-1}\right]= \frac{1}{(1+cu)^\kappa}e^{-\frac{\theta_n c u}{1+cu}X_{t-1}}, \qquad \forall u>0,
$$
and
\begin{align*}
\mathbb{E}[X_{t} | X_{t-1}]& = \alpha_n X_{t-1} + c\kappa, \\
\mathbb{V}[X_{t} | X_{t-1}]&= 2 \alpha_n c X_{t-1} +c^2 \kappa,
\end{align*}
where $\alpha_n=\theta_n c$. Thus, $(X_t)$ satisfies \eqref{eq1}-\eqref{eq2} with $\sigma^2=2c$, $\mu=c\kappa$, and $\delta=c^2\kappa$. The conditional distribution of $X_t$ given $X_{t-1}$ is noncentral gamma, and is therefore strictly positive. 

%The ARG process is the exact time discretization of a CIR process \citep{gourieroux2006autoregressive}, and is widely used in finance \cite{gourieroux2006affine, le2010discrete, drea2017}.
 
On the other hand, the process $(Z_t)$ is an NBAR process, with conditional LT \citep[eq.(2.3)]{gourieroux2017nbar}:
$$
\mathbb{E}\left[e^{-u Z_t}| Z_{t-1}\right]= \frac{1}{[1+\alpha_n(1-e^{-u})]^{\kappa+Z_{t-1}}}, \qquad \forall u>0,
$$
and
\begin{align*}
\mathbb{E}[Z_{t+1} | Z_{t}]& = \alpha_n (Z_t + \kappa),\\
\mathbb{V}[Z_{t+1} |Z_t]&= (1+ \alpha_n) ( \alpha_n Z_t + \alpha_n \kappa).
\end{align*}
As $\alpha_n \to 1$, $\beta_n \to 2$, so $\sigma^2=2$.  The conditional distribution of $Z_t$ given $Z_{t-1}$ is negative binomial. 
  \end{exm}

  \begin{exm}[ARG0 process]
  Over the past two decades, short-term interest rates in several economies (notably Japan, the Eurozone, and the US) have spent extended periods at or near the zero lower bound. To accommodate the resulting nonzero probability that the rate attains (or approaches) this bound, \cite{monfort2016staying} introduce the ARG0 process as a variant of the standard ARG model.
 % To account for the nonzero probability that the short-term interest rate attains (or approaches) the zero lower bound, a feature observed over the past two decades in several economies, including Japan, the Eurozone, and the US, \cite{monfort2016staying} introduce the ARG-zero process as a variant of the standard ARG model.  
  Given $X_{t-1}$, variable $Z_t$ follows a Poisson distribution with parameter $aX_{t-1} +b$, and conditional on $Z_t$, $X_t$ follows a gamma distribution with shape parameter $Z_t$ and scale parameter $\theta$. By convention, $X_t =0$ when $Z_t=0$. The resulting conditional distribution of $X_t$ given $X_{t-1}$ is a special case of the Tweedie family. 
We have:
  \begin{align*}
      \mathbb{E}[X_t \mid X_{t-1}] &= \theta (a X_{t-1}+b), \\
      \mathbb{V}[X_t \mid X_{t-1}] &= 2\theta^2 (a X_{t-1}+b).
  \end{align*}
  If $\theta$ remains constant while $a = a_n$ varies with $n$ such that $a_n \theta$ converges to 1, then $\sigma^2= 2\theta$.
  \end{exm}

  \subsection{Connection with other near-unit-root frameworks}
  %The motivations of introducing (near) unit root affine processes are twofold. First, as mentioned previously, affine processes are very tractable and  popular in finance and economics. Second,
  \paragraph{Affine models outside our framework.}
  The affine framework offers a natural extension of the existing unit root literature: the two classical models for which a unit root theory is well-established both belong to the affine family. However, neither satisfies Assumption~\ref{Assbeta}. More precisely: 
  \begin{itemize}
      \item the linear AR(1) model with intercept $X_t=\alpha_n X_{t-1}+\mu+\epsilon_t$, with i.i.d. innovations $\epsilon_t$, is affine with $a(u)$ given by $a(u)=u \alpha_n $, and %$b(u)=\log \mathbb{E}[e^{-u \epsilon_t}]$ in the case without intercept, or 
      $b(u)=\mu u-\log \mathbb{E}[e^{-u \epsilon_t}]$. Its limit theory has recently been developed by \cite{fei2018limit,liu2019asymptotic, guo2019testing, liu2023limit, yu2025inference}. In this case, $\beta_n=-a''(0)=0$. 
      \item the INAR(1) model is a special case of \eqref{branching}, obtained when $(Y_{k,t})$ is Bernoulli with parameter $\alpha_n$, in which case $a(u)=-\ln(\alpha_n e^{-u}+1-\alpha_n)$. Local-to-unity and mildly integrated INAR(1) models have been studied by \cite{ispany2003asymptotic} and \cite{peng2024note}, respectively, and exhibit convergence rates similar to those of the linear AR(1) model.\footnote{For example, in the local-to-unity case, both models have convergence rates of $n^{3/2}$ for $\hat{\alpha}_n$, and $n^{1/2}$ for $\hat{\mu}_n.$ } The conditional variance is  $\alpha_n(1-\alpha_n)X_{t-1}+\mathbb{V}[\epsilon_t]$, so $\beta_n=\alpha_n(1-\alpha_n) \to 0$ as $\alpha_n$ increases to 1. 
   \end{itemize}
Because $\beta_n$ converges to zero in both models, the conditional variance is bounded in expectation: it is constant in the linear AR(1) model. In the INAR(1) model, since $\mathbb{E}[X_{t-1}]=O(\frac{1}{1-\alpha_n})$, it follows that $\alpha_n(1-\alpha_n)X_{t-1}$ is bounded in expectation. 
      %Thus, for this model, the conditional variance is bounded in expectation.   
%  Thus both models correspond to the (disallowed) limiting case of Assumption 1 with $\beta_\infty=0$.   %In this paper, we shall see that for the affine processes that we consider, the large sample theory differs significantly from these two limiting cases. 
 % Note also that the INAR(1) model implies under-dispersion, that is, in this model, the conditional variance is smaller than the conditional expectation, at least in the case where $\epsilon_t$ is Poisson distributed. In practice, many count data are over-dispersed. For this reason, we also do not consider the INAR(1) model in our paper.  
By contrast, under Assumption \ref{Assbeta}, %have explosive conditional variance, %as well as the  mildly nonstationary AR(1) literature \citet{phillips2007limit, , yu2025inference}, mainly due to the fact  in the sense that
 the conditional variance %, which by \eqref{eq2} is 
 $\beta_n X_{t-1}+\delta_n$ has an unbounded expectation. %Indeed, it will be shown in this paper that in these models, $X_{t-1}$ is typically of order $n$ (under Assumption 2) or $k_n$ (under Assumption 3). 
 %As a comparison, for both INAR(1) and AR(1) models, the conditional variance is bounded in expectation. 
%This is to be compared with the linear AR(1) literature, which usually assumes the innovation process to be stationary. That is, $\beta_n=0$.  
%While this resembles a weak AR(1) model, 
%under Assumption 2 or 3, the sequence $(W_t)$ is usually not weakly stationary, whereas it is typically stationary in the aforementioned AR(1) models. 
%On the other hand, the literature on near unit root AR(1) processes usually assumes that the marginal variance of $e_t$ is bounded, whereas \eqref{eq2} indicates an unbounded marginal variance for process $(W_t)$. 
%It will be shown in this paper that 
In this paper, we show that this ``explosive'' feature  leads to fundamentally different asymptotic properties for the near-unit-root affine processes considered here, compared with the linear AR(1) and INAR(1) models. %Our framework also does not include local-to-unity, or mildly integrated INAR(1) model. 
\paragraph{Non-affine nonlinear near-unit-root models.}
Other nonlinear near-unit-root models pursue similar goals through
different mechanisms. The hybrid stochastic local unit root with drift
(STURWD) of \citet{lieberman2020hybrid, liu2023robust}
\begin{equation}
\label{hybrid}
Y_t=  \left(1+\frac{\gamma}{n}\right) Y_{t-1} +\mu+ \frac{u_t}{\sqrt{n}}Y_{t-1} +\epsilon_t,
\end{equation}
where $(u_t, \epsilon_t)$ is a stationary error sequence, has \emph{quadratic} (rather than affine) conditional variance,
\emph{observable} (rather than latent) shocks, and does not naturally
accommodate count-valued data.

On the other hand, the limited autoregressive
\citep{cavaliere200403} and dynamic Tobit \citep{bykhovskaya2024local}
processes enforce nonnegativity by truncating a linear process
at zero. It is well documented \citep{monfort2016staying} in the term structure literature that this approach, also referred to as the \emph{shadow rate}, can lead to reduced analytical tractability, for example, in pricing and risk prediction applications.\footnote{For instance, bond pricing formulas when the interest rate is the underlying process.} %and for risk prediction purposes.  

\subsection{Ordinary Least Squares estimator}
%As $n$ increases to infinity and $\alpha_n$ approaches 1, $\beta_n$ converges to zero. Thus Assumption 1 is not satisfied by the INAR(1) model. 
%For an affine process satisfying Assumptions 1, 2, and 3 (or 4), this paper studies its stochastic properties, with particular emphasis on its scaling limit and the asymptotic distribution of the estimators of parameters $\alpha_n$ and $\mu_n$ in the conditional expectation equation \eqref{eq1}.
For an affine process satisfying Assumptions~\ref{ass1}, \ref{Assbeta}, \ref{assumptionnearlyunstable} (or \ref{assumptionmildly}), we study the scaling limit of $(X_t)$ and the asymptotic distribution of the estimators of $\alpha_n$ and $\mu_n$ in \eqref{eq1}. We mainly focus on the ordinary least squares (OLS) estimator, given by:
\begin{equation}
\label{cls_eq}
(\hat{\alpha}_n, \hat{\mu}_n)':=\arg\min_{(\alpha, \mu)'} \sum_{t=1}^n (X_t-\alpha X_{t-1}-\mu)^2.
\end{equation}
Two considerations motivate this choice. First, it facilitates comparison with the existing unit root literature. Second, its expression is considerably simpler than that of alternatives such as the (quasi-)maximum likelihood estimator. Weighted least squares and maximum likelihood are discussed in Section~\ref{sec:5inferencemild}.
%The choice of OLS estimation is motivated by two considerations. First, it facilitates the comparison with the existing unit root literature. Second, its expression is considerably simpler than that of alternative estimators, such as the maximum likelihood estimator. A discussion of weighted least squares and maximum likelihood estimation is deferred to section 5. % We show that their properties are more nonstandard and are thus not necessarily superior to the OLS estimator. %We also explore the use of the bootstrap method to approximate the limiting distribution in finite samples.  
Table~\ref{tab:regime-summary} previews the main results of the next two sections.
\begin{table}[H]
\centering
\small
\setlength{\tabcolsep}{4pt}
\begin{threeparttable}
\caption{Summary of the three near-unit-root regimes and their inferential implications. The symbol $\Rightarrow$ denotes weak convergence of stochastic processes.}
\label{tab:regime-summary}
\begin{tabularx}{\textwidth}{
>{\raggedright\arraybackslash}p{2.2cm}
>{\raggedright\arraybackslash}p{2.8cm}
>{\raggedright\arraybackslash}p{3.1cm}
>{\raggedright\arraybackslash}p{1.7cm}
>{\raggedright\arraybackslash}p{1.8cm}
>{\raggedright\arraybackslash}p{2.6cm}
>{\raggedright\arraybackslash}X}
\toprule
Regime & Scaling  limit  & Rate of $\hat\alpha_n$ & Rate of $\hat\mu_n$ & Identifiability/Estimability of $\gamma$ % & Random-weight Bootstrap validity
\\
\midrule

Local-to-unity

$\alpha_n = 1 + \gamma/n$
&
$(X_{\lfloor ns\rfloor}/n) \Rightarrow  (\Upsilon_s)$, a CIR process
&
$n$
&
inconsistent
&
Nuisance parameter
%&Invalid
\\
\hline

Mildly stationary

$\alpha_n = 1 + \gamma/k_n$, \ $\gamma<0$, \ 
&
$(X_{\lfloor k_n s\rfloor}/k_n) \Rightarrow  (\Upsilon_s)$, stationary in the long-run
&
$\sqrt{n k_n}$
&
$\sqrt{n/k_n}$
&
Non-identifiable; normalized to $\gamma=-1$
%&Valid for the random-weight bootstrap in Theorem 6(b)
\\
\hline

Mildly explosive

$\alpha_n = 1 + \gamma/k_n$, \ $\gamma>0$, \ 
&
$(X_{\lfloor k_n s\rfloor}/k_n) \Rightarrow  (\Upsilon_s)$, explosive in the long-run
%after exponential rescaling, $X_{\lfloor k_n s\rfloor}/(k_n e^{s}) \Rightarrow Z$
&
$k_n {\alpha_n^{n/2}}$
&
inconsistent
&
Non-identifiable; Normalized to $\gamma=1$
%&Not established in the current draft
\\

\bottomrule
\end{tabularx}
\end{threeparttable}
\end{table}

\section{The local-to-unity model}
\label{sec:3localtounity}
This section studies the local-to-unity regime. Throughout, we maintain Assumptions~\ref{ass1}--\ref{assumptionnearlyunstable} and impose the following additional moment condition, which is used to obtain the diffusion approximation and the subsequent limit theory for the OLS estimator.
%In this section, we consider affine processes satisfying Assumptions 1-3. We also need the following moment condition. 
\begin{Ass}[Finite conditional moment of order $p$]
\label{AssAffineMoment}
There exists $p>2$ such that 
\[
\sup_n \mathbb E[Y_n^p] <\infty,
\qquad
\sup_n \mathbb E[\eta_n^p] <\infty,
\]
where $Y_n$ and $\eta_n$ denote random variables with LT's $\mathbb E[e^{-uY_n}]=e^{-a_n(u)}$ and
$\mathbb E[e^{-u\eta_n}]=e^{-b_n(u)}$, respectively.
\end{Ass}

\subsection{Scaling limit}
\begin{theo} 
\label{scaling}
Consider an affine model satisfying \eqref{eq1}-\eqref{eq2} with initial value $X_0=o(n)$, as well as Assumptions~\ref{ass1}, \ref{Assbeta}, \ref{assumptionnearlyunstable} and \ref{AssAffineMoment}. Then the rescaled process $(Y_s:=\frac{X_{\lfloor{ns}\rfloor}}{n}, s>0)$ converges weakly to the CIR process $(\Upsilon_s, s>0)$ defined by
\begin{equation}
\label{sde}
\mathrm{d}\Upsilon_s=(\mu+\gamma \Upsilon_s) \mathrm{d}s+ \sigma \sqrt{\Upsilon_s}\mathrm{d} B_s, \qquad \Upsilon_0=0,
\end{equation}
where $(B_s)$ is a standard Brownian motion.
\end{theo}
Theorem~\ref{scaling} shows that local-to-unity affine processes admit a continuous-time limit of CIR type, rather than the Ornstein--Uhlenbeck limit familiar from linear AR(1) models. This difference reflects the fact that the conditional variance grows proportionally with the current state. %As a result, the affine framework preserves positivity or nonnegativity in the limit, while still allowing for near-unit-root persistence.
The limiting CIR process, which is itself an affine process in continuous time\footnote{That is, its conditional LT $\mathbb{E}\left[e^{-uX_{t+h}}\vert X_t\right]$ is exponential-affine in $X_t$ for any $t\geq 0$ and $h>0$.}, is the benchmark model in the term structure of interest rates \citep{duffie2000transform}.

The intuition behind this convergence result is as follows. The conditional expectation of the rescaled process satisfies 
$$
\mathbb{E}\left[Y_{s+\frac{1}{n}}\vert Y_s\right]=\frac{1}{n}\mathbb{E}\left[X_{\lfloor n(s +\frac{1}{n})\rfloor} \vert X_{\lfloor ns\rfloor}\right]=\frac{1}{n} (\alpha_n  X_{\lfloor ns\rfloor}+\mu_n)=\alpha_n Y_s+\frac{\mu_n}{n},$$
so that
$$\mathbb{E}\left[Y_{s+\frac{1}{n}}-Y_s\vert Y_s\right]=\frac{1}{n }(\gamma Y_s+\mu_n)=\frac{1}{n }(\gamma Y_s+\mu)+o(1/n).$$
Similarly, the conditional variance is given by 
$$\mathbb{V}\left[Y_{s+\frac{1}{n}}-Y_s\vert Y_s\right]=\frac{1}{n} \left(\beta_n Y_s+ \frac{\delta_n}{n}\right)= \frac{1}{n} \sigma^2 Y_s+o(1/n),$$
hence the linear drift and volatility functions in \eqref{sde}. 
\begin{proof}
    See Appendix~\ref{proofscaling}.
\end{proof}
%Note that if one retains only the conditional mean and variance conditions \eqref{eq1}-\eqref{eq2} and Assumptions 2-3, but drops the affine structure (and Assumption 1), then the weak convergence in Theorem \ref{scaling} is no longer guaranteed. The affine property therefore provides a convenient sufficient condition for convergence. More generally, it may be replaced by alternative conditions, such as the third and fourth conditional cumulants of $X_t$ given $X_{t-1}$ being bounded by affine functions of $X_{t-1}$.\footnote{If $(X_t)$ is affine and the third and fourth derivatives of the LTs $e^{-a}$, $e^{-b}$ exist at zero, then the third and fourth cumulants are also affine, similarly to \eqref{eq1}-\eqref{eq2}. See, for example, \cite{michel2020limiting} for a related convergence result based on higher order conditional moments. } Our proof relies on the affine structure, but does not require the existence of higher-order conditional moments.

Finally, note that the diffusion process defined in \eqref{sde} is well-defined (i.e., nonnegative) for any positive time $s$, and not only for $s\in[0,1]$. This diffusion may or may not reach 0 again after the starting time $s=0$, depending on the relation between $\mu$ and $\sigma^2$, we have \citep[p.391]{cox1985theory}:
\begin{itemize}
    \item If $2 \mu \geq \sigma^2$, then the process $(\Upsilon_s)$ remains strictly positive almost surely. This is known as the \emph{Feller condition}.
\item If $2 \mu < \sigma^2$, then zero will be revisited almost surely. This feature is particularly relevant in applications where the initial process $(X_s)$ frequently takes small or zero values, such as the short-term risk-free interest rate. 
\end{itemize}
%The limiting CIR process, which is an affine process in continuous time\footnote{That is, its conditional LT $\mathbb{E}[e^{-X_{t+h}}\vert X_t]$ is exponential affine in $X_t$ for any real date $t$ and real positive maturity $h$.}, is the benchmark model in the term structure of interest rates \citep{duffie2000transform}. This limit is different from the usual Ornstein-Uhlenbeck (OU) limiting process obtained for a linear process without intercept \citep{phillips1987towards}.

Because of the time and scale normalization, Theorem~\ref{scaling} can be interpreted as providing a macroscopic description of the process $(X_t)$. In contrast, Theorem~\ref{general} provides a high-frequency counterpart. 
\begin{theo}
\label{general}
    For any affine process $(X_t)$ satisfying \eqref{eq1}-\eqref{eq2}, and given $t$, $n$, the conditional distribution of $X_t$ given $X_{t-1}=x$ is approximately normal for large $x$:
    \begin{equation}
\label{infinitelydivisible}
\left(\frac{X_t-\mathbb{E}[X_t\mid X_{t-1}=x]}{\sqrt{\mathbb{V}[X_t\mid X_{t-1}=x]}}\vert X_{t-1} =x \right)\overset{w}{\longrightarrow} \mathcal{N}(0, 1)
\end{equation}
as $x$ increases to infinity, where $\overset{w}{%
\longrightarrow }$ denotes weak convergence of a distribution.
\end{theo}
 %, which explains the same asymptotic distribution of WLS and MLE. 
\begin{proof}
    See Appendix~\ref{proofgeneral}.
\end{proof}
%This proposition can be interpreted as an equivalent form of the weak convergence result derived for process $(\frac{X_{\lfloor k_ns \rfloor}}{k_n})$. 

Theorem~\ref{general} can also be interpreted through the lens of the central limit theorem (CLT). For expository purposes, consider the case where the functions $a$ and $b$ satisfy $a=b$ and $x$ is an integer. In this setting, $\mathbb{E}\left[e^{-uX_{t+1}} \vert X_t=x\right]=e^{-(x+1)a(u)}$, which implies that $X_{t+1}$, conditional on \(X_t = x\), can be represented as the convolution of \((x+1)\) i.i.d. random variables with LT \(e^{-a(u)}\). Therefore, the normalized quantity in \eqref{infinitelydivisible} corresponds directly to the classical CLT.%, since it involves the sum of an increasing number of i.i.d. terms as \(x\) increases to infinity.

\subsection{Large sample theory for the OLS estimator}
%Theorem~1 leads to the following large-sample properties of the OLS estimator.
The diffusion limit in Theorem~\ref{scaling} yields the following large-sample distribution for the OLS estimator.
\begin{theo} 
\label{thm1}
Consider Model~\eqref{eq1}-\eqref{eq2}, with initial condition $X_0=o(n)$, and suppose Assumptions~\ref{ass1}, \ref{Assbeta}, \ref{assumptionnearlyunstable} and \ref{AssAffineMoment} hold. Then, as $n$ increases to infinity, the joint distribution of 
$
\begin{bmatrix}n (\hat{\alpha}_n-\alpha_n)\\
\hat{\mu}_n- \mu_n 
 \end{bmatrix}
$
converges weakly:
 \begin{equation}
 \label{cmt}
 \begin{bmatrix} 
    n(\hat{\alpha}_n-\alpha_n)\\
    \hat{\mu}_n - \mu_n 
    \end{bmatrix} \overset{w}{\longrightarrow}
 \begin{bmatrix}
         \int_0^1 \Upsilon_s^2 \mathrm{d}s &     \int_0^1 \Upsilon_s \mathrm{d}s\\
         \int_0^1 \Upsilon_s \mathrm{d}s & 1
 \end{bmatrix}^{-1} 
 \begin{bmatrix}
   \sigma  \int_0^1 \Upsilon_s^{3/2} \mathrm{d}B_s\\
   \sigma \int_0^1 \Upsilon_s^{1/2} \mathrm{d}B_s
 \end{bmatrix}.
\end{equation}
 
\end{theo}
Note that if the process $(X_t)$ starts from $X_0=cn$ instead, the result still holds, except that the limiting diffusion process in \eqref{sde} starts from $\Upsilon_0=c$. 
\begin{proof}
    See Appendix~\ref{proofthm1}.
\end{proof}
 %This is to be compared with \cite{ispany2003asymptotic, ispany2014asymptotic}, who show that in the nearly unstable INAR(1) process, the asymptotic distribution of the estimator $(\alpha, \mu)$ is jointly Gaussian and is consistent. 

%\begin{remark}

This rate is slower than in the local-to-unity INAR(1) \citep{ispany2003asymptotic} and linear AR(1) models with intercept \citep{fei2018limit}, where the slope and intercept estimators typically converge at rates $n^{3/2}$ and $n^{1/2}$, respectively. The difference is due to the state-dependent affine variance: when $X_{t-1}$ is of order $n$, the conditional innovation variance is also of order $n$, reducing the effective information about the conditional mean parameters.

Theorem~\ref{thm1} also shows that local-to-unity affine models inherit the main inferential drawback of local-to-unity linear AR(1) models. Indeed, the natural estimator of parameter $\gamma$ is $\hat{\gamma}_n= -n(1-\hat{\alpha})$, which is inconsistent, since $\hat{\alpha}_n$ is only $n$-consistent. Moreover, the intercept estimator $\hat\mu_n$ is also inconsistent. Because the asymptotic distribution of $\hat{\alpha}_n$ depends on both $\gamma$ and $\mu$, the asymptotic distribution derived in Theorem~\ref{thm1} is not directly usable for inference. These limitations motivate the alternative, mildly integrated framework. 

\section{Probabilistic properties under mild integration}
\label{sec:4probmild}
This section studies the probabilistic implications of the mildly integrated specification in Assumption~\ref{assumptionmildly}. The key difference from the local-to-unity framework is that the relevant time horizon is \(n/k_n\), which diverges. As a result, the long-run behavior of the limiting CIR diffusion determines the large-sample properties of the triangular array.

%\section{Mildly integrated affine processes: probabilistic properties}
%As shown in the previous section, the local-to-unity affine model exhibits several  well-known drawbacks of local-to-unity asymptotics, including the impossibility of consistently estimating $\gamma$. For this reason, we now turn our attention to the alternative mildly integrated framework.

%Recently, \cite{liu2023asymptotic, peng2024note} replace, for an INAR(1) process, the nearly unstable condition $\alpha_n=1-\frac{\gamma}{n}$ by Similar mildly integrated models have also been considered for AR(1) processes \citep{phillips2007limit, yu2025inference}.
%\subsection{The main result}
First, note that the parameterization in \eqref{mildly} identifies the sign of \(\gamma\), but not its magnitude. Indeed, replacing \((\gamma,k_n)\) by \((c\gamma,ck_n)\), for any \(c>0\), leaves \(\alpha_n\) unchanged. We therefore impose the following normalization.
%\eqref{mildly} is invariant to multiplying both the numerator and denominator of $\frac{\gamma}{k_n}$ by the same positive constant. In other words, the parameter $\gamma$ is not identifiable in model \eqref{mildly}. As a consequence, we have the following proposition:
\begin{prop}
 Under Assumption~\ref{assumptionmildly}, without loss of generality, we may set $\gamma = 1$ if $\gamma > 0$ and $\gamma = -1$ if $\gamma < 0$.
\end{prop}
The values $1$ and $-1$ are referred to as \emph{pseudo-true values} \citep{phillips2023estimation}. This normalization serves as an identification condition and will be maintained throughout the remainder of the paper. Accordingly, we adopt the following convention: 
%Throughout the remainder of the paper, we impose this normalization condition:
\[\gamma = \left\{ \begin{array}{rl} 1,& \text{if $\gamma >0$}, \\ -1, & \text{if $\gamma <0$.} \end{array} \right. \]

 \subsection{Scaling properties}
 We have the following analog of Theorem~\ref{scaling}:

\begin{prop}
\label{scalinglimit}
 Under Assumptions~\ref{ass1}, \ref{Assbeta}, \ref{assumptionmildly} and \ref{AssAffineMoment}, as well as the initial value condition $X_0=o(k_n)$, the rescaled process $(Y_s, s > 0) = \left(\frac{X_{\lfloor k_n s \rfloor}}{k_n}, s > 0\right)$ converges weakly in \(D([0,\infty[)\) to the diffusion process defined in \eqref{sde}, with parameter $\gamma = 1$ if $\gamma > 0$, and $\gamma = -1$ if $\gamma < 0$.
\end{prop}
In other words, a mildly stationary process $(X_t)$ behaves like a local-to-unity affine process, up to both a time and scale transformation. This result is a direct consequence of Theorem~\ref{scaling}. Indeed, restricting attention to the first $\lfloor k_n \rfloor$ observations of $(X_t)$ yields a local-to-unity system, to which Theorem~\ref{scaling} applies.\footnote{This interpretation of the mildly integrated model as a sequence of local-to-unity blocks originates from \cite{phillips2001estimate, phillips2010smoothing}.} Therefore, the entire sample $X_t, t=1,\dots,n$ can be viewed as a local-to-unity model with autoregressive coefficient $1\pm \frac{1}{k_n}$, but with $n$ observations instead of $k_n$ observations. %Then we recall that Theorem \ref{thm1} can also be applied to a time series with a much larger sample size of $n$ instead of $k_n$. 

 By Proposition~\ref{scalinglimit}, the analysis of $X_t, t=1,\dots,n$, now involves the behavior of the diffusion \eqref{sde} over the interval $[0, n/k_n]$, whose upper bound increases to infinity as $n$ increases. This contrasts with the previous local-to-unity case, for which the rescaled process only depends on the trajectory of the diffusion on $[0,1]$. Therefore, we study the long-run behavior (i.e., as $s \uparrow \infty$) of this diffusion. We have the following:
  %   If $\gamma>0$, then $(\Upsilon_t)$ is asymptotically explosive at an exponential rate: %Intuitively, in the three terms on the right hand side of eq.\eqref{sde}, the second term $\gamma \Upsilon_t \mathrm{d}t$ is dominant. Thus we have, informally, $\mathrm{d} \Upsilon_t \approx \gamma \Upsilon_t \mathrm{d}t$. More precisely, we have the following result:
    \begin{prop}[\cite{going2003survey}]
    \label{stationaritycir}
    \begin{enumerate}[$i)$]
        \item If $\gamma<0$, that is, if the process $(X_t)$ is mildly stationary, then the CIR process $(\Upsilon_s)$ defined in eq.\eqref{sde} is asymptotically stationary and nonnegative, with a gamma invariant distribution.
        
        \item If $\gamma>0$ and is normalized to 1, then $(\Upsilon_s)$ is asymptotically explosive at an exponential rate. More precisely, $\frac{\Upsilon_s}{e^{\gamma s}}$  converges almost surely to a random variable $Z$ that is gamma distributed with shape parameter $\frac{2\mu}{\sigma^2}$ and scale parameter $\frac{\sigma^2}{2}$. In particular, it has mean $\mu$ and variance $\frac{\mu \sigma^2}{2}$. 
    \end{enumerate}    
    \end{prop}

%    \begin{proof}       Part $i)$ and $ii)$ can be found in \cite{going2003survey}, eq.(4). Part $iii)$ is due to the fact that the conditional distribution of $\Upsilon_t$ given $\Upsilon_0=0$ is chi-square with degrees of freedom 4, and scale parameter $2(e^t-1)$. Thus $Z$ is gamma with shape parameter 2 and rate parameter 2.    \end{proof}

   \subsection{Comparison with mildly integrated AR(1)} 
Let us now compare the mildly stationary affine model with the mildly stationary AR(1) process with intercept studied in \citep{fei2018limit}:
\begin{equation}
    \label{feimodel}
    X_t = \alpha_n X_{t-1} + \mu + \epsilon_t, \qquad t = 1, \ldots, n, \quad X_0 = 0,
\end{equation}
where $(\epsilon_t)$ is i.i.d. with mean zero and variance $\sigma^2$, and $\alpha_n=1-\frac{1}{k_n}$, with $\mu \neq 0$. By iteration (see \cite{fei2018limit}, Equation (5)), we obtain
$$
X_t=\mu k_n (1-\alpha_n^t)+ \sum_{j=1}^t \alpha_n^{t-j} \epsilon_j. 
$$
The first term is deterministic, while the second term $X_{0,t}:=\sum_{j=1}^t \alpha_n^{t-j} \epsilon_j$ is a mildly stationary AR(1) process without drift in the sense of \cite{phillips2007limit}. Equivalently, by the time-change argument following Proposition~\ref{scalinglimit}, it can be viewed as a local-to-unity linear AR(1) model with autoregressive coefficient $1-\frac{1}{k_n}$, observed over a horizon of order $n$. By Lemma~1 of \cite{phillips1987towards}, we obtain: 
\begin{prop} 
\label{feilm1}
    In model \eqref{feimodel}, under the mildly stationary assumption $\alpha_n=1-\frac{1}{k_n}$,  the rescaled process $\left(\frac{ X_{\lfloor k_n s \rfloor }}{k_n}, s\geq0\right)$ converges weakly to the deterministic process $\mu(1-e^{-s})$ as $n \to \infty$. 
   %     \item in the mildly stationary case $\alpha_n=1+\frac{1}{k_n}$, the rescaled process $(\frac{ X_{\lfloor k_n s \rfloor }}{k_n e^{s}})$ converges to the constant process $\mu$ as $n$ increases to infinity.  
\end{prop}
The deterministic limit converges rapidly to $\mu$ as $s$ increases. This feature contrasts with its affine counterpart, whose weak limit is the diffusion process $(\Upsilon_s)_s$ and thus allows for non-trivial variation. Indeed, even though both models share the same mean dynamics $(\mathbb{E}[X_t])_t$, their asymptotic behaviors differ substantially. For the affine model, Propositions~\ref{scalinglimit} and \ref{stationaritycir} show that, for both mildly stationary and mildly explosive affine processes, the limiting distribution of $\left(\frac{ X_{\lfloor k_n s \rfloor }}{k_n}\right)$ is a nondegenerate gamma random variable. In particular, compared to our affine model, the mildly stationary AR(1) model with intercept exhibits substantially lower variability beyond the first few observations. 

As an illustration, Figure~\ref{fig:mild.int} displays sample paths of  a mildly stationary INARCH process and a mildly integrated AR(1) process with intercept. For both processes, we set \(k_n=100, n/k_n=100\), \(\mu=1\), \(\sigma=1\), and \(\gamma=-1\). The mildly integrated AR(1) process quickly rises from 0 toward its long-run mean and subsequently fluctuates only mildly around it. In contrast, the mildly stationary INARCH process exhibits substantially greater variability and explores a much wider range of values.
\begin{figure}[H]
    \centering
    \includegraphics[scale = 0.2]{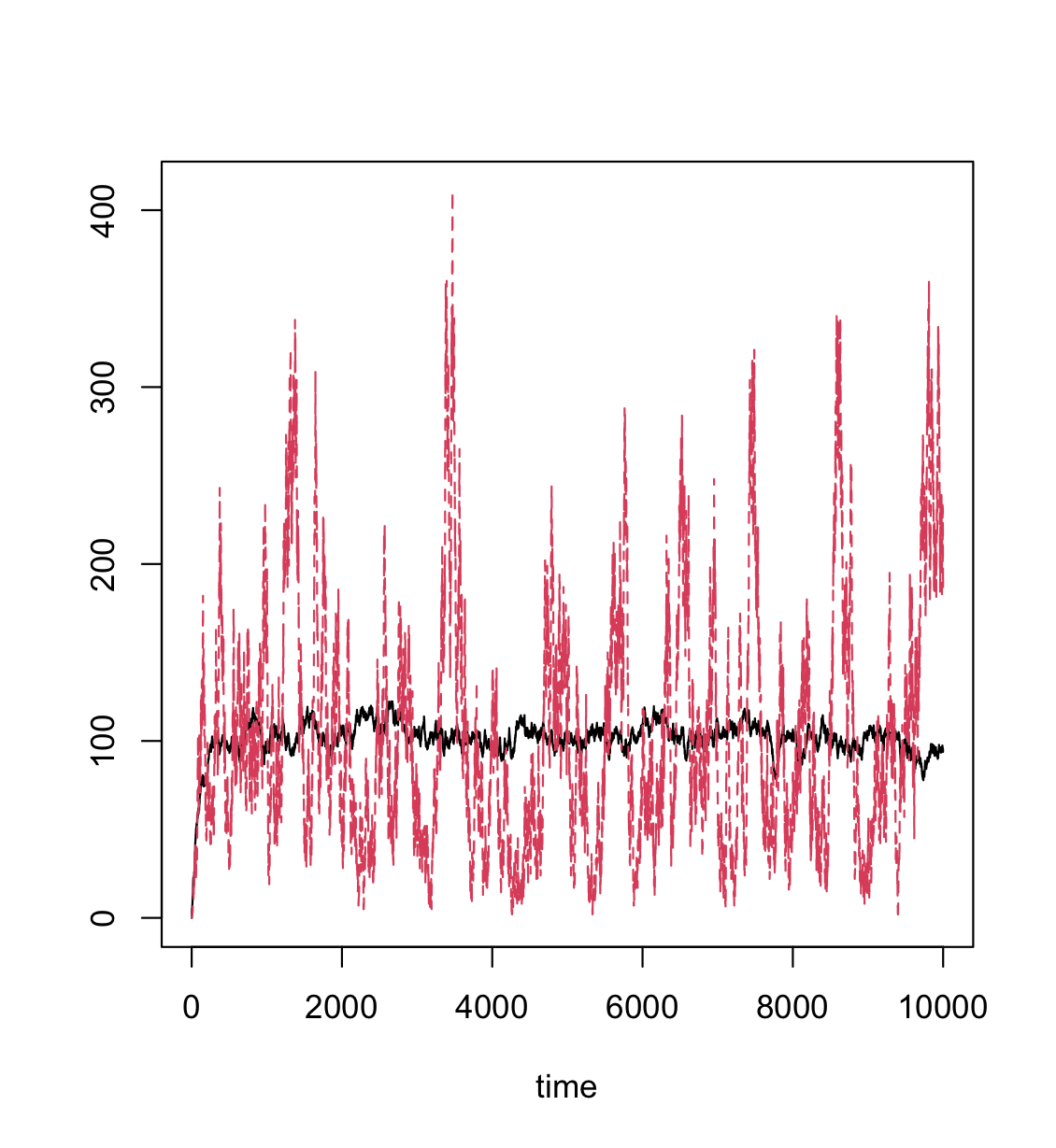}
    \caption{Sample paths: mildly stationary INARCH process (red dashed line) and mildly stationary AR(1) process with intercept (black solid line).}
    \label{fig:mild.int}
\end{figure}

%\begin{remark}One natural question for model \ref{feilm} is whether or not once the dominant term is removed, the error term behaves like a stationary process. For instance, in the mildly stationary case, whether $\frac{ X_{\lfloor k_n s \rfloor }}{k_n}-\mu$ could become stationary upon proper scaling. The answer to this question is negative. \end{remark}

%One can also get a second order refinement of this result. For instance, in the mildly stationary case, \cite[Proposition 4.1]{ispany2003asymptotic} show that $\frac{ X_{\lfloor k_n s \rfloor }}{k_n}-\mu$ is approximately a Gaussian distributed, with mean of order $e^{-ts}$, and variance of order $\frac{1}{k_n}$. Thus, $$X_{\lfloor k_n s \rfloor} \approx k_n (\mu+ O(\frac{1}{\sqrt{k_n}})), \qquad \forall s>0,$$  
The same distinction appears on the mildly explosive side. In the AR(1) model with intercept, the first-order explosive limit is deterministic, whereas in the affine model the corresponding limit is random:
%We now compare the mildly explosive affine process with the mildly explosive AR(1) process. For the latter, we have: %\footnote{Note that \cite[Theorem 2.6 (a)]{fei2018limit} contains a typographical error: the authors state that the weak limit is a constant, whereas their proof in the online appendix correctly shows that the limit is Gaussian.}  
%
\begin{prop}[\cite{fei2018limit}, Theorem 2.6 a)]
\label{feilm2}
    In model \eqref{feimodel}, and under the mildly explosive assumption $\alpha_n=1+\frac{1}{k_n}$, the rescaled process $\frac{ X_{\lfloor k_n s \rfloor }}{k_n e^{s}}$ converges weakly to the deterministic process $\mu(1-e^{-s})$ as $n$ increases to infinity. 
 \end{prop}
% CHECK THE ABOVE. 
%Thus, even though in model \eqref{feimodel} the mean process $\mathbb{E}[X_t]$ increases exponentially, the probability that $X_t$ takes negative values does not vanish as $t$ increases to infinity. %In other words, a mildly explosive AR(1) model with positive intercept has a positive probability of featuring a negative bubble. 
%This property is undesirable, if the process of interest is expected to remain nonnegative for some economic reasons such as no-arbitrage. This should be compared with the mildly explosive affine process, 
On the other hand, Propositions \ref{scalinglimit} and \ref{stationaritycir} lead to the following property:
     \begin{prop}
     In the mildly explosive affine process, $\frac{X_{\lfloor{ k_n s\rfloor}}}{k_n e^s}$ converges, as $n$ increases to infinity, to the process $\Upsilon_s/e^s$. 
    \end{prop}
    Note that this latter process converges, as $s$ increases to infinity, to the random variable $Z$ with mean $\mu$.
%Thus, in the mildly explosive case, the affine model ensures that the scaling limit is positive, in contrast to the AR(1) model  \eqref{feimodel}. %Thus, the affine process could be more suitable for applications in which the positivity of the variable $X_t$ is important. 

 %This comparison reinforces the same qualitative point: affine models retain state-dependent stochastic variation after the natural normalization, whereas the linear AR(1) model with intercept is dominated by its deterministic component.

To summarize, affine models retain state-dependent stochastic variation after the natural normalization, whereas the linear AR(1) model with intercept is essentially deterministic upon normalization. 

\subsection{Local bubble-like episodes under mild stationarity}

%\subsection{Spurious local explosion of a mildly stationary process}
%The preceding comparison suggests that the main advantage of the affine specification is not the mildly explosive case per se. Rather, it is that a mildly stationary affine process can display substantial local variation while remaining globally mean-reverting. This feature is useful for applications in which the observed series exhibits recurrent episodes of rapid growth without a permanently explosive data-generating mechanism.

\label{sec:spurious}

The mildly explosive AR(1) model is widely used in the empirical literature on (log) asset-price bubbles
\citep{phillips2011dating, phillips2015testing,
phillips2015testinglimit, yu2025inference, kejriwal2025improved}.
However, this modeling framework has limitations for the type of data that we are interested in. In particular, for count data, the log transform might not be well defined; for some positive processes such as interest rates, an explosion might also seem implausible. 

%First, without an intercept or positivity restriction, a linear explosive process may generate negative trajectories with positive probability, which is difficult to reconcile with nonnegative asset prices. Second, mildly explosive specifications naturally describe a single explosive episode; in applications with multiple boom-and-bust episodes, the literature therefore typically applies the model locally to selected subsamples.

%We argue that such local evidence of explosiveness can be misleading when the global DGP is a mildly stationary affine process. 
We argue that for nonnegative affine processes, bubble-like patterns can also arise in the mildly stationary case, that is, even without a locally explosive slope. Rather, when $X_{t-1}$ is large, because the conditional variance $\mathbb{V}[X_{t}|X_{t-1}]=\beta_n X_{t-1}+\delta_n$ is also large, the next value $X_t$ has a significant probability of exceeding $X_{t-1}$. At the same time, since $\alpha_n<1$ in the mildly stationary case, the level of $X_t$ will eventually decrease. This leads to bubble-like episodes with rise, peak and decline. This is exactly what we observe in Figure 1, in which we see roughly 10 bubbles. A simple simulation shows that  for the mildly stationary affine process simulated in Figure 1, if we split the full sample into 10 blocks of $1000=10k_n$ observations each, then the probability that at least one block has a block maximum above \(3\) times the marginal mean $\frac{\mu_n}{1-\alpha_n}$ is 99.8\%. Yet, for blocks whose block maximum exceeds this threshold, 99.6\% still have a local OLS estimate \(\hat\alpha_{\text{block}}\le 1\). For the mildly stationary linear AR(1) comparator in
Figure~\ref{fig:mild.int}, which has the same \(\alpha_n\) but constant
innovation variance, the probability of observing at least one block maximum
above the same threshold is below \(0.1\%\).

We are \emph{not} claiming that mildly stationary affine processes should replace mildly explosive linear AR(1) models in applications to asset-price bubbles. The point is instead that the two frameworks are suited to different empirical objects. Mildly explosive models are natural when sustained explosive growth is economically or physically plausible, as in some asset-price applications. By contrast, many nonnegative series considered in this paper, such as short-term interest rates and disaster frequencies, cannot plausibly follow a globally explosive path over long horizons. For such series, the mildly stationary affine model provides a way to generate bubble-like rise--peak--decline episodes while preserving global mean reversion and the relevant state-space constraint.

\section{Inference under mild integration}
\label{sec:5inferencemild}
\subsection{Large sample theory for the OLS estimator}
Let us now consider the OLS estimators of $\alpha_n$ and $\mu_n$ in the mildly integrated case.

To establish the analog of Theorem~\ref{thm1}, we need to strengthen Assumption~\ref{AssAffineMoment}. 
\begin{Ass}
\label{AssAffineMomentBis}
Assumption~\ref{AssAffineMoment} holds with some exponent $p>3$.
\end{Ass}
The stronger moment requirement is needed because, in the mildly stationary regime, the limit theory involves the diffusion on the unbounded time domain $[0,\infty)$ rather than on a compact interval, and uniform-in-time moment control becomes necessary.
%Intuitively, this stronger assumption is required for establishing large sample properties, because in the mildly stationary case, the scaling limit now involves the trajectory of the diffusion process on the whole domain $(0, \infty)$, rendering some uniformity conditions necessary. 

This leads to the following analog of Theorem~\ref{thm1}:

\begin{theo}
\label{cls}
Under Assumptions~\ref{ass1}, \ref{Assbeta}, \ref{assumptionmildly}, \ref{AssAffineMomentBis} and the initial condition $X_0 = o(k_n)$, the following holds:
\begin{enumerate}[$a)$]
 %   \item The rescaled process $(Y_s, s>0)=(\frac{X_{\lfloor{ k_n s\rfloor}}}{k_n},s>0)$ converges to the diffusion defined in \eqref{sde}, with parameter $\gamma=1$ if positive and $\gamma=-1$ if negative.
    \item If $\gamma<0$ (hence $\gamma=-1$), then the OLS estimator is consistent and asymptotically normal:
\begin{equation}
\begin{bmatrix}
\sqrt{nk_{n}}(\hat{\alpha}_{n}-\alpha _{n}) \\
\sqrt{n/k_n}(\hat{\mu}_n-\mu_n)
\end{bmatrix}
\overset{w}{%
\longrightarrow } \mathcal{N}(0, \Omega^{-1} \Sigma \Omega^{-1})
%\begin{bmatrix}
%\frac{2\gamma}{\mu} (\gamma U_1-\mu U_2)\\
%(1+2\mu) U_2-2 \gamma U_1
%\end{bmatrix},
\label{limitalpha}
\end{equation}
%where random variables $U_1, U_2$ follow jointly the Gaussian distribution with a zero mean vector and covariance matrix 
where matrices $\Omega$ and $\Sigma$ are given by:
\begin{equation}
\Omega=  \begin{bmatrix}
  \mathbb{E}[\Upsilon^2_{\infty}] & \mathbb{E}[\Upsilon_{\infty}] \\
  \mathbb{E}[\Upsilon_{\infty}]& 1
\end{bmatrix}, \qquad 
\Sigma=\sigma^2 \begin{bmatrix}
  \mathbb{E}[\Upsilon^3_{\infty}] & \mathbb{E}[\Upsilon^2_{\infty}] \\
  \mathbb{E}[\Upsilon^2_{\infty}]& \mathbb{E}[\Upsilon_{\infty}]
\end{bmatrix},
\end{equation}
where $\mathbb{E}[\Upsilon_{\infty}]$, $\mathbb{E}[\Upsilon^2_{\infty}]$ and $ \mathbb{E}[\Upsilon^3_{\infty}]$ are the first three moments of the invariant distribution of the diffusion process defined in \eqref{sde} with $\gamma=-1$. This invariant distribution is gamma with shape parameter $\frac{2\mu}{\sigma^2}$ and rate parameter $\frac{2}{\sigma^2}$. Its moments are therefore given by: 
$$\mathbb{E}[\Upsilon^1_{\infty}]={\mu}
%{\gamma}
,  \qquad  \mathbb{E}[\Upsilon^2_{\infty}]=\mu\left(\mu+\frac{\sigma^2}{2}\right)
%\frac{2\mu^2+\mu}{2\gamma^2}
, \qquad \mathbb{E}[\Upsilon^3_{\infty}]= \mu\left(\mu+\frac{\sigma^2}{2}\right)(\mu+\sigma^2)
%\frac{2\mu^3+3\mu^2+\mu}{2\gamma^3}
.$$
In particular, we have: $
 \sqrt{{n}k_n}(\hat{\alpha}_{n}-\alpha _{n})  \overset{w}{\longrightarrow} \mathcal{N}\left(0, 2\left(\frac{\sigma^2}{\mu}+1\right)\right).$

\item If $\gamma>0$ (hence $\gamma=1$), then $\hat{\alpha}_n$ admits the following non-Gaussian limit:
$$k_n \alpha_n^{n/2}  (\hat{\alpha}_n-\alpha_n) \overset{w}{\longrightarrow} \frac{2}{\sqrt{3}}\sigma Z^{-1/2} Y,$$
where $Y\sim\mathcal{N}(0,1)$, $Z$ is the limiting gamma variable in Proposition~\ref{stationaritycir}, and $Y$ and $Z$ are independent. By contrast, $\hat{\mu}_n$ is inconsistent, and $\mathbb{V}(\hat{\mu}_n-\mu_n)$  explodes as $n$ increases to infinity. 
\end{enumerate}
\end{theo}

Note that the distribution of $Z^{-1/2} Y$ belongs to the class of normal-inverse-Gaussian mixtures \citep{barndorff1997normal}.

\begin{proof}
    See Appendix~\ref{proofclsa} for Part $a)$ and Appendix~\ref{proofclsb} for Part $b)$.
\end{proof}
%Thus, we have the following table for the convergence rate of the parameters in the different cases:\begin{table}[H]    \centering    \begin{tabular}{|c|c|c|c|c|c|c|c}    \hline &   Stationary & mildly stationary & nearly stationary & nearly explosive     & mildly explosive  \\&     $\alpha_n=\alpha_0<1$  & $\alpha_n =1-\frac{1}{k_n}$ & $\alpha_n =1+\frac{\gamma}{n}, \gamma<0$ & $\alpha_n =1+\frac{\gamma}{n}, \gamma>0$     & $\alpha_n =1+\frac{1}{k_n}$   \\ \hline$\hat{\alpha}_n$ & $\sqrt{n}$  & $\sqrt{nk_n}$ & $n$ & $n$& $k_n e^{\frac{n}{2k_n}}$ \\ \hline$\hat{\mu}$ & $\sqrt{n}$  & $\sqrt{n/k_n}$  & $O(1)$ & $O(1)$ & NA \\ \hline\end{tabular}\caption{Convergence rate of $\alpha$ and $\mu$ under different settings.}    \label{tab:}\end{table}
%As we move from the left (closer to stationarity) to the right (closer to explosion), the term $X_{t-1}$ becomes more and more dominant in the conditioning mean $\mathbb{E}[X_{t} \mid X_{t-1}] =\alpha_n X_{t-1}+\mu$, and the convergence rate of $\hat{\alpha}_n$ (resp. $\hat{\mu}$) becomes quicker and quicker (resp. slower and slower).  

Theorem~\ref{cls}(a) provides one of the main motivations for the mildly stationary framework. In contrast to the local-to-unity case, both OLS estimators are consistent and jointly asymptotically normal.

As expected, the convergence rates $(\sqrt{nk_n}, \sqrt{n/k_n})$ for the two parameters bridge the $(\sqrt{n}, \sqrt{n})$ and $(n, 1)$ rates associated with the stationary and local-to-unity regimes, respectively. The former arises as a limiting case when $k_n$ diverges slowly, while the latter corresponds to the limiting case when $k_n$ grows at a rate close to that of $n$.

 \subsection{The block local-to-unity interpretation}
%\subsubsection{Mildly integrated as a scale and time changed local-to-unity model}

%The time and scale change argument also explains the convergence speed found in part $b)$ of Theorem 2. 
To relate the convergence rates obtained in Theorem~\ref{cls} $a)$ to those in Theorem~\ref{thm1}, consider partitioning the sample $X_t, t=1,\dots,n$ into non-overlapping $\left\lfloor \frac{n}{k_n} \right\rfloor$ blocks, each containing approximately $k_n$ observations. Each block then behaves as a local-to-unity model, but with different initial values.\footnote{For the first block, the initial value is 0, while for subsequent blocks, it is of order $k_n$.} Applying Theorem~\ref{thm1} to each block produces $\left\lfloor \frac{n}{k_n} \right\rfloor$ estimators of $(\alpha_n, \mu_n)$. Within each block, the estimator of $\alpha_n$ converges at rate $k_n$, while the estimator of $\mu_n$ is inconsistent. Assuming that the CLT can be applied — justified by the stationarity of the process for fixed $k_n$ — the sample average of these $\left\lfloor \frac{n}{k_n} \right\rfloor$ estimators converges at rate $k_n \times \sqrt{\left\lfloor \frac{n}{k_n} \right\rfloor} \approx \sqrt{n k_n}$ (respectively, $\sqrt{n/k_n}$) for the aggregated estimator of $\alpha_n$ (respectively, $\mu$). Part $a)$ of Theorem~\ref{cls} shows that the OLS estimator achieves the same convergence rates as this aggregated estimator. Moreover, although each block estimator has a non-Gaussian limiting distribution (by Theorem~\ref{thm1}), aggregation across blocks restores Gaussianity.

\subsection{Comparison with mildly integrated AR(1) model}
\paragraph{Comparison with mildly stationary AR(1) model.}
%It is instructive to apply the same block local-to-unity perspective, as proposed by \cite{phillips2010smoothing}, to the mildly stationary AR(1) with intercept.
%to both the mildly stationary AR(1) model with intercept. 
%Let us apply the same idea of section 3.2.2 to the mildly integrated AR(1) processes with intercept \eqref{feimodel}. 
%Let us first recall that in the exact unit root case $\alpha_n=1$, \cite[Page 497]{hamilton1994time} shows that the convergence rates of the OLS estimator $\hat{\alpha}_n$ and $\hat{\mu}$ are $n^{3/2}$ and $n^{1/2}$, respectively. 
%Recently, this result has also been extended by \cite[Theorem 3]{liu2023limit} to the case where $\alpha_n$ satisfies the local-to-unity Assumption 2.\footnote{See also \cite{ispany2003asymptoticproceeding, ispany2003asymptotic} for a similar result for local-to-unit INAR(1) processes.}
%Throughout the rest of this subsection, let us focus on model \eqref{feimodel} in the mildly stationary case, that is, under Assumption 3, with $\gamma=-1$. 
\cite{liu2023limit}\footnote{See also \cite[Theorem 4]{peng2024note} for a similar result for a mildly stationary INAR(1) model.} establish that:
\begin{prop}
In model \eqref{feimodel}, suppose that there exist constants $h_1, h_2 \in [0,1]$ such that 	
	\begin{equation}
		\label{assumption}
	\lim_{n \to \infty}	 \frac{\sqrt{n}}{\max(\sqrt{n}, k_n)} =h_1, \qquad  	\lim_{n \to \infty}	 \frac{k_n}{\max(\sqrt{n}, k_n)} =h_2.
		\end{equation}
	Then
	$ 
	\begin{bmatrix}
		\sqrt{k_n}\max(\sqrt{n},k_n) (\hat{\alpha}_n-\alpha_n) \\
	\frac{\max(\sqrt{n},k_n)}{\sqrt{k_n}} (\hat{\mu}_n-\mu_n)
	\end{bmatrix}
	$ converges weakly to a degenerate bivariate Gaussian vector. 
\end{prop}
%This result is striking for three reasons. First, eq. \eqref{assumption} is stronger than the classical assumption $n/k_n \rightarrow \infty$. For instance, the sequence $k_n=n^{\frac{1}{2}+\frac{1}{3}\cos(n)}$ satisfies $n/k_n \rightarrow \infty$, yet it does not satisfy \eqref{assumption}. In particular, for some sequences $(k_n)$, the limiting distribution of the OLS estimator may fail to exist. Second, even when the limit exists, it is degenerate, and is therefore of limited use for inference.
%and more importantly, the convergence rate of the OLS estimator of $\hat{\alpha}_n$ found in Theorem 5 is non-standard. % and is not compatible with the block local-to-unity interpretation. For instance, 
When $k_n=\sqrt{n}$, the proposition yields a rate $\sqrt{nk_n}$ for $\hat{\alpha}_n$, the same rate  that \citet{phillips2007limit} obtain for the model \emph{without} intercept. %as the OLS estimator of $\alpha_n$ in the model without intercept studied by \cite{phillips2007limit}. 
This stands in contrast to the local-to-unity AR(1) model, where introducing a nonzero intercept improves the convergence rate from $n$ \citep{phillips1987towards, yu2025inference} to ${n^{3/2}}$, see, e.g. \cite[p.~497]{hamilton1994} for the exact unit root case and \cite[Theorem 3]{liu2023limit} for the local-to-unity case. 
%To understand this "paradox", consider again the block interpretation. 
The block interpretation clarifies this apparent paradox. The OLS estimator computed from the first block of $\lfloor k_n \rfloor$ observations is a local-to-unity AR(1) estimator with intercept and is therefore $k_n^{3/2}$-consistent. When $k_n$ grows faster than $\sqrt{n}$, $k_n^{3/2}$ coincides with the full-sample rate $\sqrt{k_n}\max(\sqrt{n},k_n)$ of Proposition~7. That is, the first block alone is essentially as informative as the entire sample.
%In model \eqref{feimodel}, focusing on the first block of $\lfloor k_n \rfloor$ observations yields a local-to-unity model with intercept. Consequently, by \cite[p.~497]{hamilton1994}, the corresponding OLS estimator of $\alpha_n=1-\frac{1}{k_n}$ estimated on this subsample is $k_n^{3/2}$-consistent. If $k_n$ grows faster than $\sqrt{n}$, then the rate $k_n^{3/2}$ coincides with the rate $\sqrt{k_n}\max(\sqrt{n},k_n)$ given in Proposition 7 for the full sample OLS estimator. Thus, the estimator based on the first block has the same convergence speed as the estimator based on the entire sample. 
%same block local-to-unity interpretation as in section 3.2.2, then we would expect a convergence rate of $k_n^{3/2} \times \sqrt{n/k_n}= \sqrt{n}k_n$ for $\hat{\alpha}_n$ based on the full sample OLS. 
%How to interpret the fact that the subsequent blocks are less ``relevant"? This "paradox" can be explained by the fact that, in model \eqref{feimodel} under the mildly integrated Assumption 4, the rescaled process $X_t/k_n$, starting from 0, converges rapidly to $\mu$ (see section 4.2). This indicates that the behavior of the process in the first few blocks differs substantially from that in later blocks. 
Why are subsequent blocks effectively less informative? %Under Assumption~4, the rescaled process $X_t/k_n$ converges rapidly from $0$ to $\mu$ (see Proposition 4), so that the early blocks differ qualitatively from later ones in their behavior.
For the $i$th block, whose initial value is $X_{(i-1)k_n}$, we can rewrite the model as
$$
X_{t}-X_{(i-1) k_n}= \left(1-\frac{1}{k_n}\right)(X_{t-1}-X_{(i-1)k_n})+ \mu-\frac{X_{(i-1)k_n}}{k_n}+\epsilon_t, \qquad (i-1) k_n<t \leq i k_n. 
$$
Thus, the process $(X_{t}-X_{(i-1) k_n})$ is a local-to-unity AR(1) process with initial value 0 and a modified intercept $\mu_i=\mu-\frac{X_{(i-1)k_n}}{k_n}$. By Proposition~\ref{feilm1}, as the index of the block $i$ increases from $1$ to $n/k_n$, the modified intercept $\mu_i$ converges to 0. Consequently, for large $i$, the $i$th block behaves like a local-to-unity model without intercept, whose convergence rate is $k_n$ rather than $k_n^{3/2}$ \citep{hamilton1994}.  Later blocks therefore contribute disproportionately little to the full-sample estimator. %As a result, later blocks are effectively less informative for estimating $\alpha_n$, and contribute relatively little to the full-sample OLS estimator. %This explains why the overall convergence is not improved by including additional blocks.
 %is incompatible with the block local-to-unity interpretation, which predicts that 
%As a comparison, for the mildly stationary affine process, because the scaling limit process is asymptotically stationary, the different blocks have similar trajectories. 
By contrast, for the mildly stationary affine process, the scaling limit is asymptotically stationary, so successive blocks have qualitatively similar trajectories, and each contributes a comparable amount of information.

\paragraph{Comparison with mildly explosive AR(1) processes.}
%In particular,  \forall n$, with $\tau$ smaller than but close to 1 (hence is larger than $1/2$), then this alternative estimator based on fewer data converges at a higher rate than the OLS estimator based on all the data, since by Theorem 1, the latter converges at the rate $\sqrt{n k_n}=n^{\frac{1+\tau}{2}}< n^{\frac{3 \tau}{2}}$, since $\tau>\frac{1}{2}$. 
Part~$b)$ of Theorem~\ref{cls} shows that, in contrast to Theorem~\ref{thm1}, the asymptotic distribution of the OLS estimator depends on the sign of $\gamma$, %Under Assumptions~1-3, the limiting distribution depends on the behavior of the CIR process over the finite interval $(0,1)$, whereas in the mildly integrated setting it depends on its behavior over the expanding interval $(0,n/k_n)$, which diverges to $(0,\infty)$.
% CHECK BELOW? Inconsistency? 
%over $(0, \infty)$, whereas in Theorem \ref{thm1} it depends on its behavior over $(0,1)$. 
since the long-run behavior of the CIR process depends on this sign. The divergence of $\hat{\mu}_n$ for $\gamma=1$ can be read off the first-order condition $\hat{\mu}_n=\frac{1}{n} \sum_{t=1}^n (X_t- \hat{\alpha}_n X_{t-1})$. %Here, $\hat{\alpha}_n$ converges at rate $k_n e^{-\frac{n}{2k_n}}$, while $\sum_{t=1}^n X_{t-1}$ is of order $k^2_n e^{\frac{n}{k_n}}$. Consequently, the term $\frac{1}{n} \hat{\alpha}_n \sum_{t=1}^n   X_{t-1}$ diverges, 
Here, $\hat{\alpha}_n-\alpha_n$ is of order $\left(k_n\alpha_n^{n/2}\right)^{-1}$, while $\sum_{t=1}^n X_{t-1}$ is of order $k_n^2\alpha_n^n$. Consequently, the term $\frac{1}{n}(\hat{\alpha}_n-\alpha_n)\sum_{t=1}^n X_{t-1}$ is of order $\frac{k_n}{n}\alpha_n^{n/2}$, which in turn causes $\hat{\mu}_n$ to diverge. Finally, note that these convergence rates differ markedly from those obtained for the mildly explosive AR(1) model with intercept. %Indeed, \cite{fei2018limit} show that in their mildly stationary cases, $\alpha_n$ is $n\sqrt{k_n}$ consistent, whereas $\mu$ is $\sqrt{n}$ consistent; 
Indeed, \cite{fei2018limit} shows that, in the mildly explosive case, $\alpha_n$ is $\alpha_n^n k_n^{3/2}$-consistent. Its slightly faster convergence rate can be explained by the fact that, in the affine model, the conditional variance grows with the state and is therefore itself explosive. %Nevertheless, the rates we obtain are closely aligned with those reported in \cite{alaya2012parameter}, who study parameter estimation for the continuous-time CIR process. %Their approach, however, assumes that only one parameter is to be estimated at a time. 

\subsection{Estimation of the variance parameter}
In the mildly stationary case, we may assume without loss of generality that $\gamma=-1$. Moreover, $\mu$ can be consistently estimated by $\hat{\mu}_n$ since 
$$
\mu-\hat{\mu}_n=(\mu-\mu_n)+(\mu_n-\hat{\mu}_n),
$$
where the first term converges to zero by Assumption~\ref{Assbeta}, and the second converges to zero in distribution. By Slutsky's lemma, it follows that $\hat{\mu}_n \overset{p}{\rightarrow}\mu$. 

However, for inference purposes, the limiting distribution in Theorem~\ref{cls} also depends on the parameter $\sigma^2$.  %In both the local-to-unity and mildly integrated cases, the limiting diffusion process depends potentially on three parameters: $\gamma, \mu$ and $\sigma^2$. Let us now discuss, for , whether or not these parameters can be consistently estimated. 
Two cases arise. $(i)$ the parametric specification of the affine process pins down $\sigma^2$ a priori. This is the case for INARCH and NBAR models, for which $\sigma^2=1$ and $\sigma^2=2$, respectively.  
%\footnote{If one is ready to fix the value of $\mu$, then one can construct a confidence interval for $\hat{\alpha}_n$, in a similar way as the Dickey-Fuller test, with null hypothesis $H_0: \alpha_n=1$ against $H_1= \alpha_n<1$. }
%Since we can assume, without loss of generality, $\gamma=1$, and $\mu$ can be consistently estimated, all the parameters in the limiting distribution of Theorem 2 are estimable, and thus we can use Theorem 2 to construct a confidence interval for $\hat{\alpha_n}$, with a null hypothesis given by Assumption 3. This approach, however, comes with the downside that $\hat{\mu}$ converges only at the rate $\sqrt{n/k_n}$, which can be quite slow if $k_n$ is close to $n$. 
%However, since $\hat{\mu}$ converges only at the rate $\sqrt{n/k_n}$, 
$(ii)$ More generally, $\sigma^2$ is left unspecified — as in the ARG model — and must be estimated. Since the martingale difference sequence $W_t$ has conditional variance $\beta_n X_{t-1}+\delta_n$,
%$\beta_n X_{t-1}+\eta$ <- inconsistency? 
which is dominated by $\beta_n X_{t-1}$, a natural estimator of $\sigma^2=\lim_{n \to \infty} \beta_n $ is:
\begin{equation}
\label{sigmahat}
   \hat{\sigma}^2= \frac{\sum_{t=1}^n \widehat{W}_t^2}{ \sum_{t=1}^n X_{t-1}},
\end{equation}
where $\widehat{W}_t=X_t-\hat{\alpha}_n X_{t-1}-\hat{\mu}_n$ denotes the empirical residual. 

We then have the following result:
\begin{prop}
\label{convergencebeta}
As $n$ increases to infinity, the estimator $\hat{\sigma}^2$ converges in probability
to its true value $\sigma^2$. 
\end{prop}
Note that, by the same argument, this estimator is also consistent in the local-to-unity case. The convergence rate of the parameter $\sigma^2$ is more delicate and would require control of conditional moments of order higher than two. We leave it to future research. 

Also, Proposition 8 does not provide the convergence rate of $\hat{\sigma}^2$. Finding such a rate would require additional assumptions, in particular the convergence rate of $\beta_n \to \sigma^{2}$, as well as possibly higher-order conditional moment conditions. Since the rate of $\hat{\sigma}^2$ is not necessary for the asymptotic distribution of $(\widehat{\alpha}_n,\widehat{\mu}_n)$, such an analysis is omitted.

\subsection{Alternative scaling without $k_n$ in the mildly stationary case}
In practice, the functional form of $k_n$ is rarely known. However,  $\alpha_n=1-\frac{1}{k_n}$ gives $k_n=\frac{1}{1-\alpha_n}$. Can we replace $\alpha_n$ by its consistent estimate $\hat{\alpha}_n$ to obtain $\hat{k}_n=\frac{1}{1-\hat{\alpha}_n}$, and substitute it into Theorem~\ref{cls} to get a large sample theory that no longer depends on the unknown $k_n$? The following theorem confirms this intuition. 
\begin{theo}
\label{plugin}
Under Assumptions~\ref{ass1}, \ref{Assbeta}, \ref{assumptionmildly} with $\gamma<0$ and Assumption~\ref{AssAffineMomentBis}, as $n$ increases to infinity, 
        \begin{equation}
    \label{uniform2}
   \begin{bmatrix}
\sqrt{\frac{n}{1- \hat{\alpha}_n}}(\hat{\alpha}_{n}-\alpha _{n}) \\
\sqrt{n (1-\hat{\alpha}_n)}(\hat{\mu}_n-\mu_n)
\end{bmatrix}
\overset{w}{%
\longrightarrow } \mathcal{N}(0, \,  \Omega^{-1} \Sigma \Omega^{-1}).
    \end{equation}
    \end{theo}
    \begin{proof}
        See Appendix \ref{proofplugin}.
    \end{proof}

%Substituting into shows that $\sqrt{n/(1-\alpha_n)}(\hat\alpha_n - \alpha_n)$ and $\sqrt{n(1-\alpha_n)}(\hat\mu_n - \mu_n)$ converge jointly to $\mathcal{N}(0, \Omega^{-1}\Sigma\Omega^{-1})$. 
    
    This result shows that, under the mildly stationary regime ($\gamma <0$), we can obtain an asymptotic joint distribution for the true parameters $(\alpha_n, \mu_n)$, \textit{without specifying the functional form of $k_n$}. The normalizing factors in \eqref{uniform2} are observable, and the limiting covariance matrix on the right side depends only on $\mu$ and $\sigma^2$, which is consistently estimated by $\hat{\mu}_n$. %Indeed, aside from $\alpha_n, \mu_n$, all terms on the left-hand side of \eqref{uniform2} are observable, while the limiting covariance matrix on the right-hand side depends only on $\mu$, which can be consistently estimated by $\hat{\mu}_n$. 
    This avoids the nuisance parameter problem in $\gamma$ that hampers inference in the local-to-unity framework of Section~\ref{sec:3localtounity}.
   % This is a significant advantage compared to the local-to-unity framework of Section 3, which suffers from the nuisance parameter problem due to $\gamma$ under Assumption 3. 

   Note that Theorem~\ref{plugin} does not extend to the local-to-unity case. Indeed, this theorem is based on the fact that $\hat{k}_n/k_n$ converges to 1 as $n$ increases to infinity. This convergence holds in the mildly stationary case, but breaks down in the local-to-unity case, since in this case, $n(1-\alpha_n)=\gamma$ is constant by Assumption~\ref{assumptionnearlyunstable}, while $n(1-\hat{\alpha}_n)$ has a nondegenerate random limit by Theorem 3. Furthermore, even if $\hat{k}_n$ were consistent, Theorem~\ref{thm1} shows that the limiting distribution is non-Gaussian, so the normal approximation in Theorem~\ref{plugin} would still fail.

%and we state the result in the following corollary:
 
%\begin{equation}
%n^{1/2}k_{n}^{-1/2}(\hat{\mu}-\mu)\overset{w}{%\rightarrow } %2\mu U_2-(2\gamma-1)U_1
%(1+2\mu) U_2-2 \gamma U_1
%\label{limitmuhat}
%\end{equation}
%\end{corollary}

%In \cite{fei2018limit}, they consider the model $$X_{t}=\mu+\alpha_n X_{t-1}+u_t$$ where $u_t$ is i.i.d. with zero mean. They show that $\alpha_n^{-n} k_n^{-2}\sum y_{t-1}$ and $\alpha_n^{-2n}(\alpha_n^2-1) k_n^{-2}\sum y^2_{t-1}$ both converge to constants, and $\alpha_n^{-n} k_n^{-3/2}\sum y_{t-1}u_t$ converges to a normal distribution. In our case, we have $X_t=\alpha_n X_{t-1}+\mu+W_t$, where $(W_t)$ is a martingale difference equation, but is not i.i.d. Moreover, it features heteroscedasticity, and is not stationary. 

%Let us first show that:
%\begin{itemize}
%\item $k_n^{-2} \alpha_n^{-n} X_{t}$ converges to a non degenerate distribution
%\item $k_n^{-2} \alpha_n^{-n} X_{t}$
%\item $k_n^{-2} \alpha_n^{-n} X_{t}$
%\end{itemize}
%We follow the same proof as Theorem 2.6 of \cite{fei2018limit}. 
%\subsection{Comparison with linear AR(1) with intercept}

\begin{remark}[Toward uniform feasible inference for $\alpha_n$]
\label{rem:uniform-alpha}
It is also possible to write a studentized version of Theorem 5. 

More precisely, we define $Z_{t-1}=(X_{t-1},1)^\top$, and
\[
\hat M_n=\frac1n\sum_{t=1}^n Z_{t-1}Z_{t-1}^\top,\qquad
\hat S_n=\frac1n\sum_{t=1}^n \hat V_t\,Z_{t-1}Z_{t-1}^\top,
\]
where $\hat V_t=\hat W_t^2$, with $\hat W_t$ the OLS residual. A feasible standard error
for $\hat\alpha_n$ is then
\[
\widehat{\mathrm{Var}}(\hat\alpha_n)=\Big[\tfrac1n\,\hat M_n^{-1}\hat S_n\hat M_n^{-1}\Big]_{11},
\]
and we have:
$(\hat\alpha_n-\alpha_n)/\widehat{\mathrm{Var}}(\hat\alpha_n)^{1/2}\Rightarrow
\mathcal N(0,1)$. 

An advantage of such a result would be that in the fixed stationary regime $\alpha_n=\alpha<1$, one expects the studentized term $(\hat\alpha_n-\alpha_n)/\widehat{\mathrm{Var}}(\hat\alpha_n)^{1/2}$ to still have a standard normal limit, hence one might get a uniform inferential result in both the fixed stationary and the mildly stationary regimes, in a similar spirit to the uniform inference developed by \citet{giraitis2006uniform} for linear AR(1) models without intercept. We do not pursue the formal derivation of this result, however, since we are not interested in the fixed stationary case. Note also that the local-to-unity regime $n(1-\alpha_n)\to \gamma <\infty$ is excluded, since there the rescaled design
converges to the random functional $\int_0^1\Upsilon_s^2\,ds$ and
$\hat\alpha_n$ has the non-Gaussian, nuisance-dependent limit of
Theorem~3.
\end{remark}

\subsection{Bootstrap for mildly stationary processes} \label{sec:bootstrap}
%In section 3 we have derived the large sample properties of the mildly integrated process. 
 
For INAR(1) processes, the literature \citep{chen2024unified, peng2024note} has also adopted the random-weighting bootstrap of \cite{zhu2016bootstrapping}. Although the plug-in method proposed in Theorem~\ref{plugin} already provides a simple tool for inference, we show below that the same bootstrap method is also valid in our framework and can serve as an alternative to the plug-in approach. More precisely, we consider the following
%In small samples, the limiting distributions derived in the previous section may not provide sufficiently accurate approximations. We therefore consider the random weighting bootstrap of \cite{zhu2016bootstrapping}, which has also been applied to mildly stationary INAR(1) models by , leading to the following
bootstrapped OLS estimators:
\begin{equation}
\label{cls_bootstrap}
(\hat{\alpha}^b_n, \hat{\mu}_n^b)^{\prime}:=\arg\min_{\alpha, \mu} \sum_{t=1}^n \delta^b_t(X_t-\alpha X_{t-1}-\mu)^2 = 
\left(
\begin{bmatrix}
\sum_{t=1}^n \delta^b_t X_{t-1}^2 & \sum_{t=1}^n  \delta^b_t X_{t-1} \\
\sum_{t=1}^n \delta^b_t X_{t-1} & \sum_{t=1}^n \delta_t^b 
\end{bmatrix}\right)^{-1} 
\begin{bmatrix}
\sum_{t=1}^n  \delta_t^b X_{t-1} X_t \\
\sum_{t=1}^n \delta^b_t X_t
\end{bmatrix}.
\end{equation}
where $b=1,\dots,B$, indexes bootstrap replications and $(\delta^b_t)$, $b, t$ varying, is an i.i.d. sequence of positive random variables, independent of the process $(X_t)$, with unit mean and unit variance. We also assume that this sequence is not too heavy-tailed:
$$\mathbb{E}|\delta_t^b-1|^{2+\eta}<\infty$$ for some $\eta$ that is positive but close to zero. 

The following theorem establishes the validity of this bootstrap in the mildly stationary case. 
\begin{theo}
\label{bootstrap}
 Under Assumptions~\ref{ass1}, \ref{Assbeta}, \ref{assumptionmildly} and \ref{AssAffineMomentBis} with $\gamma=-1$, conditionally on the data and in probability, 
 \[
\begin{bmatrix}
\sqrt{nk_n}(\hat{\alpha}^b_n-\hat{\alpha}_n)\\
\sqrt{n/k_n}(\hat{\mu}^b_n-\hat{\mu}_n)
\end{bmatrix}
\overset{w}{\longrightarrow}
\mathcal N(0,\Omega^{-1}\Sigma\Omega^{-1}),
\]
where $\Omega$ and $\Sigma$ are the matrices in Theorem~\ref{cls}.
\end{theo}
\begin{proof}
    See Appendix~\ref{proofbootstrap}.
\end{proof}
As a comparison, in the local-to-unity case an analogous conditional CLT shows that, conditionally on the data, the rescaled and centered bootstrap estimator $
\begin{bmatrix}
n(\hat{\alpha}^b_n-\hat{\alpha}_n)\\
\hat{\mu}^b_n-\hat{\mu}_n
\end{bmatrix}$
converges weakly, in probability, to the conditional law of:
\begin{equation}
\label{cmt_bs}
 \begin{bmatrix}
         \int_0^1 \Upsilon^2_s \mathrm{d}s &     \int_0^1 \Upsilon_s \mathrm{d}s\\
         \int_0^1 \Upsilon_s \mathrm{d}s & 1
 \end{bmatrix}^{-1} \begin{bmatrix}
   \sigma  \int_0^1 \Upsilon_s^{3/2} \mathrm{d}\tilde{B}_s\\
   \sigma \int_0^1 \Upsilon_s^{1/2} \mathrm{d}\tilde{B}_s
 \end{bmatrix},
\end{equation}
where $(\tilde{B}_s)$ is a standard Brownian motion independent of $(B_s)$. Conditional on the trajectory of $(\Upsilon_s)_{s \in [0,1]}$, the distribution in \eqref{cmt_bs} is Gaussian, whereas the limit in Theorem~\ref{thm1} becomes deterministic. The reason is that in $\int_0^1 \Upsilon_s^{3/2} \mathrm{d}B_s$, the integrator $(B_s)$ and the integrand $\Upsilon_s^{3/2}$ are correlated, whereas in \eqref{cmt_bs} the bootstrap noise $(\tilde{B}_s)$ is independent of $\Upsilon_s^{3/2}$. The two limits therefore differ, so the random-weighting bootstrap is \emph{invalid} in the local-to-unity case. 
%Indeed, in the stochastic integral $\int_0^1 \Upsilon_s^{3/2} \mathrm{d}B_s$ appearing in Theorem~3, the Brownian motion $(B_s)$ and the process $\Upsilon_s^{3/2}$ are correlated, whereas in \eqref{cmt_bs}, $(\tilde{B}_s)$ is independent of the process $\Upsilon_s^{3/2}$. This discrepancy leads to a different limiting distribution, implying that the bootstrap is invalid in the local-to-unity case.
 By contrast, Theorem~\ref{bootstrap} shows that the random-weighting bootstrap is valid in the mildly stationary case. Consequently, bootstrap methods can be used to construct confidence intervals, even when the limiting distribution depends on additional parameters such as $\sigma$. 

Taken together, the results of Sections~\ref{sec:3localtounity}–\ref{sec:5inferencemild} show that the mildly stationary framework offers substantial advantages over the local-to-unity framework: consistent estimation of both parameters, joint asymptotic normality, feasible inference without knowledge of $k_n$, as well as bootstrap validity.
 %Overall, the results of Sections~3-5 indicate that the mildly integrated framework offers substantial advantages over the local-to-unity framework considered in Section 3. 

 \subsection{Alternative estimators for mildly stationary processes}
For many stationary time series models, maximum likelihood (ML) and weighted least squares (WLS) estimators are more efficient than OLS, but their properties are much less well understood in nonstationary settings. Let us now briefly discuss two alternative estimators. %the regression model we work on features conditional heteroscedasticity, and 
Because $\gamma$ remains a nuisance parameter in the local-to-unity framework, we restrict our attention to the mildly stationary model.

%It will be shown in this section that for local-to-unity or mildly integrated affine processes, the properties of these estimators are rather complicated, making them not as competitive as in the stationary case. 
\subsubsection{Weighted Least Squares (WLS)}
This estimator, first introduced by \cite{wei1990estimation} for exactly integrated affine count processes, solves the following:
\begin{equation}
\label{wls1}
(\hat{\hat{\alpha}}_n,\hat{\hat{\mu}}_n):=\arg \min_{(\alpha, \mu)} \sum_{t=1}^n \frac{(X_t-\alpha X_{t-1}-\mu)^2}{1+X_{t-1}}.
\end{equation}
It follows that
\begin{align*}
\begin{bmatrix}
    \hat{\hat{\alpha}}_n-\alpha_n\\
    \hat{\hat{\mu}}_n-\mu_n
\end{bmatrix}=
    \begin{bmatrix}
\sum_{t=1}^n  \frac{ X^2_{t-1}}{ X_{t-1}+1 } & \sum_{t=1}^n  \frac{ X_{t-1}}{ X_{t-1}+1}  \\
\sum_{t=1}^n  \frac{ X_{t-1}}{ X_{t-1}+1}   & \sum_{t=1}^n  \frac{1}{ X_{t-1}+1} 
\end{bmatrix}^{-1}
\begin{bmatrix}
    \sum_{t=1}^n  \frac{W_t X_{t-1}}{ X_{t-1}+1}\\
    \sum_{t=1}^n   \frac{W_t }{ X_{t-1}+1}
\end{bmatrix}:=H_n^{-1}v_n.
\end{align*}
 By analyzing the asymptotic behavior of matrix $H_n$ and vector $v_n$, we obtain the following result.
\begin{prop}
\label{transient}
In the mildly stationary model ($\gamma=-1$), if $\frac{2\mu}{\sigma^2}>1$, then
\begin{align*}
    \begin{bmatrix}
 \sqrt{nk_n}  ( \hat{\hat{\alpha}}_n-\alpha_n)\\
 \sqrt{n/k_n}(   \hat{\hat{\mu}}_n-\mu_n)
\end{bmatrix}
\overset{w}{\longrightarrow}\mathcal{N}\left(0, \sigma^2 \begin{bmatrix}
    \mathbb{E}[\Upsilon_\infty] & 1 \\
    1 &  \mathbb{E}[1/\Upsilon_\infty]
\end{bmatrix}^{-1}\right),
\end{align*}
where $\mathbb{E}[1/\Upsilon_\infty]=\left(\mu-\frac{\sigma^2}{2}\right)^{-1}$ is finite. In particular, $\sqrt{nk_n}  ( \hat{\hat{\alpha}}_n-\alpha_n) \overset{w}{\rightarrow} \mathcal{N}(0,2). $
\end{prop}

\begin{proof}
See Appendix~\ref{prooftransient}.
\end{proof}

Thus, the WLS estimators of both $\alpha_n$ and $\mu_n$ achieve the same convergence rates as their OLS counterparts. We now compare their asymptotic covariance matrices. By Theorem~\ref{cls}, the rescaled OLS estimator $ \sqrt{nk_n} (\hat{\alpha}_n-\alpha_n)$ has asymptotic variance $2\left(\frac{\sigma^2}{\mu}+1\right)$, while Proposition~\ref{transient} shows that the corresponding WLS estimator $\sqrt{nk_n} ( \hat{\hat{\alpha}}_n-\alpha_n)$ has asymptotic variance equal to $2$. Hence, WLS is asymptotically more efficient, with relative efficiency $\frac{\sigma^2}{\mu}+1$, which is between 1 and 3 under the assumption $\frac{2\mu}{\sigma^2}>1$. 

The main drawback of the WLS estimator is that Proposition~\ref{transient} requires $2\mu > \sigma^2$, and its properties remain an open question when $2\mu  \leq  \sigma^2$. This issue also exists in affine count processes with an exact unit root $(\alpha_n=1$, $\beta_n=\sigma^2)$: \cite{wei1990estimation, wei1991convergence} report that, when $2\mu  \leq \sigma^2$, the convergence rate of the WLS estimator \eqref{wls1} is unknown.

\subsubsection{(Quasi) Maximum Likelihood (QML)}
The QML approach consists in using a simple, but potentially misspecified conditional distribution to estimate the parameters. Let $\Theta$ denote the vector of all model parameters, and define the $\log$-quasi-likelihood function by 
$$
L(\Theta)=\sum_{t=1}^n  { l(X_{t} \mid X_{t-1},  {\Theta})},
$$
where $l(X_{t} \mid X_{t-1},  {\Theta})$ is the log of the conditional pmf/pdf. Expanding the score around the true value $\Theta_0$, the QML estimator $\tilde{\Theta}$ satisfies the linear system
%A Taylor expansion around the true parameter value $\Theta_0$ yields that the MLE $\tilde{\Theta}$ satisfies the linear system 
\begin{equation}
\label{I0}
\left(\frac{\partial^2 L}{\partial \Theta \partial \Theta'}\right) (\tilde{\Theta}-\Theta_0)=-\frac{\partial L}{\partial \Theta}.
\end{equation}
The asymptotic properties of the QML estimator, therefore, depend on the functional form of the conditional distribution.

Because we are mainly interested in $(\alpha_n, \mu_n, \sigma^2)$, the misspecified distribution should ideally not contain extra parameters. Let us now discuss the Gaussian and the Poisson QML. 

\paragraph{Gaussian QMLE. }The WLS in eq.\eqref{wls1} can be viewed as a simplified version of the Gaussian QMLE given by:
$$
(\hat{\hat{\alpha}}_n,\hat{\hat{\mu}}_n, \hat{\hat{\beta}}_n,\hat{\hat{\delta}}_n):=\arg \min_{(\alpha, \mu, \beta, \delta)} \sum_{t=1}^n \Big[ \frac{(X_t-\alpha X_{t-1}-\mu)^2}{\delta+ \beta X_{t-1}}+\ln (\delta+ \beta X_{t-1}) \Big].
$$
The downside of this raw Gaussian QMLE is that it has no closed form expression. This prompts \cite{wei1990estimation} to replace $\frac{(X_t-\alpha X_{t-1}-\mu)^2}{\delta+ \beta X_{t-1}}$ by $\frac{(X_t-\alpha X_{t-1}-\mu)^2}{\beta+ \beta X_{t-1}}$, the two terms being asymptotically equivalent since $X_{t-1}$ is of order $k_n$. Minimizing with respect to $\alpha$ and $\mu$ only leads to the WLS.

 %Unlike OLS and WLS, the ML estimator jointly estimates the full parameter vector of the affine process, rather than a subset such as.
%Compared with OLS and WLS, %the ML approach has two fundamental differences. First, while OLS and WLS only estimate 
%the maximum likelihood (ML) approach differs in that it jointly estimates the full parameter vector of the affine process, rather than a subset such as $(\alpha_n, \mu_n, \sigma^2)$.  %Thus the OLS/WLS and ML estimators are usually not directly comparable, with the asymptotic properties of the latter potentially much harder to derive and highly model dependent. 
 
%In particular, if the affine model has at least three parameters, then the ML estimator is expected to have significantly different properties from the OLS and WLS. 
%Furthermore, the conditional pmf/pdf of an affine process could be highly nonlinear. 
%The following two examples illustrate that even when tractable, the MLE does not necessarily offer a clear advantage over simpler estimators.
\paragraph{Poisson QMLE.}
In this case, %the parameter vector reduces to $(\alpha_n,\mu_n)$ (since $\sigma=1$), %Let us denote by $\tilde{\Theta}=(\tilde{\alpha}_n,\tilde{\mu}_n)$ the MLE, defined by the first-order condition:\begin{align*}0=   \sum_{t=1} \frac{\partial l(X_{t} \mid X_{t-1},  {\theta})}{\partial \Theta} \mid_{\tilde{\Theta}},\end{align*}
%The rest of the derivation depends on the parametric form of the conditional pmf/pdf. Let us look at some examples.
%and
the conditional $\log$-quasi-likelihood depends only on two parameters $\alpha_n$ and $\mu_n$, and is given by:
$$l(X_{t} \mid X_{t-1},  {\Theta})=-(\alpha_n X_{t-1}+\mu_n)+X_t \ln (\alpha_n X_{t-1}+\mu_n)-\ln (X_t!).$$ 
Then \eqref{I0} implies % \begin{align*}   \sum_{t=1}^n   X_{t-1}& = \sum_{t=1}^n  \frac{X_t X_{t-1}}{\tilde{\alpha}_n X_{t-1}+\tilde{\mu}},      \sum_{t=1}^n   1& = \sum_{t=1}^n  \frac{X_t }{\tilde{\alpha}_n X_{t-1}+\tilde{\mu}}.\end{align*}
%By a Taylor's expansion around the true value $(\alpha_n, \mu)$, one gets: \begin{align*}    \sum_{t=1}^n   X_{t-1}& = \sum_{t=1}^n  \frac{X_t X_{t-1}}{{\alpha}_n X_{t-1}+{\mu}}-\sum_{t=1}^n  \frac{X_t X^2_{t-1}}{({\alpha}_n X_{t-1}+{\mu})^2}(\tilde{\alpha}_n-\alpha_n)-\sum_{t=1}^n  \frac{X_t X_{t-1}}{({\alpha}_n X_{t-1}+{\mu})^2}(\tilde{\mu}_n-\mu), \\     \sum_{t=1}^n   1& = \sum_{t=1}^n  \frac{X_t }{{\alpha}_n X_{t-1}+{\mu}}-\sum_{t=1}^n  \frac{X_t X_{t-1}}{({\alpha}_n X_{t-1}+{\mu})^2}(\tilde{\alpha}_n-\alpha_n)-\sum_{t=1}^n  \frac{X_t }{({\alpha}_n X_{t-1}+{\mu})^2}(\tilde{\mu}_n-\mu).\end{align*}
%Or equivalently:
 \begin{align*}
\begin{bmatrix}
    \tilde{\alpha}_n-\alpha_n \\
    \tilde{\mu}_n-\mu_n
\end{bmatrix}
=
\begin{bmatrix}
\sum_{t=1}^n  \frac{X_t X^2_{t-1}}{({\alpha}_n X_{t-1}+{\mu_n})^2} & \sum_{t=1}^n  \frac{X_t X_{t-1}}{({\alpha}_n X_{t-1}+{\mu_n})^2} \\
\sum_{t=1}^n  \frac{X_t X_{t-1}}{({\alpha}_n X_{t-1}+{\mu_n})^2} & \sum_{t=1}^n  \frac{X_t }{({\alpha}_n X_{t-1}+{\mu_n})^2} 
\end{bmatrix}^{-1}
\begin{bmatrix}
    \sum_{t=1}^n  \frac{W_t X_{t-1}}{{\alpha}_n X_{t-1}+{\mu_n}}\\
    \sum_{t=1}^n   \frac{W_t }{\alpha_n X_{t-1} +\mu_n}
\end{bmatrix}
\end{align*}
Since $\frac{1}{{\alpha}_n X_{t-1}+{\mu_n}} \approx \frac{1}{1+X_{t-1}}$, it follows that the QMLE behaves similarly to the WLS estimator under the condition $2\mu>\sigma^2$. That is, in the mildly stationary case,
$
  \big[\sqrt{nk_n}  (  \tilde{\alpha}_n-\alpha_n),
   \sqrt{n/k_n}  ( \tilde{\mu}_n-\mu_n)\big]
'$ converges to the same weak limit as in Proposition~\ref{transient}. %, with the extra constraint that $\sigma^2=1$ due to the Poisson assumption. In other words, in the Poisson case, the WLS is asymptotically as efficient as the MLE.

%In other words, for the Poisson model, the usual efficiency gain of the MLE over competing estimators in a stationary framework becomes smaller and smaller, because for large $n,$ the conditional distribution resembles more and more the 

In the mildly stationary INARCH regime, the QMLE therefore provides no efficiency gain over the WLS estimator. This is consistent with Theorem~\ref{general}, which says that the conditional Poisson distribution of $X_t$ given $X_{t-1}$ is approximately Gaussian for large $X_{t-1}$. Since the WLS estimator \eqref{wls1} effectively acts as a Gaussian QMLE, both have the same asymptotic distribution.

\section{Simulations}
\label{sec:6simulations}
%For each of the simulation experiments below, we assume without loss of generality that \(\gamma = 1\). 
 
\subsection{Non-normality in the local-to-unity case} %5.1

We begin by considering two local-to-unity scenarios $\alpha_n = 1 \pm \frac{\gamma}{n}$, with $\gamma$ and $\mu_n$ both set to 1. %To maintain consistency with the mildly integrated notation, we write\[\alpha_n = 1 \pm \frac{\gamma}{k_n},\] and restrict our attention to the local-to-unity case by setting $k_n = n$. Under this choice, the scaling $\sqrt{n k_n}$ reduces to $n$, so that the normalization for $\hat{\alpha}_n$ coincides with the standard local-to-unity rate.

We use the INARCH model (with $\sigma^2 = 1$) to simulate $M = 50{,}000$ independent trajectories of $(X_t)$, each of length $n = 3{,}000$, and compute the OLS estimators of $\alpha_n$ and $\mu_n$.

Figure~\ref{fig:case_3_local} plots the empirical histogram of the scaled and centered estimators
 $\left(n(\hat{\alpha}_n - \alpha_n), \hat{\mu}_n - \mu_n\right)
$. For a given trajectory of $(X_t)$, we also plot the empirical histogram of $n(\hat{\alpha}^b_n - \hat{\alpha}_n), b=1,\dots,B=50{,}000$. We see that while the bootstrap distribution is close to normality, the distributions of $\hat{\alpha}_n$ and $\hat{\mu}_n$ are clearly skewed and non-Gaussian, in line with Theorem~\ref{thm1}.

\begin{figure}[H]
\centering
\includegraphics[scale=0.5]{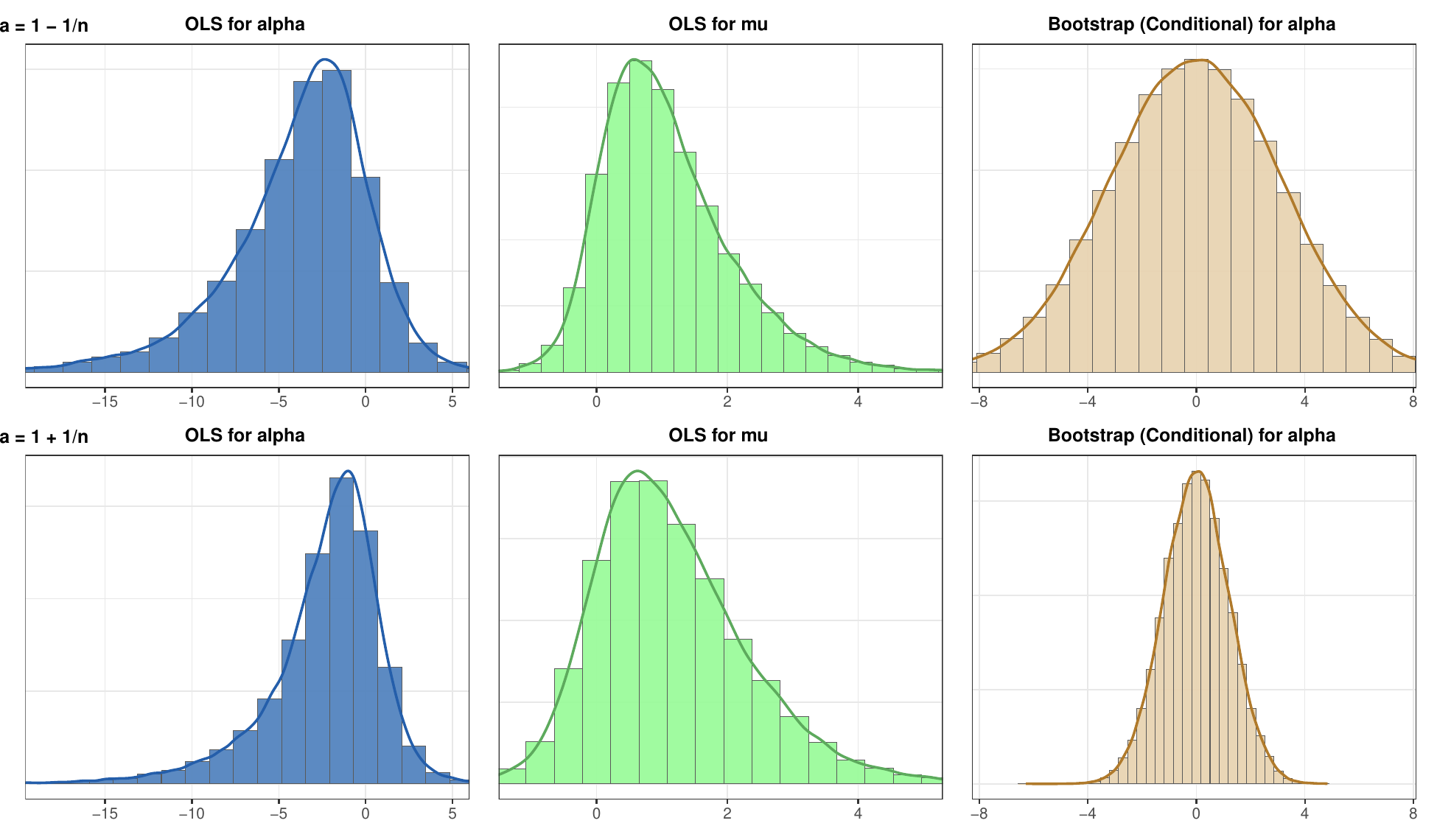}
\caption{Empirical distributions of $\left(n(\hat{\alpha}_n - \alpha_n)\right)$ (left panel) and of $(\hat{\mu}_n - \mu_n)$ (mid panel), as well as the bootstrap distribution of $n(\hat{\alpha}^b_n - \hat{\alpha}_n)$ for a given trajectory of $(X_t)$ (right panel). The upper panel corresponds to \(\alpha_n = 1 - \frac{1}{n}\), and the lower panel corresponds to \(\alpha_n = 1 + \frac{1}{n}\).}
\label{fig:case_3_local}
\end{figure}

We also report below the mean and variance of the sample, and compare them with their theoretical counterparts, obtained by tabulating the limiting distribution in Theorem~\ref{thm1} using $100{,}000$ simulated paths of the CIR process and $5{,}000$ time steps between 0 and 1. Table~\ref{tab:section61-sample-theory} shows that the sample and theoretical moments match very well, and as expected, $\hat{\mu}_n$ is inconsistent. Note that $n(\hat{\alpha}_n-\alpha_n)$ and $n(1-\hat{\alpha}_n)$ share the same sample variance  because they only differ by a constant, and this variance is very large (equal to 17.7 in the case where $\alpha_n=1-1/n$), confirming that $\gamma$ cannot be consistently estimated by $n(1-\hat{\alpha}_n)$. 
\begin{table}[H]
\centering
\small
\begin{tabular}{lrrrrrrrr}
\hline
Case
& \multicolumn{4}{c}{$n(\hat{\alpha}_n-\alpha_n)$}
& \multicolumn{4}{c}{$\hat{\mu}_n-\mu_n$} \\
\cline{2-5}\cline{6-9}
& \shortstack{Sample\\mean}
& \shortstack{Asymp. \\mean}
& \shortstack{Sample\\var.}
& \shortstack{Asymp.\\var.}
& \shortstack{Sample\\mean}
& \shortstack{Asymp.\\mean}
& \shortstack{Sample\\var.}
& \shortstack{Asymp.\\var.} \\
\hline
$\alpha_n = 1 - 1/n$
& -3.8536 & -3.8697 & 17.7075 & 18.0376
& 1.0720 & 1.0722 & 1.0037 & 1.0196 \\
$\alpha_n = 1 + 1/n$
& -2.3838 & -2.3600 & 11.4331 & 11.3631
& 1.1476 & 1.1441 & 1.4187 & 1.4136 \\
\hline
\end{tabular}
\caption{Sample and theoretical mean and variance with $n = 3{,}000$ and $M = 50{,}000$.}
\label{tab:section61-sample-theory}
\end{table}

\subsection{Normality in the mildly stationary case} %5.2
\label{sec:simulations62}
We now consider the mildly integrated setting, defined by 
\[
\alpha_n = 1 + \frac{\gamma}{k_n},\qquad  k_n = n^{\tau}, \qquad \tau \in (0,1), \qquad \gamma = -1, \qquad \mu_n=1.
\]
%First, the true parameter values \((\alpha_n, \mu_n)\) are computed. 
%Starting from \(X_0 = 0\), w
We generate \(M = 50{,}000\) independent trajectories of length \(n = 3{,}000\) from the INARCH model. Then we compute:

\begin{enumerate}
    \item the sample distribution of $\left(\sqrt{n k_n}(\hat{\alpha}_n - \alpha_n), \quad \sqrt{\frac{n}{k_n}}(\hat{\mu}_n - \mu_n)\right)$, which according to Theorem~\ref{cls} is asymptotically normal. 
    \item the plug-in counterpart $\left(\sqrt{\frac{n}{1- \hat{\alpha}_n}}(\hat{\alpha}_{n}-\alpha _{n}),
\sqrt{n (1- \hat{\alpha}_n)}(\hat{\mu}_n-\mu_n)\right)$ given in Theorem~\ref{plugin}, in which $k_n$ is replaced by $\frac{1}{1-\hat{\alpha}_n}$. We note that this rescaled estimator can only be computed if $\hat{\alpha}_n < 1$. Under the parameter setting described above, we find that around 1 percent of the estimates of $\hat{\alpha}_n$ are greater than 1. 
    \item the bootstrap distribution of $\left(\sqrt{n k_n}(\hat{\alpha}^b_n - \hat{\alpha}_n), \quad \sqrt{\frac{n}{k_n}}(\hat{\mu}^b_n - \hat{\mu}_n)\right)$ 
    obtained with $B=50{,}000$ and an arbitrarily chosen trajectory of $(X_t)$, whose OLS estimator is denoted by $(\hat{\alpha}_n, \hat{\mu}_n)$. The random weights \(\left(\delta_t^b\right)\) are drawn from an \(\mathrm{Exp}(1)\) distribution. %Note that the bootstrap estimator is centered around $(\hat{\alpha}_n, \hat{\mu}_n)$, which is itself random. 
\end{enumerate}
%Note that this first distribution is given for illustration purpose only, and cannot be used in practice, since it requires the knowledge of $k_n$.

Figure~\ref{fig:mild.int.comp}  plots these histograms and kernel density estimates for $\tau=0.4$ (upper panel) and $\tau=0.8$ (lower panel), respectively. For $\tau=0.4$, the six sample distributions are closer to normality, and within the same row, the shapes of the distributions are similar. For $\tau=0.8$ (which is quite close to the local-to-unity case), we see that the sample distributions are still quite far away from normality, except for the bootstrap distribution. 

% requires \usepackage{booktabs}

%The empirical distribution of the scaled estimators is displayed in the left panel of Figure~\ref{fig:mild.int.comp}.

%For the bootstrap analysis, a single trajectory of length \(n = 3\,000\) is selected. Based on this trajectory, \(B = 50{,}000\) bootstrap samples are generated using \eqref{cls_bootstrap}. The bootstrap relies on weights \(\left(\delta_t^b\right)\) drawn independently from an \(\mathrm{Exp}(1)\) distribution. The resulting bootstrap estimators \((\hat{\alpha}_n^b, \hat{\mu}_n^b)\) are used to construct the bootstrap distribution, shown in the right panel of Figure~\ref{fig:mild.int.comp}.

\begin{figure}[H]
\centering
\subfloat[$\tau=0.4$]{\includegraphics[width=\textwidth,height=0.42\textheight,keepaspectratio]{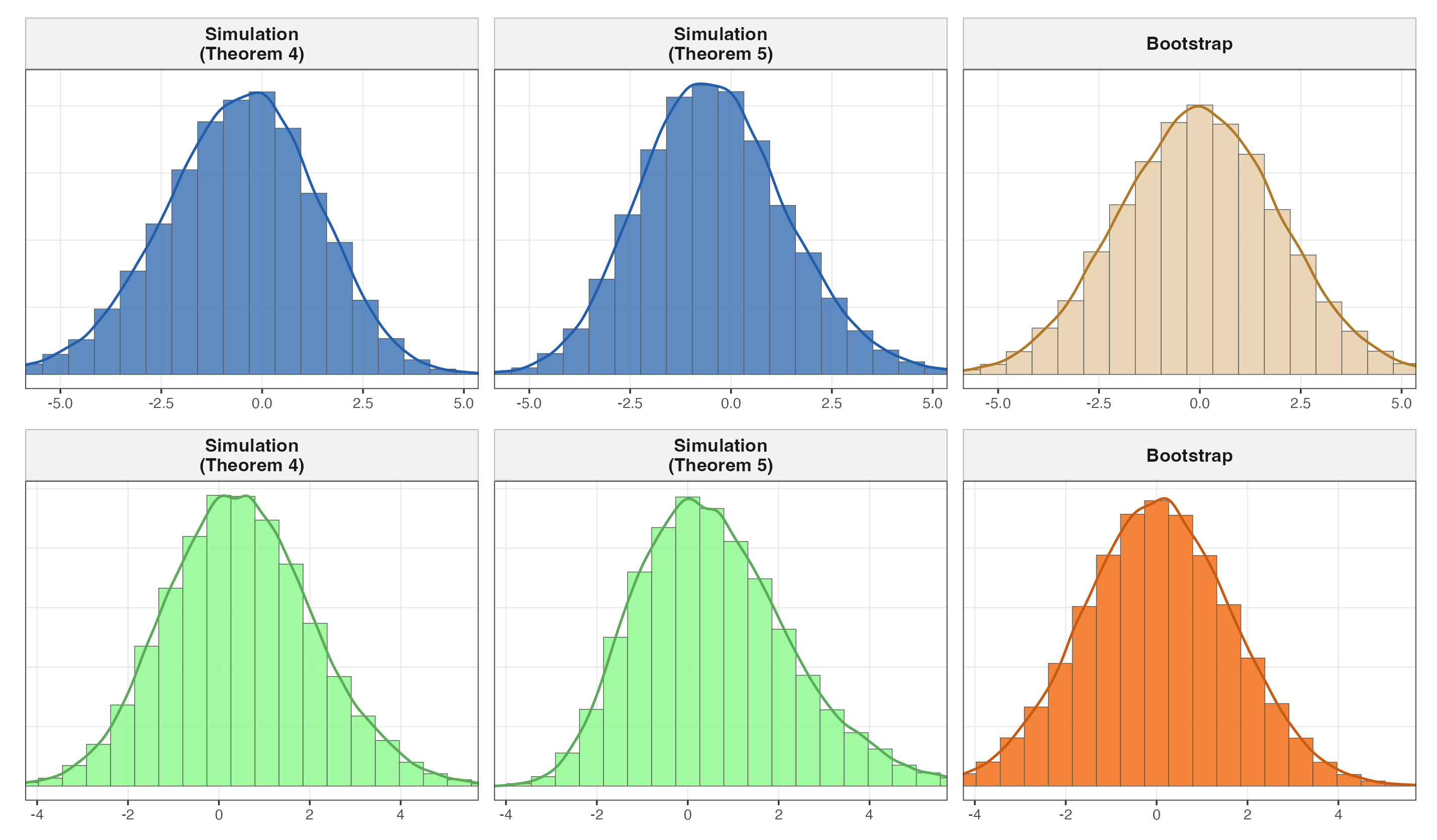}}\\[0.5em]
\subfloat[$\tau=0.8$]{\includegraphics[width=\textwidth,height=0.42\textheight,keepaspectratio]{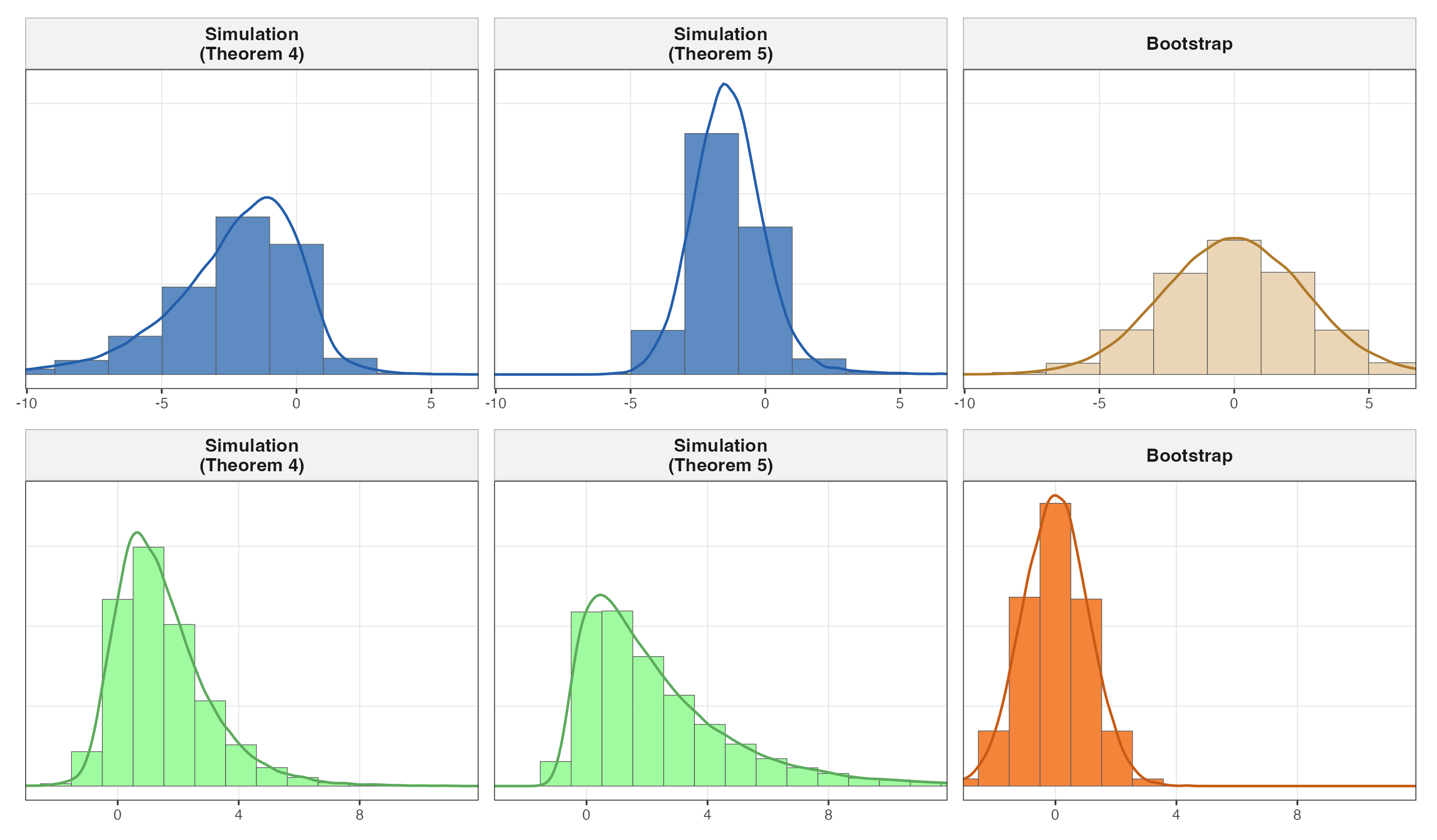}}
\caption{Sample distributions of the rescaled OLS estimators in the mildly stationary case for $\tau=0.4$ (panel (i)) and $\tau=0.8$ (panel (ii)). In each panel: upper row is $\sqrt{n k_n}(\hat{\alpha}_n - \alpha_n)$, lower row is $\sqrt{n/k_n}(\hat{\mu}_n - \mu_n)$; the left column shows the sample distribution from $M=50{,}000$ trajectories, the middle column the plug-in counterpart (Theorem~\ref{plugin}), and the right column the bootstrap counterpart (Theorem~\ref{bootstrap}). All simulations use $n=3{,}000$.}
\label{fig:mild.int.comp}
\end{figure}

We also report the sample mean and covariance matrix of these distributions in Table~\ref{tab:section62-tau04-tau08-booktabs}.

 \begin{table}[H]
\centering
\small
\setlength{\tabcolsep}{5pt}
\renewcommand{\arraystretch}{1.15}
\begin{tabular}{c l c c}
\toprule
$\tau$ & Distribution & Mean vector $(\alpha,\mu)$ & Covariance matrix \\
\midrule
0.4
& \shortstack{Benchmark\\(Theorem~\ref{cls})}
& $\begin{pmatrix}
-0.6376\\
\phantom{-}0.5358
\end{pmatrix}$
& $\begin{pmatrix}
\phantom{-}3.5847 & -2.5005\\
-2.5005 & \phantom{-}2.5829
\end{pmatrix}$ \\
0.4
& \shortstack{Plug-in\\(Theorem~\ref{plugin})}
& $\begin{pmatrix}
-0.4548\\
\phantom{-}0.6752
\end{pmatrix}$
& $\begin{pmatrix}
\phantom{-}3.2947 & -2.5351\\
-2.5351 & \phantom{-}2.9409
\end{pmatrix}$ \\
0.4
& \shortstack{Bootstrap\\(Theorem 6)}
& $\begin{pmatrix}
-0.0104\\
\phantom{-}0.0070
\end{pmatrix}$
& $\begin{pmatrix}
\phantom{-}3.1076 & -1.9999\\
-1.9999 & \phantom{-}1.9384
\end{pmatrix}$ \\
\addlinespace
0.8
& \shortstack{Benchmark\\(Theorem~\ref{cls})}
& $\begin{pmatrix}
-2.3781\\
\phantom{-}1.4943
\end{pmatrix}$
& $\begin{pmatrix}
\phantom{-}6.2898 & -2.1388\\
-2.1388 & \phantom{-}2.3414
\end{pmatrix}$ \\
0.8
& \shortstack{Plug-in\\(Theorem~\ref{plugin})}
& $\begin{pmatrix}
-1.3482\\
\phantom{-}2.4723
\end{pmatrix}$
& $\begin{pmatrix}
\phantom{-}2.2620 & -2.3732\\
-2.3732 & \phantom{-}7.7503
\end{pmatrix}$ \\
0.8
& \shortstack{Bootstrap\\(Theorem 6)}
& $\begin{pmatrix}
-0.0062\\
\phantom{-}0.0060
\end{pmatrix}$
& $\begin{pmatrix}
\phantom{-}3.5665 & -2.5221\\
-2.5221 & \phantom{-}2.4345
\end{pmatrix}$ \\
\bottomrule
\end{tabular}
\caption{Sample mean and covariance matrix for the sample distributions in Figure~\ref{fig:mild.int.comp}. As a comparison, the theoretical value of the covariance matrix is $\begin{pmatrix}
\phantom{-}4 & -3\\
-3 & \phantom{-}3
\end{pmatrix}$.}
\label{tab:section62-tau04-tau08-booktabs}
\end{table}
We see that the OLS estimator has a finite sample bias that is negative for $\alpha_n$, and positive for $\mu_n$. The negative finite sample bias for $\alpha_n$ is consistent with similar results found for linear AR(1) models; see, for example, \cite{marriott1954bias} for the stationary case and \cite{stoykov2019least} for the mildly stationary case, and the positive bias for $\mu_n$ is simply a consequence of the normal equation of the OLS estimator. Furthermore, Table~\ref{tab:section62-tau04-tau08-booktabs} shows that as $\tau$ increases from $\tau=0.4$ to $\tau=0.8$, the magnitude of the biases increases, and the sample covariance matrices deviate from their common theoretical values. Note that the bootstrap mean vector is smaller in magnitude, but this simply reflects the fact that the bootstrap estimator is centered around $(\hat{\alpha}_n, \hat{\mu}_n)$, with the latter being random and biased in finite samples. Note, also, that with $\tau=0.8$, $\alpha_n$ is very close to 1, and both the plug-in and the bootstrap methods estimate quite poorly the asymptotic variance of $\hat{\mu}_n$ (last two rows of Table 3), even for a sample size of $n=3,000$. 

\subsection{Empirical Coverage}
\label{sec:coverage}
Let us focus on the mildly stationary model and evaluate the empirical coverage of nominal 90\% confidence intervals for \((\alpha_n,\mu_n)\). We use the simulated trajectories of Section~\ref{sec:simulations62}, and compute the empirical coverage for the parameter $\alpha_n$, defined as:
\[
\text{Coverage}
= \frac{1}{N} \sum_{i=1}^{N} \mathbbm{1}_{\bigl\{\alpha_{\text{true}} \in \text{CI}^{(i)}\bigr\}},
\]
and analogously for \(\mu_n\). Similarly to Section~\ref{sec:simulations62}, we use two methods to compute confidence intervals: 
\begin{enumerate}
 %   \item The Gaussian approximation of Theorem 4. Note that this is for illustration purpose only, and cannot be used to real data in practice, since it requires the knowledge of $k_n$.
    \item The Gaussian approximation of Theorem~\ref{plugin}, in which $k_n$ is replaced by $\frac{1}{1-\hat{\alpha}_n}$. In the following, we call this the plug-in approach. Here, the confidence interval is built using the Gaussian distribution assumption, whose variance parameter depends, among others, on $\sigma^2$, which %Note that the Gaussian approximation of Theorem~\ref{cls} cannot be applied directly to real data due to the unknown value of $k_n$. 
   is either fixed as in the INARCH specification, or estimated using Proposition~\ref{convergencebeta}. Empirically, we find that both approaches lead to almost identical coverage rates and thus we will only report results with fixed $\sigma^2=1$.   %Moreover, it can only be applied if, for the data at hand, $\hat{\alpha}_n$ is effectively smaller than 1. 
    \item The bootstrap estimator in Theorem~\ref{bootstrap}, for a given trajectory of $(X_t)$. In this case, the confidence interval can be computed using either the empirical distribution of the bootstrapped samples, or the Gaussian approximation. Since we have shown in Section~\ref{sec:simulations62} that the bootstrap distribution is very close to (conditional) normality, the empirical coverage we obtain is almost identical across these two methods. Thus, we retain the normal approximation for simplicity. In particular, for the bootstrap method, we do not use the plug-in estimator $\hat{k}_n$ of $k_n$. 
%  Note that alternatively, we could use the sample bootstrap distribution of $(\hat{\alpha}^b_n, \hat{\mu}^b_n)$, $b=1,\dots,B$ to construct the confidence intervals.   
\end{enumerate}
%For the first method, the parameter $\sigma^2$ should be estimated (or fixed, if an assumption such as the INARCH model is made), since the covariance matrix of the limiting distribution depends on $\sigma^2$. On the other hand, for the bootstrap method, the covariance matrix can be directly estimated using $B$ bootstrap samples, without estimating $\sigma^2$. 
Since the mildly stationary regime imposes $\alpha_n<1$, the upper end of each confidence interval for $\alpha_n$ is capped at unity, analogous to the lower-bound truncation at 1.001 used by \citet{yu2025inference} in the mildly explosive setting. Table~\ref{fig:coverage.summary} reports the empirical coverage. Note, however, that for this simulation exercise, capping does not change the coverage rate, since the true DGP always corresponds to the case where $\alpha_n<1$. 

\begin{table}[ht]
\centering
\small
\begin{tabular}{ccccccc}
\hline
$n$ & $\tau$ &  $\alpha_n$ &\shortstack{$\alpha_n$\\Plug-in} & \shortstack{$\alpha_n$\\Bootstrap} & \shortstack{$\mu_n$\\Plug-in} & \shortstack{$\mu_n$\\Bootstrap} \\
\hline
%200 & 0.4 & 0.9200 & 0.8475 & 0.9350 & 0.8850 \\
%200 & 0.8 & 0.8475 & 0.7400 & 0.9400 & 0.7600 \\
%500 & 0.4 & 0.9175 & 0.8900 & 0.9575 & 0.9175 \\
%500 & 0.8 & 0.8450 & 0.7675 & 0.9400 & 0.7575 \\
75   & 0.4 & 0.822 & 0.943 & 0.818 & 0.955 & 0.880 \\
75   & 0.8 & 0.968 & 0.863 & 0.723 & 0.968 & 0.745 \\
\midrule
200  & 0.4 & 0.880 & 0.963 & 0.875 & 0.945 & 0.880 \\
200  & 0.8 & 0.986 & 0.875 & 0.790 & 0.960 & 0.772 \\
\midrule
2000 & 0.4 & 0.952 & 0.927 & 0.910 & 0.940 & 0.907 \\
2000 & 0.8 & 0.998 & 0.833 & 0.790 & 0.925 & 0.770 \\
\hline
\end{tabular}
\caption{Empirical coverage probabilities of nominal $90\%$ confidence intervals in Section~\ref{sec:coverage}, based on $N=400$ Monte Carlo replications and $B=50{,}000$ bootstrap draws. The asymptotic intervals use Theorem~\ref{plugin}, and the bootstrap-normal intervals replace empirical bootstrap quantiles by the Gaussian approximation $\hat{\theta}\pm z_{0.95}\widehat{\mathrm{sd}}_{\mathrm{boot}}$.}
\label{fig:coverage.summary}
\end{table}
We have explained in Section~6.2 that in general, $\hat{\alpha}_n$ (resp. $\hat{\mu}_n$) has a negative (resp. positive) finite sample bias. This leads to almost systematic undercoverage of both $\alpha_n$ and $\mu_n$ by the bootstrap method. For the confidence intervals computed using Theorem 5, the situation is more complicated. Indeed, the negative bias on $\alpha_n$ also leads to a bias on $k_n=\frac{1}{1-\alpha_n}$. Thus, compared to the true confidence interval, the plug-in confidence interval is on average shifted downwards, but also widened. These two biases imperfectly offset. For instance, for parameter $\alpha_n$, Theorem 5 overcovers when $\tau=0.4$, but undercovers when $\tau=0.8. $ Overall, the coverage rate is closer to the nominal rate for Theorem 5 than for Theorem 6, because of this partial compensation of the two biases.

%Generally, the empirical coverage level obtained using the bootstrap is undersized, especially for parameter $\alpha_n$. This reflects the fact that in finite-sample, the bootstrap sample distribution is centered around $(\hat{\alpha}_n, \hat{\mu}_n)$, which itself has a negative bias. On the other hand, the coverage level obtained using the normal approximation in Theorem~\ref{plugin} is closer to the nominal level of 90\%. This is explained by the fact that although the OLS is negatively biased in finite-sample, the estimation of $k_n$ by $\hat{k}_n=\frac{1}{1-\hat{\alpha}_n}$ is also negatively biased. These two biases partially cancel out in the resulting confidence interval. Note that in Table~\ref{fig:coverage.summary}, we have assumed $\sigma^2$ to be fixed to 1. If we instead estimate $\sigma^2$, the coverage we obtain is almost unchanged.  

\subsection{Testing for mild stationarity}
\label{sec:test}
The confidence interval built in Section~\ref{sec:coverage} can be used as a preliminary visualization tool for the econometrician to choose between local-to-unity (including exact unit root) and mild stationarity, by simply checking whether 1 belongs to the confidence interval. Nevertheless, because mild stationarity does not contain the exact unit root as a special case, this visual diagnosis should not be regarded as a proper unit root test. Instead, under mild stationarity, one can test:
\[
\left\{
\begin{array}{l}
    H_0: \alpha = \alpha_0, \\
    H_1: \alpha \neq \alpha_0
\end{array}
\right.
\]
where \(\alpha_0\) is strictly smaller than 1 due to the mildly stationary Assumption \ref{assumptionmildly}.  %The significance level is set at 10\%.

%The test is based on the randomly weighted bootstrap procedure described in Section~\ref{sec:bootstrap}. Given a sample \(\{X_t\}_{t=1}^n\), let \(\hat{\alpha}_n\) denote the OLS estimator, and let \(\{\hat{\alpha}_n^{b}\}_{b=1}^B\) denote the corresponding bootstrap estimators obtained using i.i.d. weights \(\{\delta_t^b\} \sim \mathrm{Exp}(1)\). We estimate the standard deviation of \(\hat{\alpha}_n\) using the bootstrap estimator:\[\hat{\sigma}_n= \left\{ \frac{1}{B} \sum_{b=1}^B (\hat{\alpha}_n^{b} - \hat{\alpha}_n)^2 \right\}^{1/2},\]
%which corresponds to the standard deviation of the bootstrap distribution of \(\hat{\alpha}_n^{b} - \hat{\alpha}_n\).

For each candidate value \(\alpha_0\), we define the test statistic
\[
t_n(\alpha_0) = \frac{\hat{\alpha}_n - \alpha_0}{{SD(\hat{\alpha}_n)}},
\]
where the denominator ${SD(\hat{\alpha}_n)}$ is an estimator of the standard deviation of $\hat{\alpha}_n$ and can be computed using two different approaches: the plug-in method of Theorem 5, or the bootstrap procedure of Theorem 6. 
%is the standard deviation of $\hat{\alpha}_n^{b}, b$ varying. Following \cite{peng2024note}, we approximate the distribution of \(t_n(\alpha)\) by the standard normal distribution (justified by the bootstrap result in Theorem~\ref{bootstrap}). 
The decision rule at the 10\% significance level is to reject the null hypothesis if and only if
\[
2\bigl(1 - \Phi(|t_n(\alpha_0)|)\bigr) < 0.10,
\] 
where \(\Phi\) denotes the standard normal CDF.

We note that the probability of rejection is a deterministic function of the true value $\alpha_n$ of the DGP, as well as the choice $\alpha_0$. It is equal to 
%Equivalently, this procedure defines a confidence set\[\mathcal{C}_{0.10} = \{\alpha : p(\alpha) \geq 0.10\},\]which can be interpreted as an approximate 90\% confidence interval for \(\alpha\).
%Rejection of the null hypothesis indicates that the persistence parameter differs from \(\alpha_0\), providing evidence toward either stronger mean reversion (that is, \(\alpha < \alpha_0\)) or weaker mean reversion / explosive behavior (that is, \(\alpha > \alpha_0\)). In particular, values of \(\alpha\) close to one correspond to local-to-unity behavior, while values exceeding one indicate mildly explosive dynamics.
the size of the test (when $\alpha_0=\alpha_n$), and to the power of the test (when $\alpha_n \neq \alpha_0)$. To compute these probabilities, we carry out the following simulation experiment.  
%\subsection{Size and power}
%The estimation of size and power follows a modified version of the procedure proposed in \cite{peng2024note}. Let:
%\begin{itemize}
 %   \item We choose one of the three sample sizes \(n_j \in \{75, 200,2000\}\), \(j=1,2,3\);
  %  \item \(\alpha_0 \in \{0.95,0.99\}\) denotes the values of the autoregressive parameter under the null hypothesis;
  %  \item \(\alpha^{(i)} \in \{0.80 + 0.004i : i = 0,\dots,50\}\) denotes the true values of \(\alpha\) used in the data generating process (DGP);
  %  \item \(\hat{\alpha}_{n_j,r}^{(i)}\) denotes the OLS estimator computed from the \(r\)th simulated series generated under \(\alpha^{(i)}\);
   % \item \(N = 400\) denotes the number of Monte Carlo replications;
    %\item \(r = 1,\dots,N\) indexes the simulation replications;
    %\item \(B = 400\) denotes the number of bootstrap replications.
%\end{itemize}
We vary the sample size $n \in \{75, 200, 2000\}$ and the true autoregressive parameter $\alpha$ over an equally spaced grid of 51 points from $0.80$ to $1.00$ (step $0.004$). For each $(n, \alpha)$ cell we generate $N = 400$ independent trajectories, compute the OLS estimator on each, and apply the plug-in and the weighted bootstrap with $B = 5000$ replications to test $H_0\colon \alpha = \alpha_0$ at the two values $\alpha_0 \in \{0.95, 0.99\}$.
 
Figure 4 below reports the size, power, as well as size-corrected powers of the test.

 \begin{figure}[H]
    \centering
    \includegraphics[width=\textwidth]{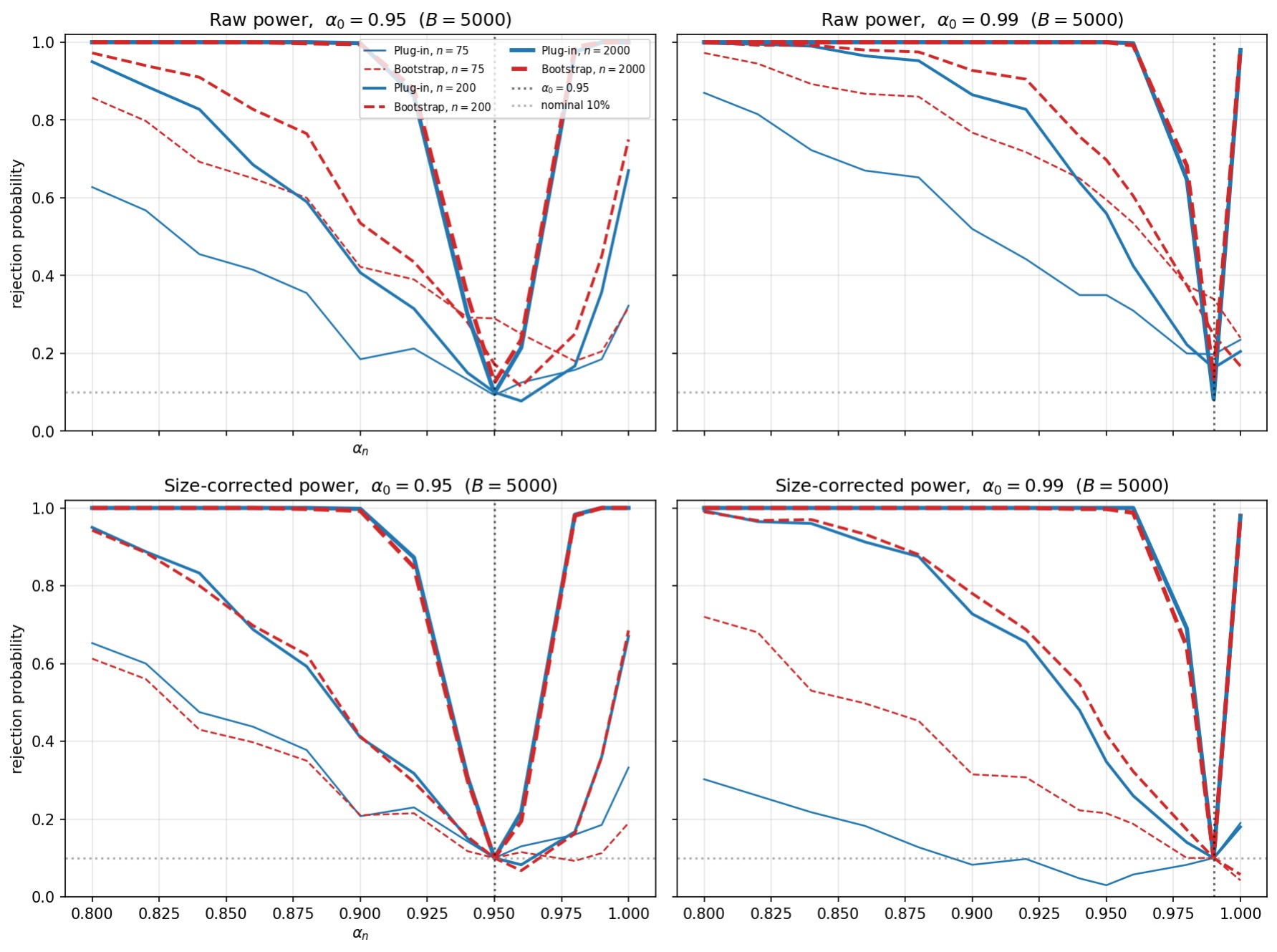}
    \caption{Raw (top row) and size-corrected (bottom row) power curves for
    the two-sided $10\%$ test $H_0: \alpha = \alpha_0$ versus
    $H_1: \alpha \neq \alpha_0$, using the plug-in procedure of
    Theorem~5 and the bootstrap-normal procedure. Left column:
    $\alpha_0 = 0.95$. Right column: $\alpha_0 = 0.99$. Raw rejection rates
    use the nominal critical value $z_{0.95} = 1.6449$; size-corrected
    rejection rates use empirical critical values calibrated to give exactly
    $10\%$ size under $H_0$. The vertical dotted line indicates $\alpha_0$;
    the horizontal dotted line indicates the nominal $10\%$ level. Based on
    $N = 400$ Monte Carlo replications and $B = 5000$ bootstrap replications
    per cell, with Poisson INARCH DGP ($\mu = 1$).}
    \label{fig:power_curves}
\end{figure}
%We first generate \(N = 400\) samples of size \(n_j\) from\[(X_t \mid \mathcal{F}_{t-1}) \sim \text{Poisson}(\mu + \alpha_0 X_{t-1}), \qquad t = 1,\dots,n_j.\]For each replication \(r\), we compute the bootstrap-based standard deviation estimator \(\hat{\sigma}_{n_j,r}\) using \(B = 400\) bootstrap samples, and form the test statistic\[t_{n_j,r}(\alpha_0) = \frac{\hat{\alpha}_{n_j,r} - \alpha_0}{\hat{\sigma}_{n_j,r}}.\]The empirical size is then estimated as \[\hat{\text{Size}} = \frac{1}{N} \sum_{r=1}^{N} \mathbbm{1}_{\left\{ \bigl|t_{n_j,r}(\alpha_0)\bigr| > z_{1-0.10/2} \right\}},\]where \(z_{1-0.10/2} = \Phi^{-1}(1 - 0.10/2)\).

From the top two panels on raw power, we see that the bootstrap test is uniformly more over-sized than the plug-in test. This is consistent with the findings of Table~4. Because of the size-distortion, we report in the lower panel the size-corrected power of the test. %At $n=2000$ and $n=200$, the test has quite good size-corrected power, regardless of the method (plug-in or bootstrap), whenever $\alpha_n$ is not too close to $\alpha_0$. 
At $n=2000$ the size-corrected power is essentially saturated ($\geq 0.99$) once $|\alpha_n - \alpha_0|\geq 0.04$; at $n=200$ it exceeds $0.65$ once $|\alpha_n - \alpha_0|\geq 0.07$, regardless of method (plug-in or bootstrap).
At $n=75$, and $\alpha_n=0.99$, however, the size-corrected power of the plug-in method almost collapses. This can be explained by the fact that in this case, $k_n=\frac{1}{1-\alpha_n}=100$ is even larger than $n=75$. In other words, for ultra-small sample size and $\alpha_n$ very close to 1, the mildly stationary asymptotic approximation may be less accurate than the local-to-unity assumption. Indeed, in this case, we expect $\hat{k}_n$  to have a large bias, which makes the plug-in method ineffective. The bootstrap method, on the other hand, retains some power since it is not affected by the bias of $\hat{k}_n$. 
Finally, at moderate sample sizes ($n \in \{75, 200\}$) and for alternatives close to $\alpha_0$, power against alternatives below $\alpha_0$ exceeds power against alternatives the same distance above, reflecting the negative finite sample bias of $\hat\alpha_n$ documented in Section~\ref{sec:simulations62}. At $n = 2000$ this asymmetry can reverse for intermediate distances, because the standard error of $\hat\alpha_n$ shrinks as $\alpha$ approaches unity, an effect that dominates the diminishing bias. The practical implication is unchanged: non-rejection for values of $\alpha_0$ close to one at moderate sample sizes should not be interpreted as evidence for mild stationarity. It mainly reflects limited power to distinguish mildly stationary, local-to-unity, and nearby mildly explosive regimes in the unit-root neighborhood.
%Finally, power against alternatives below $\alpha_0$ exceeds power against alternatives the same distance above $\alpha_0$ for both procedures, reflecting the negative finite sample bias of $\hat\alpha_n$ documented in Section~\ref{sec:simulations62}. Hence, non-rejection for values of $\alpha_0$ close to one should not be interpreted as evidence for mild stationarity. It mainly reflects limited power to distinguish mildly stationary, local-to-unity, and nearby mildly explosive regimes in the unit-root neighborhood.%For instance, at $(n, \alpha_0) = (200, 0.95)$, The plug-in test has size-corrected power $0.280$ against $\alpha_{\mathrm{DGP}} = 0.92$ (three points below the null) versus $0.155$ against $\alpha_{\mathrm{DGP}} = 0.98$ (three points above). This asymmetry is intrinsic to the testing problem rather than to either method, and explains why the largest size-corrected powers in Table~\ref{tab:power_sc} occur at alternatives below the null.
\section{Applications}
\label{sec:7empirical}
%\section{Applications}
%\label{sec:applications-revised}

This section applies the inference developed above to three datasets relevant for insurance and finance.  The first two are count-valued: annual wind-related event counts in Arizona and quarterly U.S. business bankruptcy filings.  The third is the monthly effective federal funds rate, expressed in basis points. This series is positive-valued, but with a significant portion of values near zero due to the ZLB. An additional benchmark comparison with the INAR(1) crime-count analysis
of \citet{peng2024note} is deferred to Appendix~B.%As we have explained in Section 6, the aim is not to conduct a formal unit-root test, but to check whether the data are compatible with  the mildly stationary assumption.

\subsection{Data}
\label{sec:applications-data}

The Arizona data are taken from the Spatial Hazard Events and Losses Database for the United States (SHELDUS)\footnote{\url{https://sheldus.asu.edu}}, which is a popular dataset in the economic/insurance literature, see e.g. \cite{barrot2016input, bourdeau2020natural}. We use the annual number of wind-related records in Arizona from 1960 to 2022. The series is nonnegative and integer-valued, with \(n=63\) observations. Modeling natural disaster counts using affine models has a long tradition in the count time series literature, see e.g., \cite{cui2016conditional,pei2025forecasting}.

The bankruptcy data are quarterly business bankruptcy filings compiled by the American Bankruptcy Institute from U.S. Courts statistics.\footnote{The source
table is ABI, ``Quarterly U.S. Business Filings (1980-present),''
\url{https://abi-org.s3.amazonaws.com/Newsroom/Bankruptcy_Statistics/QUARTERLY-BUSINESS-1980-PRESENT.pdf}.} Because the Bankruptcy Abuse Prevention and Consumer Protection Act (BAPCPA) took effect on October 17, 2005, generating a filing rush followed by a sharp
level shift, we use the pre-BAPCPA window from 1980Q1 to 2005Q3 as the main bankruptcy sample.  This gives \(n=103\) quarterly observations. Recently, \cite{agosto2016modeling} apply an INARCH-type model to similar corporate default count data.

The interest-rate application uses the monthly effective federal funds rate
from FRED\footnote{\url{https://fred.stlouisfed.org/series/FEDFUNDS}},
converted from percentage points to basis points.  %The sample runs from July 1954 to April 2026 and contains \(n=862\) observations. 
This series is positive valued, with spells near the zero lower bound. Previously, \cite{monfort2016staying} apply the affine ARG0 model to a similar dataset (Japanese government bond yield). %We follow \cite{monfort2016staying} and use data in 

Figure 5 plots the paths of these three series. 
\begin{figure}[H]
    \centering
    \IfFileExists{section7_timeseries_panel.png}{%
        \includegraphics[width=0.95\textwidth]{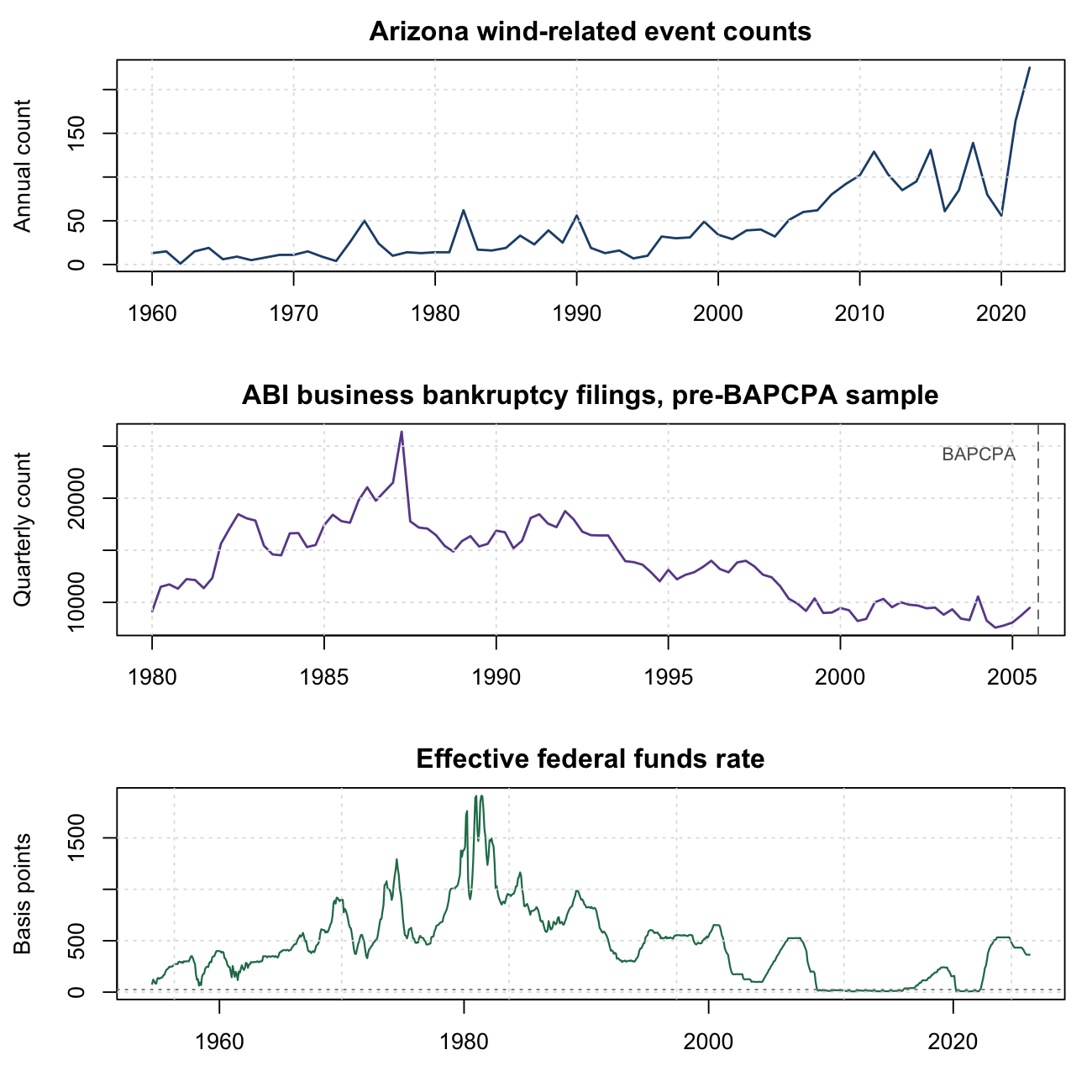}%
    }{%
        \fbox{\parbox{0.90\textwidth}{\centering
        Placeholder for the three-panel time-series plot: Arizona wind counts,
        pre-BAPCPA ABI business bankruptcy filings, and monthly FEDFUNDS in
        basis points.}}%
    }
    \caption{Path of the three time series.}
    \label{fig:section7-data-panel}
\end{figure}

\subsection{Main Estimates and Persistence Scale}
\label{sec:applications-estimates}
For this section, we use the specification:
\[
    \alpha_n = 1-\frac{1}{k_n}, \qquad \text{ where } k_n=n^\tau, \qquad 0<\tau<1.
\]
For each dataset, we estimate the model parameters $\alpha_n, \mu_n$ and $\sigma^2$, and also compute 
$\hat k_n= \frac{1}{1-\hat\alpha},  
    \frac{\hat k_n}{n},  
    \hat\tau = \frac{\ln \hat k}{\ln n}.
$ 
%In order for a dataset to be consistent with the mildly stationary assumption, we expect that for this dataset, $\frac{\hat k_n}{n}$ and $\hat\tau$ are smaller than 1. 
Table~\ref{tab:section7-main-estimates} reports these estimates.  %The reported \(n\) is the number of observations in each sample.
For FEDFUNDS, \(\hat\mu\) and \(\hat\sigma^2\) are reported after converting the rate to basis points.

\begin{table}[H]
\centering
\caption{Main estimates and persistence scales}
\label{tab:section7-main-estimates}
\scriptsize
\resizebox{\textwidth}{!}{%
\begin{tabular}{@{}l l r r r r r r l@{}}
\toprule
Series & Type of data/variable& \(n\) & \(\hat\alpha_n\) & \(\hat\mu_n\) & \(\hat\sigma^2\)
& \(\hat k_n\) & \(\hat k_n/n\) & \(\hat\tau\) \\
\midrule
Arizona wind counts
& natural disaster/count
& 63 & 0.941 & 5.9 & 16.6 & 16.9 & 0.27 & 0.685 \\
ABI bankruptcies, pre-BAPCPA
& default/count
& 103 & 0.930 & 973 & 138.3 & 14.3 & 0.14 & 0.575 \\
FEDFUNDS, monthly
& interest rate/positive 
& 862 & 0.990 & 4.9 & 5.0 & 101.7 & 0.12 & 0.684 \\
\bottomrule
\end{tabular}%
}
\end{table}

All three series have \(\hat \alpha_n\) close to one, suggesting high persistence.  At the same time, \(\hat k/n\) remains below
one in each case.  This distinction is especially important for FEDFUNDS:
although \(\hat\alpha_n=0.990\) is much closer to one than the two count-series
estimates, the corresponding scale \(\hat k=101.7\) is still small relative to
the full monthly sample size $n=862$.  Thus the full-sample FEDFUNDS estimate is
near-unit in the economically relevant sense, but it remains compatible with
the mild stationarity assumption that  \(k_n=o(n)\). We should also keep in mind that although $\hat{k}_n/n$ is smaller than 1 for all three datasets, its value is fairly large, especially for the Arizona data (0.27). Moreover, the implied value of $\hat{\tau}$ is also close to the value of 0.8 used in the simulations of section 6. These datasets therefore fall in the regime where, the normal approximation could be only moderately accurate.

Note that the estimate of \(\hat \alpha_n\) (and $\hat{k}_n$)  can be sensitive to the choice of the window. For instance, \cite{monfort2016staying} focused on the period 1995-2014 for their Zero Lower Bound application since the ultra low interest rate period started in Japan in the mid-1990s. If we follow this approach for FEDFUNDS, and use the period 2005-2026, then we would get an estimate $\hat{k}_n$ of 440, which is even larger than the sample size $n=255$; if instead we use the period 2007-2026, then we would get $\hat{k}_n=102$ against $n=231$. This substantial variation of $\hat{k}_n$ is consistent with the fact that between 2005 and 2007 (pre-subprime crisis period), interest rates in the US increased quickly, leading to a local explosiveness feature of the FED effective rates during this period. This echoes our discussions in section 4.3, that under mild stationarity, a process can feature local explosion. For this reason, for each dataset, we used the maximal window at our disposal (except for ABI bankruptcy count data, for which we removed the post-reform period since the reform led to a regime switch).

Finally, a direct diagnostic for the affine variance specification is the slope $\widehat{a}$ of $\log \widehat{W}_t^2$ on $\log X_{t-1}$, which estimates the exponent in $\mathbb{V}\mathrm{ar}(X_t \mid X_{t-1}) \propto X_{t-1}^{a}$ for large values of $X_{t-1}$. Under the affine specification, the theoretical value is $a=1$. Empirically, we obtain $\widehat{a} = 1.15 \pm 0.51$ (at 95\% confidence interval) for the Arizona wind counts ($n = 62$, $R^2 = 0.25$), $\widehat{a} = 1.02 \pm 1.46$ for the ABI pre-BAPCPA bankruptcies ($n = 102$, $R^2 = 0.02$), and $\widehat{a} = 0.60 \pm 0.13$ for the FEDFUNDS monthly series 1954--2026 ($n = 861$, $R^2 = 0.09$). %\footnote{The low $R^2$ values reflect the structural noise of the log-squared-residual regression: $\log \widehat{W}_t^2$ has an irreducible $\log \chi^2$-type component of variance $\pi^2/2 \approx 4.93$, so even under correctly specified affine variance the achievable $R^2$ is bounded above by roughly 0.15--0.25 for series of modest log-dispersion; this is the same phenomenon underlying low $R^2$ in White (1980) and Engle (1982) heteroskedasticity-test regressions, where power comes from $n$ rather than from $R^2$. The informative statistics are $\widehat{a}$ and its confidence interval.} 
The two count-valued series are consistent with the affine value $a = 1$. The FEDFUNDS estimate is sublinear, but is significantly different from 0 (corresponding to a linear AR(1) model). It might be worthwhile to investigate other non-affine conditional variance specifications of the form $\mathbb{V}ar[X_t|X_{t-1}]=\beta X^a_{t-1}+\delta$, where the exponent $a$ is between 0 and 1 and is an extra parameter. %Such an extension, however, is clearly out of the scope of this paper. As a consequence, the affine specification is still the most appropriate near unit root type theory at our disposal for this data.
For this data, the affine specification remains a useful benchmark, although the FEDFUNDS diagnostic suggests that sublinear variance specifications may be worth studying in future work.

%sub-window estimates on either side of the BVolcker / ZLB transitions individually return $\widehat{a}$ near one (e.g., $\widehat{a}=1.18 \pm 0.34$ on 1954--2007 and $\widehat{a}=1.00 \pm 0.16$ on 1988--2026), so the full-sample value reflects regime mixing across structurally different rate environments rather than a global violation of the affine specification. %We report the full-window value here, consistent with the whole-sample inferential framing of Section~\ref{sec:spurious-explosion}.

\subsection{Inference Within the Mildly Stationary Family}
\label{sec:applications-pvalue}

  %Both functions should be read as compatibility maps within the mildly stationary family.  They report which candidate values of \(\alpha\) are consistent with the data under the corresponding approximation.  They should not be interpreted as formal tests of the exact unit-root null \(\alpha=1\), since the exact boundary is outside the mildly stationary asymptotic regime.
Table 6 below reports the confidence interval for $\alpha_n$, obtained either using the plug-in method or the bootstrap method (with $B=50,000$).
\begin{table}[H]
\centering
\caption{Plug-in and bootstrap-normal inference for \(\alpha_n\)}
\label{tab:section7-plugin-bootstrap}
\scriptsize
\begin{tabular}{@{}l r c r c@{}}
\toprule
Series
& Plug-in SE
& Plug-in 10\% region
& Bootstrap SE
& Bootstrap 10\% region \\
\midrule
Arizona wind counts
& 0.0855 & \([0.800,1.000]\)
& 0.1292 & \([0.728,1.000]\) \\
ABI bankruptcies, pre-BAPCPA
& 0.0396 & \([0.865,0.995]\)
& 0.0687 & \([0.817,1.000]\) \\
FEDFUNDS, monthly
& 0.0068 & \([0.979,1.000]\)
& 0.0120 & \([0.970,1.000]\) \\
\bottomrule
\end{tabular}
\end{table}

The plug-in regions are narrower than the bootstrap-normal regions in all
three applications because the bootstrap-normal standard errors are larger.
For example, in the pre-BAPCPA bankruptcy sample, the plug-in region is
\([0.865,0.995]\), whereas the bootstrap-normal region reaches the upper end
of the plotted grid.  These regions should be interpreted only as
compatibility regions within the mildly stationary family: they indicate which
candidate values of \(\alpha\) are not rejected under that approximation.  They
do not constitute a formal unit-root test, and the non-rejection of values very
close to one should not be read as evidence that the exact unit-root boundary
has been ruled out.

Figure \ref{fig:section7-pvalue-panel} below reports the p-value functions of the test. 
\begin{figure}[H]
    \centering  \IfFileExists{section7_pvalue_panel.png}{%
        \includegraphics[width=0.98\textwidth]{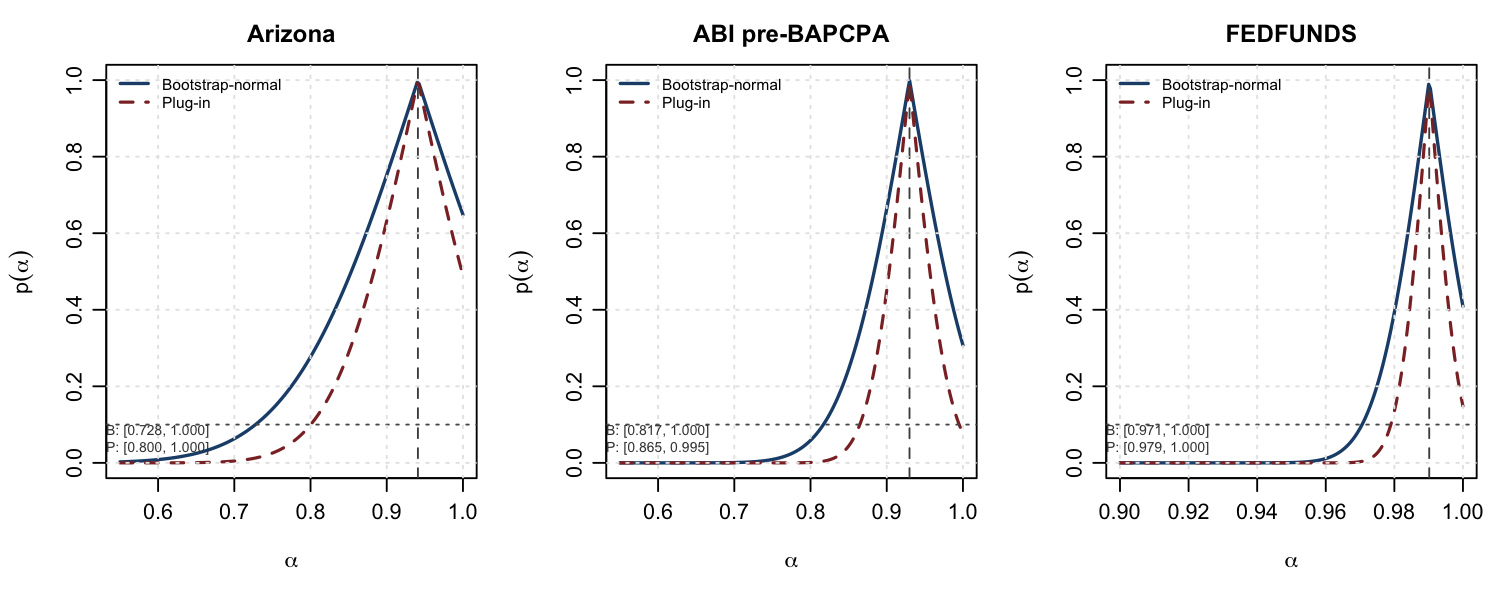}%
    }{%
        \fbox{\parbox{0.90\textwidth}{\centering
        Placeholder for the three-panel bootstrap-normal \(p(\alpha)\) plot.
        Each panel should show \(p(\alpha)\), a horizontal 10\% line, a vertical
        line at \(\hat\alpha\), and the 10\% non-rejection region.}}%
    }
    \caption{Plug-in and bootstrap-normal \(p(\alpha)\) functions for the three empirical applications.}
    \label{fig:section7-pvalue-panel}
\end{figure}

In these p-value plots, any value of $\alpha$ whose p-value is larger than 10\% is compatible with the mild stationarity assumption. For instance, for the Arizona data, we cannot reject any $\alpha$ between 0.73 and 1 using the bootstrap method, or any $\alpha$ between 0.8 and 1 using the plug-in method. When a non-rejection region reaches the upper grid boundary, this should be
read as weak separation between candidate values of \(\alpha_0<1\) close to one
and the exact unit-root boundary, not as evidence in favor of the exact
unit-root model.

Note that it is not automatically the case that when the same procedure is
applied to any data, values near 1 will have a \(p\)-value that is larger than
10 percent.  In Appendix B, we also propose a benchmark
comparison using the break-and-enter dwelling crime count data used by
\citet{peng2024note}.  For two of the three cities considered there, the
upper boundary \(\alpha=1\) is rejected under both the plug-in and
bootstrap-normal calibrations, and it is only borderline for the remaining
city.  This comparison shows that the non-rejection of values close to one in
our main applications reflects the persistence and uncertainty structure of
those samples, rather than a mechanical consequence of the \(p(\alpha)\)
construction.

\section{Conclusion}
\label{sec:conclusion}
This paper has established the near-unit-root theory for a class of Markov affine processes with exploding conditional variance, encompassing count, positive, or nonnegative processes. We have considered both the local-to-unity and the mildly integrated frameworks. The mildly stationary framework with autoregressive coefficient $\alpha_n=1-\frac{1}{k_n}$ is of particular interest, as it $(i)$ solves the nuisance parameter problem; $(ii)$ leads to more plausible stochastic properties of the process than the mildly integrated AR(1) model with intercept; $(iii)$ admits feasible inference, with a consistent and jointly asymptotically normal OLS estimator, a plug-in procedure that does not require knowledge of $k_n$, and an asymptotically valid randomly weighted bootstrap; and $(iv)$ can generate bubble-like episodes without locally explosive
least-squares estimates.

Several directions remain open. First, the WLS estimator studied in
Proposition~\ref{transient} requires the condition $2\mu > \sigma^2$; its properties when this condition fails are
unknown, even in the exact unit-root case \citep{wei1990estimation}. Second,  simulations have revealed that the finite-sample bias in $\hat{\alpha}_n$ can materially affect inference through either the plug-in or the bootstrap methods. This suggests that bias correction for \(\hat\alpha_n\) may be a useful direction for future work. Third, our inference is carried out within the mildly stationary family; developing a formal test that discriminates between the mildly stationary and local-to-unity regimes is a much harder problem, especially given the nuisance parameter problem of the local-to-unity case. This is left for future research. 

%Moreover, it admits a valid randomly weighted bootstrap, which is convenient for simple statistical inference. %, and We have shown that the local-to-unity model suffers from many similar issues of the local-to-unity AR(1) model, but the mildly integrated model not only inherits many nice properties of the local-to-unity AR(1) model, but also 

 \bibliographystyle{apalike}
 \bibliography{lib}
%this explosiveness of the variance leads to drastically different properties than those previously found for the linear AR(1) model, and many of which could be more plausible for real applications involving positive or count processes. 

%Recently, \cite{liu2023asymptotic, peng2024note} replace, for an INAR(1) process, the local-to-unity condition $a_n=1-\frac{\gamma}{n}$ by \begin{equation}\label{mildly}a_n=1-\frac{\gamma}{k_n},\end{equation}where $a_n$ is the thinning parameter of the INAR(1) process and $k_n$ is a sequence that increases to infinity more slowly than $n$: $k_n \rightarrow \infty,  \frac{k_n}{_n} \rightarrow 0$. For instance, if $k_n=n^{\tau}$ where $\tau$ is close to zero, then the estimator of $\mu$ in their mildly integrated INAR(1) process has a convergence rate of $\sqrt{n^{1-\tau}}$, which is close to $\sqrt{n}$ when $\tau$ is close to zero. In other words, the convergence rate seems to be ``continuous" with the standard parametric rate $\sqrt{n}$.  %In this section, we explore whether in the mildly integrated case eq.\eqref{mildly}, the CLS estimator of $\alpha$ and $\mu$ could both converge. 
\appendix

\section{Appendix}
\subsection{Proof of Theorem \ref{scaling}}
\label{proofscaling}
%This theorem extends Theorem~1 of \cite{barreto-souza2023}, which is established for the Poisson INARCH model. As noted in Remark~3 of \cite{barreto-souza2023}, their proof can be extended beyond the Poisson setting under suitable conditions. 

Let
\[
Y_n(s)=\frac{X_{\lfloor ns\rfloor}}{n}.
\]
We verify the local characteristics of the Markov chain \(Y_n\). Conditional
on \(X_{t-1}=x\), put
\[
\Delta_{n,t}:=\frac{X_t-X_{t-1}}{n}.
\]
Then
\[
n\mathbb{E}[\Delta_{n,t}\mid X_{t-1}=x]
=
(\alpha_n-1)x+\mu_n
=
\gamma_n\frac{x}{n}+\mu_n,
\]
which converges locally uniformly, for \(x/n\) in compact sets, to
\[
\gamma y+\mu,\qquad y=x/n.
\]
Moreover,
\[
\begin{aligned}
n\mathbb{E}[\Delta_{n,t}^2\mid X_{t-1}=x]
&=
\frac1n\mathbb{E}[(X_t-X_{t-1})^2\mid X_{t-1}=x]\\
&=
\frac{\beta_nx+\delta_n}{n}
+
\frac{\{(\alpha_n-1)x+\mu_n\}^2}{n},
\end{aligned}
\]
which converges locally uniformly to \(\sigma^2 y\).

It remains to verify the local Lindeberg condition. Let \(p>2\) be the
exponent in Assumption~\ref{AssAffineMoment}. By the affine representation and Rosenthal's inequality,
\[
\mathbb{E}\!\left[
|W_t|^p\mid X_{t-1}=x
\right]
\le C(1+x^{p/2})
\]
uniformly in \(n\), where
\[
W_t=X_t-\mathbb{E}[X_t\mid X_{t-1}].
\]
Since
\[
X_t-X_{t-1}
=
W_t+(\alpha_n-1)x+\mu_n,
\]
and since \((\alpha_n-1)x+\mu_n\) is uniformly bounded on
\(\{x/n\le K\}\), it follows that
\[
\mathbb{E}\!\left[
|X_t-X_{t-1}|^p\mid X_{t-1}=x
\right]
\le C_K n^{p/2}.
\]
Therefore, for every \(\varepsilon>0\),
\[
\begin{aligned}
n\mathbb{E}\!\left[
\Delta_{n,t}^2
\mathbf 1_{\{|\Delta_{n,t}|>\varepsilon\}}
\mid X_{t-1}=x
\right]
&\le
n\varepsilon^{-(p-2)}
\mathbb{E}\!\left[
|\Delta_{n,t}|^p
\mid X_{t-1}=x
\right]  \\
&\le
C_K\varepsilon^{-(p-2)} n^{1-p/2}
\longrightarrow 0,
\end{aligned}
\]
locally uniformly for \(x/n\le K\).

The limiting local characteristics are therefore
\[
b(y)=\mu+\gamma y,\qquad a(y)=\sigma^2y.
\]
The martingale problem associated with
\[
\mathcal Af(y)
=
(\mu+\gamma y)f'(y)
+\frac12\sigma^2y f''(y)
\]
is well posed, and its solution is the CIR diffusion
\[
d\Upsilon_s=(\mu+\gamma\Upsilon_s)\,ds
+\sigma\sqrt{\Upsilon_s}\,dB_s,
\qquad \Upsilon_0=0.
\]
By the standard diffusion approximation theorem for Markov chains
\citep[Theorem~7.4.1]{ethier1986markov},
\[
\left(\frac{X_{\lfloor ns\rfloor}}{n}\right)_{s\ge0}
\Rightarrow
(\Upsilon_s)_{s\ge0}.
\]
This proves Theorem~\ref{scaling}.

%Below, we only provide an informal proof. By the affine property, we have:$$ \mathbb{E}[Y_{s+\frac{1}{n}}\vert Y_s]=\frac{1}{n}\mathbb{E}[X_{\lfloor n(s +\frac{1}{n})\rfloor} \vert X_{\lfloor ns\rfloor}]=\frac{1}{n} (\alpha_n  X_{\lfloor ns\rfloor}+\mu_n)=\alpha_n Y_s+\mu_n$$ or$$\mathbb{E}[Y_{s+\frac{1}{n}}-Y_s\vert Y_n]=\frac{1}{n }(\gamma Y_s+\mu_n)=\frac{1}{n }(\gamma Y_s+\mu)+o(1/n).$$Similarly, the corresponding conditional variance is:$$\mathbb{V}[Y_{s+\frac{1}{n}}-Y_s\vert Y_n]=\frac{1}{n} (\beta_n Y_s+ \frac{\delta_n}{n})= \frac{1}{n} \sigma^2 Y_s+o(1/n).$$From these conditional mean and variance, we can conclude that the limiting process of $(Y_s)$ is the diffusion \eqref{sde}. 

\subsection{Proof of Theorem \ref{general}}
\label{proofgeneral}
If \(x\) is integer-valued, the desired weak convergence becomes CLT. It therefore suffices to establish the result for non-integer values of \(x\), which arise only when the state space of \((X_t)\) is continuous. In this case, the LT \(e^{-a}\) is infinitely divisible.

Write \(x = \lfloor x \rfloor + (x - \lfloor x \rfloor)\). By the affine property, conditional on \(X_{t-1} = x\), the random variable \(X_t\) admits the decomposition
\[
X_t \overset{d}{=} \sum_{i=1}^{\lfloor x \rfloor} Y_i + \epsilon + V,
\]
where $(Y_i)$, $i=1,\dots$ are i.i.d. random variables with LT $e^{-a}$, $\epsilon$ has LT $e^{-b(u)}$ and $V$ has LT $e^{-(x- \lfloor x \rfloor) a}$, with $Y_i$, $\epsilon$, $V$ mutually independent. Then it suffices to remark that for fixed values of $x-\lfloor x \rfloor$, as $\lfloor x \rfloor$ increases to infinity, both $\epsilon$ and $V$ are negligible, while the first term $\sum_{i=1}^{ \lfloor x \rfloor } Y_{i}$ is approximately Gaussian, by the CLT, after appropriate centering and normalization. Theorem \ref{general} follows.

\subsection{Proof of Theorem~\ref{thm1}}
\label{proofthm1}
%The proof of Part $a)$ is similar to \cite{wei1989some}, and is omitted.
%In what follows, we will only prove the results for the INARCH model \eqref{inarch}, where $\sigma^2=1$. The proof for the general case is similar.
Let us %now prove part $b)$. We 
introduce the rescaled terms:
\begin{align*}
&\tilde{F}_n:= \begin{bmatrix}
n^{-3/2} & 0 \\
0 & n^{-1/2}
\end{bmatrix} F_n
  \begin{bmatrix}
n^{-3/2} & 0 \\
0 & n^{-1/2}
\end{bmatrix} = \begin{bmatrix}
        n^{-3} \sum_{t=1}^n X_{t-1}^2& n^{-2} \sum_{t=1}^n X_{t-1}\\
        n^{-2} \sum_{t=1}^n X_{t-1}& 1
    \end{bmatrix},
\\
&\tilde{d}_n:=\begin{bmatrix}
n^{-2} & 0 \\
0 & n^{-1}
\end{bmatrix} d_n = \begin{bmatrix}
        n^{-2} \sum_{t=1}^n X_{t-1}W_t \\
        n^{-1} \sum_{t=1}^n W_t
    \end{bmatrix}.
\end{align*}
Then we use the following Lemma:
\begin{lemma}
\label{lem:firstlemma}
We have the joint weak convergence:
\begin{itemize}
    \item $a)$ of the two processes: the martingale process $(M_s)$ defined by $M_s=\sum_{t=1}^{[ns]} W_t/n$ and $(Y_s)$:
    $$ 
    \begin{bmatrix}  
    (X_{\lfloor n s \rfloor }/n)  \\
        ( \sum_{t=1}^{[ns]} W_t/n)
    \end{bmatrix}
    \overset{w}{%
\longrightarrow } 
\begin{bmatrix}
(\Upsilon_s) \\
(\sigma \int_{0}^{s}\sqrt{\Upsilon_{r}}dB_r)
\end{bmatrix}
$$
\item $b)$ of the following random variables:
\begin{equation}
\label{jointweakconvergence}
\begin{bmatrix}
    n^{-3} \sum_{t=1}^n X_{t-1}^2 \\
     n^{-2} \sum_{t=1}^n X_{t-1} \\
      n^{-2} \sum_{t=1}^n X_{t-1}W_t \\
        n^{-1} \sum_{t=1}^n W_t
\end{bmatrix} \overset{w}{\rightarrow}
\begin{bmatrix}
    \int_0^1 \Upsilon^2_s \mathrm{d}s    \\
    \int_0^1 \Upsilon_s \mathrm{d}s \\
    \int_0^1 \sigma \Upsilon_s^{3/2} \mathrm{d}B_s\\
     \int_0^1 \sigma \Upsilon_s^{1/2} \mathrm{d}B_s
\end{bmatrix}.
\end{equation}
\end{itemize}

\end{lemma}
\begin{proof}[Proof of Lemma~\ref{lem:firstlemma}]
Convergence $a)$ is simply due to the fact that:
\begin{equation}
\frac{1}{n} \sum_{t=1}^{\lfloor n s \rfloor } W_t=\frac{1}{n}\sum_{t=1}^{\lfloor n s \rfloor } \left( X_t - \mu_n - \alpha_n X_{t-1} \right) = \frac{X_{\lfloor n s \rfloor }}{n}-\frac{\gamma}{n^2} \sum_{t=1}^{\lfloor n s \rfloor } X_{t-1}- \frac{\lfloor n s \rfloor}{n} \mu_n,
\end{equation}
which, by the continuous mapping theorem (CMT), converges to $\Upsilon_s- \gamma \int_0^s \Upsilon_r \mathrm{d}r-s \mu=\sigma \int_{0}^{s}\sqrt{\Upsilon_{r}}dB_r$.

The first, second, and fourth convergences in $b)$ are simple consequences of the CMT. The third convergence says that when the two processes in Lemma~\ref{lem:firstlemma} $a)$ converge jointly, their stochastic integral also converges to the corresponding limiting stochastic integral. This kind of result requires some tightness argument. We refer to the proof of Theorem~1 in \cite{wls} for a similar proof.
\end{proof}

Let us now prove Theorem \ref{thm1}. By CMT, the joint distribution of $(\tilde{F}_n, \tilde{d}_n)$ converges to
$$
\left(F_{\infty},d_{\infty}):=\Big( \begin{bmatrix}
 \int_0^1 \Upsilon^2_s \mathrm{d}s    &   \int_0^1 \Upsilon_s \mathrm{d}s \\
  \int_0^1 \Upsilon_s \mathrm{d}s & 1
\end{bmatrix}, 
\begin{bmatrix}
     \int_0^1 \sigma \Upsilon_s^{3/2} \mathrm{d}B_s \\
    \int_0^1 \sigma \Upsilon_s^{1/2} \mathrm{d}B_s 
\end{bmatrix}
\right)
$$

 Then we have:
\begin{align*}
    F_n^{-1}d_n&= \begin{bmatrix}
n^{-3/2} & 0 \\
0 & n^{-1/2}
\end{bmatrix} \tilde{F}_n^{-1}  \begin{bmatrix}
n^{-3/2} & 0 \\
0 & n^{-1/2}
\end{bmatrix}
 \begin{bmatrix}
n^{2} & 0 \\
0 & n
\end{bmatrix} 
\tilde{d}_n 
= \begin{bmatrix}
n^{-1} & 0 \\
0 & 1
\end{bmatrix} \tilde{F}_n^{-1} \tilde{d}_n
\end{align*}

Hence, the distribution of $$ \begin{bmatrix}
n  & 0 \\
0 & 1
\end{bmatrix} F_n^{-1}d_n=\begin{bmatrix}
n  & 0 \\
0 & 1
\end{bmatrix} \begin{bmatrix} \hat{\alpha}-\alpha_n \\
\hat{\mu}_n- \mu_n
 \end{bmatrix}=\begin{bmatrix}n (\hat{\alpha}-\alpha_n)\\
\hat{\mu}_n- \mu_n
 \end{bmatrix}$$ converges to the distribution of $F_{\infty}^{-1}d_{\infty}$. Theorem~\ref{thm1} follows. %In other words, we get the weak convergence in eq.\eqref{cmt}.
\subsection{Proof of Part $a)$ of Theorem \ref{cls}}
By the normal equations, we have:
\label{proofclsa}
\begin{equation}
\label{cls_eq_2}
(\hat{\alpha}_n, \hat{\mu}_n)':= {\underbrace{\begin{bmatrix}
\sum_{t=1}^n X_{t-1}^2 & \sum_{t=1}^n X_{t-1} \\
\sum_{t=1}^n X_{t-1} & n
\end{bmatrix}} _{:=F_n}}^{-1} 
{\begin{bmatrix}
\sum_{t=1}^n X_{t-1}X_t \\
\sum_{t=1}^n X_{t}
\end{bmatrix}.}
\end{equation}
Thus, after centering around their true values, we obtain the following:
\begin{equation}
\label{errorterm}
\begin{bmatrix} \hat{\alpha}_n  \\
\hat{\mu}_n 
 \end{bmatrix}-\begin{bmatrix} \alpha_n \\
 \mu_n
 \end{bmatrix}
 =F_n^{-1} \Big( 
 \begin{bmatrix}
\sum_{t=1}^n X_{t-1}X_t \\
\sum_{t=1}^n X_{t}
\end{bmatrix}
-F_n \begin{bmatrix}
\alpha_n\\
 \mu_n
 \end{bmatrix}  \Big)=F_n^{-1} 
\underbrace{\begin{bmatrix}
\sum_{t=1}^n X_{t-1}W_{t} \\
\sum_{t=1}^n W_{t}
\end{bmatrix}}_{:=d_n},
\end{equation}
where $$W_t= X_{t}-\alpha_n X_{t-1}- \mu_n $$ forms a martingale difference sequence (MDS). 

%Part $b)$ of Theorem \ref{cls} is based on 
Let us first prove the following two Lemmas, which are analogues of Lemma~\ref{lem:firstlemma}. 
\begin{lemma}[Mildly stationary LLN]
\label{lem:ms-lln}
Assume Assumptions~\ref{ass1}, \ref{Assbeta}, \ref{assumptionmildly}, and \ref{AssAffineMomentBis}, and suppose $\gamma<0$.
After the normalization $\gamma=-1$, let $Y_{n,t}=X_t/k_n$. Then
\[
\frac1n\sum_{t=1}^n Y_{n,t-1}^j
\overset{p}{\longrightarrow}
\mathbb E[\Upsilon_\infty^j],
\qquad j=1,2,3.
\]
Moreover, for every $\eta>0$,
\[
\lim_{R\to\infty}\limsup_{n\to\infty}
P\left(
\frac1n\sum_{t=1}^n
Y_{n,t-1}^3\mathbf 1_{\{Y_{n,t-1}>R\}}
>\eta
\right)=0.
\]
\end{lemma}
This lemma says that in the long-run, the mildly stationary process, once scaled by $k_n$, is ``almost ergodic". Let us, for instance, take $j=1$. Then:
\begin{align}
   {(nk_n)^{-1 }} \sum_{t=1}^n X_t  &= (nk_n)^{-1 }\int_{1}^{n}X_{\lfloor t \rfloor} \mathrm{d} t  \nonumber \\
   &= \frac{k_n}{n}\int_{1/k_n}^{n/k_n}\frac{X_{ \lfloor k_n r \rfloor}}{k_n} \mathrm{d} r \overset{w}{\rightarrow }  \frac{1}{n/k_n}\int_{0}^{{n/k_n}}\Upsilon_{r}\mathrm{d} r \qquad \text{for fixed $n/k_n$, and $n \rightarrow \infty$} \nonumber \\
   & \rightarrow \mathbb{E}[\Upsilon_{\infty}] \qquad \text{as $n/k_n\rightarrow \infty$}
\end{align}
where in the last step, we have used the ergodic theorem for the stationary process $(\Upsilon_s)$. Here, in the last step, the argument of letting $n$ and $n/k_n$ go to infinity separately is heuristic and has also been used in the unit root literature by \cite{phillips2010smoothing} [see their equations (6)-(7)]. Since this type of double asymptotic technique will be used throughout the paper, Lemma~\ref{lem:ms-lln} formalizes this argument. 
\begin{proof}[Proof of Lemma~\ref{lem:ms-lln}]
Let $\pi_n$ be the invariant distribution of $Y_{n,t}$. Since
$\alpha_n=1-1/k_n<1$, this invariant distribution is unique. By
Assumption~\ref{AssAffineMomentBis}, Rosenthal's inequality and the affine
moment recursion imply that, for the exponent $p>3$ in
Assumption~\ref{AssAffineMomentBis},
\[
\sup_n\int y^p\,\pi_n(dy)<\infty .
\tag{A.1}
\]
Indeed, for $X_t$ in stationarity, one has
\[
\mathbb E[X_t^p\mid X_{t-1}=x]
\le
\alpha_n^p x^p+C(1+x^{p-1}),
\]
uniformly in $n$. Dividing by $k_n^p$, integrating with respect to
$\pi_n$, and using $1-\alpha_n^p\asymp k_n^{-1}$ gives (A.1).

By tightness, take a subsequence such that $\pi_n\Rightarrow\pi$. Applying
Proposition~\ref{scalinglimit} with initial law $\pi_n$ shows that the
CIR diffusion with initial law $\pi$ is stationary. Hence $\pi$ is the
unique invariant law of the limiting CIR diffusion, namely
$\mathcal(\Upsilon_\infty)$. Therefore
\[
\pi_n\Rightarrow \mathcal(\Upsilon_\infty).
\tag{A.2}
\]
Together, (A.1) and (A.2) imply convergence of moments up to order three:
\[
\int y^j\,\pi_n(dy)\to \mathbb E[\Upsilon_\infty^j],
\qquad j=1,2,3.
\tag{A.3}
\]

For any bounded Lipschitz function $h$, the subcritical affine chain mixes
at rate $\alpha_n^\ell$. Thus
\[
\left|
\operatorname{Cov}_{\pi_n}\big(h(Y_{n,0}),h(Y_{n,\ell})\big)
\right|
\le C_h\alpha_n^\ell .
\]
Consequently, under stationarity,
\[
\operatorname{Var}\left(\frac1n\sum_{t=1}^n h(Y_{n,t-1})\right)
\le
\frac{C_h}{n^2}\sum_{\ell=0}^{n-1}(n-\ell)\alpha_n^\ell
\le C_h\frac{k_n}{n}\to0.
\]
The initial condition contributes only a transient term of order
$k_n/n=o(1)$. Hence
\[
\frac1n\sum_{t=1}^n h(Y_{n,t-1})
-
\int h(y)\pi_n(dy)
\xrightarrow{p}0.
\tag{A.4}
\]

Apply (A.4) to the bounded Lipschitz truncations
$h_{j,R}(y)=y^j\wedge R$, $j=1,2,3$, and then let $R\to\infty$. The
truncation error is negligible by (A.1). This proves
\[
\frac1n\sum_{t=1}^n Y_{n,t-1}^j
\xrightarrow{p}
\mathbb E[\Upsilon_\infty^j],
\qquad j=1,2,3.
\]
Finally, by Markov's inequality and (A.1),
\[
P\left(
\frac1n\sum_{t=1}^n
Y_{n,t-1}^3\mathbf 1_{\{Y_{n,t-1}>R\}}
>\eta
\right)
\le
\frac{C}{\eta}R^{3-p},
\]
which tends to zero as $R\to\infty$ because $p>3$.
\end{proof}

\begin{lemma}
\label{mildlystationary}
We have the following joint convergence of: 
\begin{itemize}
    \item $a)$ The processes: $$
      \begin{bmatrix}  
    (X_{\lfloor k_n s \rfloor }/k_n)  \\
        ( \sum_{k=1}^{[k_ns]} W_k/k_n)
    \end{bmatrix}
    \overset{w}{%
\rightarrow } 
\begin{bmatrix}
(\Upsilon_s) \\
(\sigma \int_{0}^{s}\sqrt{\Upsilon_{r}}dB_r)
\end{bmatrix}
    $$
        \item $b)$ the random vector
  $   \begin{bmatrix}
n^{-1/2}k_{n}^{-3/2}\sum_{t=1}^{n}X_{t-1}W_{t} \\
(nk_{n})^{-1/2}\sum_{t=1}^{n}W_{t}
\end{bmatrix}
\overset{w}{\longrightarrow } \mathcal{N}(0, \Sigma).$ 

    % \lim_{s\to \infty} \frac{1}{\sqrt{s}}\int_{0}^{{s}}\Upsilon_{r}d\mathcal{W}_r 
%\notag %\\
%& &\overset{w}{\rightarrow }U_2\notag  

\end{itemize}

\end{lemma}
%Here, $(i)$ and $(ii)$ %are laws of large number type results and roughly say that upon normalizing by $k_n$, the process $(\frac{X_t}{k_n}, t=1,...)$ and its squared version $(\frac{X^2_t}{k^2_n}, t=1,...)$ satisfy the ergodic property. They are to be compared with the first convergence in eq.\eqref{jointweakconvergence}. For local-to-unity processes (under which $\alpha_n=1-\frac{\gamma}{n}$ is ``too close to unity"), the law of large numbers (LLN) does not apply, and the sample mean of the process for $t=1,...,n$ is random. In the case of mildly stationary processes, $\alpha_n=1-\frac{\gamma}{k_n}$ is closer to stationarity, hence the sample mean is asymptotically deterministic. %$(iii)$ strengthens Proposition 2 on the scaling limit of $ (X_{\lfloor k_n s \rfloor }/k_n)$, by showing that this convergence is joint with the weak convergence of the martingale process $ (\sum_{k=1}^{[k_ns]} W_k/k_n)$. %$iv)$ is a consequence of $(iv)$. 

\begin{proof}[Proof of Lemma \ref{mildlystationary}] 
For convergence $a)$, we have:
\begin{equation}
\frac{1}{k_n} \sum_{t=1}^{[k_ns]} W_t=\frac{1}{k_n}\sum_{t=1}^{\lfloor k_ns \rfloor} \left( X_t - \mu_n - \alpha_n X_{t-1} \right) = \frac{X_{\lfloor k_ns \rfloor}}{k_n}-\frac{\gamma}{k^2_n} \sum_{t=1}^{\lfloor k_ns \rfloor} X_{t-1}- \mu_n \frac{\lfloor k_ns \rfloor}{k_n}, \label{dotW}
\end{equation}
%Dividing by $k_n$ on both sides, we get:\[\frac{  \sum_{t=1}^{[k_ns]} W_t}{k_n}= \frac{X_{\lfloor k_n s \rfloor}}{k_n} + \gamma \int_{0/k_n}^{(\lfloor k_n s \rfloor-1)/k_n}\frac{X_{ \lfloor k_n r \rfloor}}{k_n} \mathrm{d} r - \beta \frac{\lfloor k_n s \rfloor}{k_n}.\]where \begin{equation}\frac{\gamma}{k_n^2} \sum_{t=1}^{\lfloor k_ns \rfloor} X_{t-1}={\gamma}{k_n^{-2}} \int_{0}^{\lfloor k_ns \rfloor-1}X_{\lfloor t \rfloor} dt=\gamma \int_{0/k_n}^{(\lfloor k_n s \rfloor-1)/k_n}\frac{X_{ \lfloor k_n r \rfloor}}{k_n} \mathrm{d} r, \end{equation} using the method of change of variables, i.e., $t=k_n r $.
which, by CMT, converges weakly to:
\[
 \Upsilon_s - \gamma \int_0^s \Upsilon_r \mathrm{d} r - \mu s= \sigma\int_0^s \sqrt{\Upsilon_r}\mathrm{d}B_r.
\]

For part $b)$, define
\[
\xi_{n,t}:=
\begin{bmatrix}
n^{-1/2}k_n^{-3/2}X_{t-1}W_t\\
(nk_n)^{-1/2}W_t
\end{bmatrix}.
\]
Then $(\xi_{n,t},\mathcal F_t)$ is a martingale-difference array. Let us check the conditions to apply the martingale central
limit theorem (MCLT). Its
conditional covariance matrix equals
\[
\begin{bmatrix}
\frac1{nk_n^3}\sum X_{t-1}^2(\beta_nX_{t-1}+\delta_n)
&
\frac1{nk_n^2}\sum X_{t-1}(\beta_nX_{t-1}+\delta_n)
\\
\frac1{nk_n^2}\sum X_{t-1}(\beta_nX_{t-1}+\delta_n)
&
\frac1{nk_n}\sum(\beta_nX_{t-1}+\delta_n)
\end{bmatrix}.
\]
By Lemma~\ref{lem:ms-lln}, and since $\beta_n\to\sigma^2$, and $\delta_n=O(1)$,
this converges in probability to $\Sigma$.

It remains only to verify the conditional Lindeberg condition. Let
$p>3$ be the exponent in Assumption~\ref{AssAffineMomentBis}, and set
$\delta=p-2>1$. Assumption~\ref{AssAffineMomentBis} implies
\[
E[|W_t|^p\mid\mathcal F_{t-1}]
\le C(1+X_{t-1}^{p/2}).
\]
Fix $R>0$ and write $A_{n,t}(R)=\{Y_{n,t-1}\le R\}$. On $A_{n,t}(R)$,
$X_{t-1}\le Rk_n$, and hence
\[
\|\xi_{n,t}\|\le C_R\frac{|W_t|}{\sqrt{nk_n}}.
\]
Therefore,
\[
\sum_{t=1}^n
E\left[
\|\xi_{n,t}\|^p\mathbf 1_{A_{n,t}(R)}
\mid\mathcal F_{t-1}
\right]
\le C_R n^{1-p/2}=o(1).
\]
By Markov's inequality,
\[
\sum_{t=1}^n
E\left[
\|\xi_{n,t}\|^2
\mathbf 1_{\{\|\xi_{n,t}\|>\varepsilon\}}
\mathbf 1_{A_{n,t}(R)}
\mid\mathcal F_{t-1}
\right]
\le
\varepsilon^{-\delta}
\sum_{t=1}^n
E\left[
\|\xi_{n,t}\|^{2+\delta}
\mathbf 1_{A_{n,t}(R)}
\mid\mathcal F_{t-1}
\right]
=o(1).
\]
On $A_{n,t}(R)^c$, the Lindeberg term is bounded by the corresponding
conditional variance term. For $R\ge1$,
\[
\sum_{t=1}^n
E\left[
\|\xi_{n,t}\|^2\mathbf 1_{A_{n,t}(R)^c}
\mid\mathcal F_{t-1}
\right]
\le
C\frac1n\sum_{t=1}^n
Y_{n,t-1}^3\mathbf 1_{\{Y_{n,t-1}>R\}}+o_p(1),
\]
which is negligible as $R\to\infty$ by Lemma~\ref{lem:ms-lln}. Thus the
conditional Lindeberg condition holds. By the martingale central limit
theorem, part $b)$ follows.

\end{proof}

 Let us now prove Part $a)$ of Theorem~\ref{cls}.  
%We now have the ingredients necessary to prove part $b)$ of Theorem \ref{cls}.
We follow the proof of Theorem~\ref{thm1}, and define $\tilde{F_n}$ and $\tilde{d_n}$ by:
$$\tilde{F}_n:= \begin{bmatrix}
n^{-1/2} k_n^{-1} & 0 \\
0 & n^{-1/2}
\end{bmatrix} F_n \begin{bmatrix}
n^{-1/2} k_n^{-1} & 0 \\
0 & n^{-1/2}
\end{bmatrix}=\begin{bmatrix}
    (nk_n^2)^{-1} \sum_{t=1}^n X_{t-1}^2  & {(nk_n)^{-1 }} \sum_{t=1}^n X_{t-1} \\
    {(nk_n)^{-1 }} \sum_{t=1}^n X_{t-1} & 1
\end{bmatrix},  $$
and
$$\tilde{d}_n:=\begin{bmatrix}
n^{-1/2} k_n^{-3/2} & 0 \\
0 & n^{-1/2} k_n^{-1/2}
\end{bmatrix} d_n=\begin{bmatrix}
n^{-1/2}k_{n}^{-3/2}\sum_{t=1}^{n}X_{t-1}W_{t} \\
(nk_{n})^{-1/2}\sum_{t=1}^{n}W_{t}
\end{bmatrix}.
$$
Then by the previous two Lemmas, we get the joint convergence of:
$$
\tilde{F}_n \overset{w}{\rightarrow} \Omega, \qquad \tilde{d}_n \overset{w}{\rightarrow} \mathcal{N}(0, \Sigma),
$$
Thus:
\begin{align*}
    F_n^{-1}d_n&= \begin{bmatrix}
n^{-1/2} k^{-1}_n & 0 \\
0 & n^{-1/2}
\end{bmatrix} \tilde{F}_n^{-1} \begin{bmatrix}
n^{-1/2} k^{-1}_n & 0 \\
0 & n^{-1/2}
\end{bmatrix}
\begin{bmatrix}
n^{1/2} k_n^{3/2} & 0 \\
0 & n^{1/2} k_n^{1/2}
\end{bmatrix}
\tilde{d}_n \\
&= \begin{bmatrix}
(nk_n)^{-1/2} & 0 \\
0 & (k_n/n)^{1/2}
\end{bmatrix} \tilde{F}_n^{-1} \tilde{d}_n.
\end{align*}
Hence, the distribution of 
 $ 
\begin{bmatrix}
(nk_n)^{1/2} & 0 \\
0 & (n/k_n)^{1/2}
\end{bmatrix} \tilde{F}_n^{-1} \tilde{d}_n=
\begin{bmatrix}
    (nk_n)^{1/2} (\hat{\alpha}_n-\alpha_n) \\
    (n/k_n)^{1/2} (\hat{\mu}_n-\mu_n)
\end{bmatrix} 
$ converges weakly to $\mathcal{N}(0, \Omega^{-1}\Sigma \Omega^{-1})$ by CMT. Part $a)$ of Theorem~\ref{cls} follows. 

\subsection{Proof of Part $b)$ of Theorem~\ref{cls}}
\label{proofclsb}
We first prove the following lemma:
\begin{lemma}
\label{lemma3}
We have the joint weak convergence of the following terms:
\begin{enumerate}[$i)$]
\item $k_n^{-1} \alpha_n^{-n} X_{n}\overset{w}{\rightarrow}  Z$
\item $k_n^{-2}  \alpha_n^{-n}\sum_{t=1}^{n} X_{t-1} \overset{w}{\rightarrow} Z$,  
\item $2 k_n^{-3} \alpha_n^{-2n} \sum_{t=1}^{n} X^2_{t-1} \overset{w}{\rightarrow} Z^2$, 
\item $3 k_n^{-4} \alpha_n^{-3n}\sum_{t=1}^{n}X_{t-1}^3
\overset{w}{\rightarrow} Z^3$,
\item  $   
\displaystyle  \begin{bmatrix}
  \frac{1}{\sqrt{k_nX_n^3}} {\sum_{t=1}^n X_{t-1} W_t} \\
  \frac{1}{\sqrt{k_nX_n}}  {\sum_{t=1}^n W_t} 
  %e^{-\frac{n}{2k_n}} \sum_{t=1}^n W_t\\
%  \frac{1}{k^2_n}e^{-\frac{3n}{2k_n}} \sum_{t=1}^n X_t W_t 
  \end{bmatrix}\overset{w}{\longrightarrow}  \mathcal{N}\left(0,\sigma^2 \begin{bmatrix}
    1/3 & \frac{1}{2} \\
     \frac{1}{2} & 1
\end{bmatrix}\right)$, where random variable $Z$ is defined in Proposition~\ref{stationaritycir} and the Gaussian vector in $v)$ is independent of $Z$. 
\end{enumerate}

\end{lemma}
%This result is to be compared with Theorem 2.6 of \cite{fei2018limit}. In their mildly explosive linear AR(1) model, the error term is i.i.d., and then obtain a constant limit for $k_n^{-1} \alpha_n^{-n} X_{t}$. In our case, since the error term $W_t$ is also non stationary (and hence leaving to ``more uncertainty", the limit we obtain in part $a)$ is random. 
\begin{proof}[Proof of Lemma \ref{lemma3}]
We first prove $(i)$ directly under the normalization $\alpha_n^n$. Fix
$\lambda\ge0$ and set
\[
u_{n,0}=\frac{\lambda}{k_n\alpha_n^n},
\qquad
u_{n,m+1}=a_n(u_{n,m}), \qquad m\ge0 .
\]
Iterating the affine Laplace transform gives
\[
\mathbb{E}\left[
\exp\left\{-\lambda\frac{X_n}{k_n\alpha_n^n}\right\}
\right]
=
\exp\left\{
-X_0u_{n,n}-\sum_{m=0}^{n-1}b_n(u_{n,m})
\right\}.
\]
Using
\[
a_n(u)=\alpha_n u-\frac{\beta_n}{2}u^2+o(u^2),
\qquad
b_n(u)=\mu_n u+o(u),
\]
uniformly for the arguments above, the usual discrete Riccati calculation gives
\[
X_0u_{n,n}\to0,
\qquad
\sum_{m=0}^{n-1}b_n(u_{n,m})
\to
\frac{2\mu}{\sigma^2}
\log\left(1+\frac{\sigma^2\lambda}{2}\right).
\]
Hence
\[
\mathbb{E}\left[
\exp\left\{-\lambda\frac{X_n}{k_n\alpha_n^n}\right\}
\right]
\to
\left(1+\frac{\sigma^2\lambda}{2}\right)^{-2\mu/\sigma^2},
\]
which is the Laplace transform of the random variable $Z$ in
Proposition~\ref{stationaritycir}. This proves $(i)$.

We next prove $(ii)$--$(iv)$. Let
\[
R_{n,t}:=\frac{X_t}{k_n\alpha_n^t}.
\]
For every fixed $M>0$, Doob's inequality, the martingale decomposition
\[
R_{n,t}-R_{n,t-1}
=
\frac{\mu_n+W_t}{k_n\alpha_n^t},
\]
and the affine moment recursion implies
\[
\sup_{0\le r\le Mk_n}
\left|R_{n,n-r}-R_{n,n}\right|
\xrightarrow{p}0.
\]
Thus, for $j=1,2,3$,
\[
j\,k_n^{-(j+1)}\alpha_n^{-jn}
\sum_{t=1}^{n}X_{t-1}^j
=
\frac{j}{k_n}\sum_{r=1}^n
\alpha_n^{-jr}R_{n,n-r}^j .
\]
Splitting the last sum at $\lfloor Mk_n\rfloor$, the terminal-window part satisfies
\[
\frac{j}{k_n}\sum_{r=1}^{\lfloor Mk_n\rfloor}
\alpha_n^{-jr}R_{n,n-r}^j
=
R_{n,n}^j
\frac{j}{k_n}\sum_{r=1}^{\lfloor Mk_n\rfloor}\alpha_n^{-jr}
+o_p(1)
\Rightarrow
Z^j(1-e^{-jM}).
\]
The remaining part is negligible as $M\to\infty$; indeed, by the same
moment recursion,
\[
\lim_{M\to\infty}\limsup_{n\to\infty}
\mathbb{P}\left(
\frac{j}{k_n}\sum_{r>\lfloor Mk_n\rfloor}
\alpha_n^{-jr}R_{n,n-r}^j>\eta
\right)=0,
\qquad \eta>0.
\]
Letting $M\to\infty$ gives
\[
j\,k_n^{-(j+1)}\alpha_n^{-jn}
\sum_{t=1}^{n}X_{t-1}^j
\overset{w}{\rightarrow}
Z^j,
\qquad j=1,2,3.
\]
The cases $j=1,2,3$ give $(ii)$, $(iii)$, and $(iv)$, respectively.

   It remains to prove $(v)$. This is the same martingale-CLT argument as
in the proof of part $a)$ of Theorem~\ref{cls}, except that the mildly
stationary LLN is replaced by the explosive terminal-window limits. Let us define the deterministically normalized martingale
array
\[
\zeta_{n,t}:=
\begin{bmatrix}
k_n^{-2}\alpha_n^{-3n/2}X_{t-1}W_t\\
k_n^{-1}\alpha_n^{-n/2}W_t
\end{bmatrix}.
\]
Its conditional covariance matrix is
\[
\sum_{t=1}^n
E[\zeta_{n,t}\zeta_{n,t}'\mid\mathcal F_{t-1}]
=
\begin{bmatrix}
\frac{1}{k_n^4\alpha_n^{3n}}\sum X_{t-1}^2(\beta_nX_{t-1}+\delta_n)
&
\frac{1}{k_n^3\alpha_n^{2n}}\sum X_{t-1}(\beta_nX_{t-1}+\delta_n)
\\
\frac{1}{k_n^3 \alpha_n^{2n}}\sum X_{t-1}(\beta_nX_{t-1}+\delta_n)
&
\frac{1}{k_n^2 \alpha_n^{n}}\sum(\beta_nX_{t-1}+\delta_n)
\end{bmatrix}.
\]
Using $(ii)$, $(iii)$, $(iv)$, and
$\beta_n\to\sigma^2$, this converges in distribution to
\[
\sigma^2
\begin{bmatrix}
Z^3/3 & Z^2/2\\
Z^2/2 & Z
\end{bmatrix}.
\]
The conditional Lindeberg condition follows exactly as in the proof of
part $a)$ of Theorem~\ref{cls}, using Assumption~\ref{AssAffineMomentBis}.
Hence the martingale central limit theorem gives the stable convergence
\[
\sum_{t=1}^n\zeta_{n,t}
\Rightarrow
\sigma
\begin{bmatrix}
Z^{3/2}G_1\\
Z^{1/2}G_2
\end{bmatrix},
\]
where $(G_1,G_2)$ follows $\mathcal{N}\left(0, \begin{bmatrix}
    \frac{1}{3}   & \frac{1}{2} \\
    \frac{1}{2} & 1
\end{bmatrix}\right)$, and is independent of $Z$.

Finally, by $(i)$,
\[
\frac{X_n}{k_n \alpha_n^{n}}\Rightarrow Z.
\]
Therefore, by Slutsky's theorem,
\[
\begin{bmatrix}
(k_nX_n^3)^{-1/2}\sum_{t=1}^nX_{t-1}W_t\\
(k_nX_n)^{-1/2}\sum_{t=1}^nW_t
\end{bmatrix}
=
\begin{bmatrix}
\left(\frac{X_n}{k_n\alpha_n^{n}}\right)^{-3/2} & 0\\
0 & \left(\frac{X_n}{k_n \alpha_n^{n}}\right)^{-1/2}
\end{bmatrix}
\sum_{t=1}^n\zeta_{n,t}
\Rightarrow
\sigma
\begin{bmatrix}
G_1\\
G_2
\end{bmatrix}.
\]
This proves $(v)$, and the limit is independent of $Z$.
 
    \end{proof}

%Finally, the convergences $i)-iv)$ are joint due to the joint convergence in Lemma \ref{lemma3}. 
Let us now prove part $b)$ of Theorem~\ref{cls}.  $(v)$ of Lemma~\ref{lemma3} says that $\sum_{t=1}^n X_{t-1}W_t \approx \sigma G_1 k_n^2 \alpha_n^{3n/2} Z^{\frac{3}{2}}$ and $\sum_{t=1}^n W_t \approx  \sigma G_2 k_n \alpha_n^{n/2} Z^{\frac{1}{2}}$. %where $(G_1, G_2) \sim \mathcal{N}(0, \Omega)$. 
Straightforward algebra gives:

%Let us now prove part $b)$ of Theorem \ref{cls}.  $iv)$ of Lemma 4 says that $\sum_{t=1}^n X_{t-1}W_t \approx \sigma G_1 k_n^2 \alpha_n^{3n/2} Z^{\frac{3}{2}}$ and $\sum_{t=1}^n W_t \approx  \sigma G_2 k_n \alpha_n^{n/2} Z^{\frac{1}{2}}$ where $(G_1, G_2) \sim \mathcal{N}(0, \Omega)$. Straightforward algebra gives:
\begin{align*}
     \begin{bmatrix} \hat{\alpha}_n-\alpha_n \\
\hat{\mu}_n- \mu_n
 \end{bmatrix}
& =\begin{bmatrix}
\sum_{t=1}^n X_{t-1}^2 & \sum_{t=1}^n X_{t-1} \\
\sum_{t=1}^n X_{t-1} & n
\end{bmatrix}^{-1} 
 \begin{bmatrix}
\sum_{t=1}^n X_{t-1}W_{t} \\
\sum_{t=1}^n W_{t}
\end{bmatrix} \\
&\approx \begin{bmatrix}
\frac{1}{2}k_n^{3}  \alpha_n^{2n}Z^2 & k_n^{2}  \alpha_n^{n}Z \\
k_n^{2}  \alpha_n^{n}Z& n
\end{bmatrix}^{-1} 
 \begin{bmatrix}
\sigma G_1 k_n^2 \alpha_n^{3n/2} Z^{\frac{3}{2}} \\
 \sigma G_2 k_n \alpha_n^{n/2} Z^{\frac{1}{2}}
\end{bmatrix}\\
&= \frac{\sigma}{(\frac{n}{2}-k_n)k_n^{3}  \alpha_n^{2n}Z^2} \begin{bmatrix}
    (n G_1 -k_n G_2) k_n^2 \alpha_n^{3n/2}Z^{\frac{3}{2}} \\
    -\alpha_n^{5n/2}Z^{5/2}k_n^4 (G_1-G_2/2)
\end{bmatrix}\\
&\approx \frac{\sigma}{\frac{n}{2}k_n^{3}  \alpha_n^{2n}Z^2} \begin{bmatrix}
    n G_1 k_n^2 \alpha_n^{3n/2} Z^{\frac{3}{2}} \\
  -\alpha_n^{5n/2}Z^{5/2}k_n^4 (G_1-G_2/2)
\end{bmatrix}.
\end{align*}

Thus $\hat{\alpha}_n-\alpha_n$ converges at the rate $k_n \alpha_n^{n/2}$, whereas since $\frac{k_n}{n}\alpha_n^{n/2}$ increases to infinity, $\hat{\mu}_n- \mu_n$ does not converge. 
%and on the other side, by \cite[Theorem 1.c]{alaya2012parameter}, as $t$ increases to infinity, we have:     $$    e^{- t/2} \int_0^t \sqrt{\Upsilon_s} \mathrm{d} B_s \overset{w}{\rightarrow} YZ    $$    Thus we have the weak convergence of $\frac{1}{k_n}e^{-\frac{n}{k_n}} \sum_{t=1} W_t$ towards $YZ$.
%On the other hand, by elementary analysis and Proposition 1, it is straightforward to check that as $t$ increases to infinity, we have:$$e^{-t}\int_0^t \Upsilon_s \mathrm{d}s \rightarrow Z, \qquad a.s.$$
     \subsection{Proof of Theorem~\ref{plugin}}
    \label{proofplugin}
 Using the technique of \cite[Theorem~1]{giraitis2006uniform}, we obtain
$$
%\hat{\alpha}_{n}-\alpha _{n}=O_p(( 1-\alpha_n)^{1/2} n^{-1/2})=o_p(( 1-\alpha_n)^{1/2}).
\hat{\alpha}_{n}-\alpha _{n}=O_p(( 1-\alpha_n)^{1/2} n^{-1/2}) = O_p( n^{-1/2}k_n^{-1/2})
$$
As a consequence, 
%$$
%1- \hat{\alpha}_n=1- \alpha_n+o_p(1- \alpha_n).
%$$
\begin{align*}
    {\frac{n}{1- \hat{\alpha}_{n}}} & =  \underbrace{\frac{n}{1- \alpha_{n}}}_{O_p(nk_n)}+\underbrace{\frac{n(\hat{\alpha}_{n}-\alpha_{n})}{(1- \hat{\alpha}_{n})(1- \alpha_n)}}_{\frac{O_p(n^{1/2}k_n^{-1/2})}{O_p(k_n^{-2})}}   = {\frac{n}{1- \alpha_{n}}}+ s.o., \\
    {n (1- \hat{\alpha}_{n})} &= \underbrace{n (1- \alpha_n)}_{ O_p(nk_n^{-1})}+\underbrace{n(\alpha_{n}-\hat{\alpha}_{n})}_{O_p( n^{1/2}k_n^{-1/2})} =n (1- \alpha_n)+ s.o.,
\end{align*}
where $s.o.$ denotes a term of smaller order in
probability. This means that we can substitute $k_n$ by $\hat{k}_n=\frac{1}{1-\hat{\alpha}_n}$ in Theorem 4. Theorem~\ref{plugin} follows.  
\subsection{Proof of Theorem \ref{bootstrap}}
\label{proofbootstrap}
%The proof of part $a)$ is the same as Theorem \cite{chen2024unified}. 
First, by Jensen's inequality, there is no loss of generality to assume that $\eta$ is small enough, that is, it satisfies $3+3\eta/2<p$. In the following, we assume this is satisfied. Let
\[
\bar\delta_t^b=\delta_t^b-1,
\qquad
\hat W_t=X_t-\hat\alpha_nX_{t-1}-\hat\mu_n .
\]
Assume the bootstrap weights are independent of the data and satisfy
\[
E[\delta_t^b]=1,\qquad \operatorname{Var}(\delta_t^b)=1,\qquad
E|\delta_t^b-1|^{2+\eta}<\infty
\]
for some $\eta>0$ such that $3+3\eta/2<p$.

By the OLS normal equations,
\[
\sum X_{t-1}\hat W_t=0,
\qquad
\sum \hat W_t=0.
\]
Hence
\[
F_n^b(\hat\theta_n^b-\hat\theta_n)=d_n^b,
\]
where
\[
F_n^b=
\begin{bmatrix}
\sum\delta_t^bX_{t-1}^2 & \sum\delta_t^bX_{t-1}\\
\sum\delta_t^bX_{t-1} & \sum\delta_t^b
\end{bmatrix},
\quad
d_n^b=
\begin{bmatrix}
\sum\bar\delta_t^bX_{t-1}\hat W_t\\
\sum\bar\delta_t^b\hat W_t
\end{bmatrix}.
\]

First consider the normalized bootstrap design matrix. Its random-weight
part is controlled by truncation. For example,
\[
\frac1n\sum\bar\delta_t^bY_{n,t-1}^2
=
\frac1n\sum\bar\delta_t^bY_{n,t-1}^2
\mathbf 1_{\{Y_{n,t-1}\le R\}}
+
\frac1n\sum\bar\delta_t^bY_{n,t-1}^2
\mathbf 1_{\{Y_{n,t-1}>R\}} .
\]
Conditionally on the data, the variance of the first term is bounded by
$CR^4/n$. The second term is bounded in conditional $L^1$ by
\[
C\frac1n\sum Y_{n,t-1}^2\mathbf 1_{\{Y_{n,t-1}>R\}}
\le
\frac{C}{R}\frac1n\sum
Y_{n,t-1}^3\mathbf 1_{\{Y_{n,t-1}>R\}},
\]
which is negligible by Lemma~\ref{lem:ms-lln}. The other entries are easier.
Therefore
\[
\tilde F_n^b\overset{p^*}{\longrightarrow}\Omega
\]
in probability.

Now define the normalized bootstrap score
\[
\tilde d_n^b=
\begin{bmatrix}
n^{-1/2}k_n^{-3/2}\sum\bar\delta_t^bX_{t-1}\hat W_t\\
(nk_n)^{-1/2}\sum\bar\delta_t^b\hat W_t
\end{bmatrix}.
\]
Conditionally on the data, this is a sum of independent centered random
vectors. Its conditional covariance is
\[
V_n^*=
\begin{bmatrix}
\frac1{nk_n^3}\sum X_{t-1}^2\hat W_t^2
&
\frac1{nk_n^2}\sum X_{t-1}\hat W_t^2\\
\frac1{nk_n^2}\sum X_{t-1}\hat W_t^2
&
\frac1{nk_n}\sum \hat W_t^2
\end{bmatrix}.
\]

We show that \(V_n^*\to_p\Sigma\). For \(r=0,1,2\),
\[
\frac1{nk_n^{r+1}}\sum X_{t-1}^r
\{W_t^2-E[W_t^2\mid\mathcal F_{t-1}]\}
\overset{p}{\longrightarrow}0.
\]
Indeed, choose \(q>1\) such that \(q\le2\), \(2q\le p\), and
\((r+1)q\le p\). The affine moment recursion gives uniformly bounded
\(q\)-moments for the scaled martingale differences, and the von Bahr--Esseen inequality yields
\[
E\left|
\frac1n\sum k_n^{-(r+1)}X_{t-1}^r
\{W_t^2-E[W_t^2\mid\mathcal F_{t-1}]\}
\right|^q
=O(n^{1-q})\to0.
\]
Since
\[
E[W_t^2\mid\mathcal F_{t-1}]=\beta_nX_{t-1}+\delta_n,
\]
Lemma~\ref{lem:ms-lln} gives
\[
\frac1{nk_n^{r+1}}\sum X_{t-1}^rW_t^2
\overset{p}{\longrightarrow}
\sigma^2E[\Upsilon_\infty^{r+1}],
\qquad r=0,1,2.
\]

It remains to replace \(W_t\) by \(\hat W_t\). Put
\[
R_t=\hat W_t-W_t
=-(\hat\alpha_n-\alpha_n)X_{t-1}-(\hat\mu_n-\mu_n).
\]
By Theorem~\ref{cls}$(a)$,
\[
\hat\alpha_n-\alpha_n=O_p((nk_n)^{-1/2}),
\qquad
\hat\mu_n-\mu_n=O_p((k_n/n)^{1/2}).
\]
The only apparently fourth-order term is
\[
(\hat\alpha_n-\alpha_n)^2
\frac1{nk_n^3}\sum X_{t-1}^4
=
O_p(1)\frac1{n^2}\sum Y_{n,t-1}^4 .
\]
But
\[
\frac1{n^2}\sum Y_{n,t-1}^4
\le
\left(\frac{\max_{t\le n}Y_{n,t-1}}{n}\right)
\frac1n\sum Y_{n,t-1}^3
=o_p(1),
\]
because by Markov's inequality \(\max_{t\le n}Y_{n,t-1}/n=o_p(1)\), and
Lemma~\ref{lem:ms-lln} gives \(n^{-1}\sum Y_{n,t-1}^3=O_p(1)\).
Thus, for \(r=0,1,2\),
\[
\frac1{nk_n^{r+1}}\sum X_{t-1}^rR_t^2=o_p(1).
\]
By Cauchy-Schwarz,
\[
\frac1{nk_n^{r+1}}\sum X_{t-1}^r(\hat W_t^2-W_t^2)=o_p(1).
\]
Hence \(V_n^*\to_p\Sigma\).

Finally, the conditional Lindeberg condition follows from
\(E|\bar\delta_t^b|^{2+\eta}<\infty\) and \(3+3\eta/2<p\):
\[
\sum_t E^*\!\left[
\|\bar\delta_t^b z_{n,t}\|^2
\mathbf 1_{\{\|\bar\delta_t^b z_{n,t}\|>\varepsilon\}}
\right]
\le
C\varepsilon^{-\eta}\sum_t\|z_{n,t}\|^{2+\eta}
=o_p(1),
\]
where
\[
z_{n,t}=
\begin{bmatrix}
n^{-1/2}k_n^{-3/2}X_{t-1}\hat W_t\\
(nk_n)^{-1/2}\hat W_t
\end{bmatrix}.
\]
Therefore, conditionally on the data,
\[
\tilde d_n^b\overset{w^*}{\longrightarrow}\mathcal N(0,\Sigma)
\]
in probability. Combining this with \(\tilde F_n^b\to\Omega\) gives
\[
\begin{bmatrix}
\sqrt{nk_n}(\hat\alpha_n^b-\hat\alpha_n)\\
\sqrt{n/k_n}(\hat\mu_n^b-\hat\mu_n)
\end{bmatrix}
\overset{w^*}{\longrightarrow}
\mathcal N(0,\Omega^{-1}\Sigma\Omega^{-1}).
\]

%and $$e^{-3t}\int_0^t \Upsilon^3_s \mathrm{d}s \rightarrow Z^3 $$
 
 \subsection{Proof of Proposition \ref{convergencebeta}}
Using the same technique as in the proof of Theorem~\ref{cls}, we have:
     $$
 (nk_n)^{-1}    \sum_{t=1}^n W_t^2  \overset{w}{\rightarrow} \sigma^2 \mathbb{E}[\Upsilon_\infty]
     $$
     On the other hand, by the proof of Lemma~\ref{lem:ms-lln}, we have:
     $$
     (nk_n)^{-1}    \sum_{t=1}^n X_{t-1}   \overset{w}{\rightarrow} \mathbb{E}[\Upsilon_\infty]
     $$
Then we remark that:
$$
\hat{W}_t=W_t+ (\alpha_n-\hat{\alpha}_n)X_{t-1}+(\mu_n-\hat{\mu}_n)
$$
Thus $$\sum_{t=1}^n \hat{W}_t^2 =\sum_{t=1}^n W_t^2 + \sum_{t=1}^n [(\alpha_n-\hat{\alpha}_n)X_{t-1}+(\mu_n-\hat{\mu}_n)]^2+2 \sum_{t=1}^n W_t [(\alpha_n-\hat{\alpha}_n)X_{t-1}+(\mu_n-\hat{\mu}_n)].$$

It is easily checked that the second term is negligible with respect to the denominator $\sum_{t=1}^n X_{t-1}$, and by Cauchy-Schwarz, the third, cross term, is also negligible. Finally, taking the ratio of these two equations leads to $ \hat{\sigma}^2= \frac{\sum_{t=1}^n \hat{W}_t^2}{ \sum_{t=1}^n X_{t-1}}  \overset{w}{\rightarrow} \sigma^2.
$

\subsection{Proof of Proposition \ref{transient}}
\label{prooftransient}
This proposition is based on the following lemma, which is the analog of Lemma~\ref{lem:ms-lln}. %Lemma 2
\begin{lemma}
\label{lastlemma}
Under the assumption $\mu> \frac{\sigma^2}{2}$, we have the joint weak convergence of the random variables:
 \begin{itemize}
\item $i)$. $\frac{k_n}{n} \sum_{t=1}^n  \frac{1}{ X_{t-1}+1}    \overset{w}{\rightarrow} \mathbb{E}[1/\Upsilon_\infty] $ 
\item $ii)$. $\frac{1}{nk_n} \sum_{t=1}^n  \frac{ X^2_{t-1}}{ X_{t-1}+1 }  \overset{w}{\rightarrow} \mathbb{E}[\Upsilon_\infty]$ 
 \item $iii)$. $ \frac{1}{n}\sum_{t=1}^n  \frac{ X_{t-1}}{ X_{t-1}+1} \overset{w}{\rightarrow} 1    $ 
\item $iv)$. $   \begin{bmatrix}
n^{-1/2}k_{n}^{-1/2}\sum_{t=1}^{n}\frac{X_{t-1}W_{t}}{ X_{t-1}+1}  \\
(n/k_{n})^{-1/2}\sum_{t=1}^{n}\frac{W_{t}}{ X_{t-1}+1} 
\end{bmatrix}
\overset{w}{\rightarrow } \mathcal{N}(0, \sigma^2 \Omega_1),$ where $\Omega_1=\begin{bmatrix}
    \mathbb{E}[\Upsilon_\infty] & 1 \\
    1 &  \mathbb{E}[1/\Upsilon_\infty]
\end{bmatrix}$.  
\end{itemize}
\end{lemma}

\begin{proof}[Proof of Lemma~\ref{lastlemma}]
To investigate the term $S_n:=\sum_{t=1}^n  \frac{1}{ X_{t-1}+1}$ in convergence $i)$, we write:
$$\sum_{t=1}^n\frac{1}{ X_{t-1}+1}=\sum_{t=1}^n\frac{1}{k_n} \frac{1}{\frac{1}{k_n}+\frac{X_{t-1}}{k_n}}
%$\sum_{t=1}^n\frac{1}{n} \frac{1}{\frac{1}{n}+\frac{X_{t-1}}{n}}
\approx \sum_{t=1}^n\frac{1}{k_n} \frac{1}{\frac{X_{t-1}}{k_n}}.$$ 
Since $X_{\lfloor k_ns \rfloor}/k_n$ converges weakly, we have $$\sum_{t=1}^n\frac{1}{n} \frac{1}{\frac{X_{t-1}}{k_n}} \approx \frac{k_n}{n}\int_0^{n/k_n} \frac{1}{\Upsilon_s}\mathrm{d}s \rightarrow  \mathbb{E}[1/\Upsilon_\infty] ,$$ provided that the latter limit is finite. Since the diffusion $(\Upsilon_s)$ has a gamma invariant distribution with shape parameter $2\mu/\sigma^2$, the invariant distribution of $\frac{1}{\Upsilon_s}$ is inverse gamma with the same shape parameter, which has an infinite mean if $2\mu/\sigma^2<1$. Thus we get convergence $i).$ Convergences $ii)$ and $iii)$ are simple consequences of $i)$. Convergence $iv)$ can be proved in a similar way as convergence $b)$ in Lemma 3. 
\end{proof}

Let us now prove Proposition \ref{transient}. We introduce the rescaled square matrix
\begin{align*}\tilde{H}_n:&= \begin{bmatrix}
(nk_n)^{-1/2}   & 0 \\
0 & (n/k_n)^{-1/2}
\end{bmatrix} H_n \begin{bmatrix}
(nk_n)^{-1/2}   & 0 \\
0 & (n/k_n)^{-1/2}
\end{bmatrix}\\
&=\begin{bmatrix}
    (nk_n)^{-1} \sum_{t=1}^n \frac{ X^2_{t-1}}{ X_{t-1}+1 }   & {n^{-1 }} \sum_{t=1}^n  \frac{ X_{t-1}}{ X_{t-1}+1 } \\
  {n^{-1 }} \sum_{t=1}^n  \frac{ X_{t-1}}{ X_{t-1}+1 }   & \frac{k_n}{n} \sum_{t=1}^n  \frac{1}{ X_{t-1}+1}
\end{bmatrix} \overset{p}{\rightarrow }   \Omega_1,  
\end{align*}
and
$$\tilde{v}_n:=\begin{bmatrix}
(n k_{n})^{-1/2} & 0 \\
0 & (n/k_n)^{-1/2}
\end{bmatrix} v_n=\begin{bmatrix}
n^{-1/2}k_{n}^{-1/2}\sum_{t=1}^{n}\frac{X_{t-1}W_{t}}{ X_{t-1}+1}  \\
(n/k_{n})^{-1/2}\sum_{t=1}^{n}\frac{W_{t}}{ X_{t-1}+1} 
\end{bmatrix} \overset{w}{\rightarrow } \mathcal{N}(0, \sigma^2 \Omega_1).
$$
Thus:
\begin{align*}
    H_n^{-1}v_n&= \begin{bmatrix}
(nk_n)^{-1/2}   & 0 \\
0 & (n/k_n)^{-1/2}
\end{bmatrix} \tilde{H}_n^{-1} \begin{bmatrix}
(nk_n)^{-1/2}   & 0 \\
0 & (n/k_n)^{-1/2}
\end{bmatrix}
\begin{bmatrix}
(n k_{n})^{1/2} & 0 \\
0 & (n/k_n)^{1/2}
\end{bmatrix}
\tilde{v}_n \\
&= \begin{bmatrix}
(nk_n)^{-1/2} & 0 \\
0 & (n/k_n)^{-1/2}
\end{bmatrix} \tilde{H}_n^{-1} \tilde{v}_n.
\end{align*}
Hence the distribution of 
 $ 
\begin{bmatrix}
(nk_n)^{1/2} & 0 \\
0 & (n/k_n)^{1/2}
\end{bmatrix} \tilde{H}_n^{-1} \tilde{v}_n=
\begin{bmatrix}
    (nk_n)^{1/2} (\hat{\hat{\alpha}}_n-\alpha_n) \\
    (n/k_n)^{1/2} (\hat{\hat{\mu}}_n-\mu_n)
\end{bmatrix} 
$  converges weakly to $\mathcal{N}(0,\sigma^2\Omega_1^{-1}\Omega_1\Omega_1^{-1})=\mathcal{N}(0,\sigma^2 \Omega_1^{-1})$ by CMT. Proposition \ref{transient} follows. 

\section{Benchmark comparison with \cite{peng2024note}}

As a benchmark comparison, we apply the \(p(\alpha)\) calculation to break-and-enter dwelling crime count data used by \citet{peng2024note}. The data are monthly counts for three municipalities (Inner West, Blacktown, and Sydney) from January 1995 to September 2023. \cite{peng2024note} consider the mildly stationary INAR(1) specification for these data, corresponding to the limiting case of our framework with $\sigma^2=0$. Using our estimator $\hat{\sigma}^2$, we estimate the value of $\sigma^2$ to be 6.5, 4.3 and 7.15 for these three cities, respectively. %these dataset seem to feature conditional over-dispersion. 
Since the INAR model implies conditional under-dispersion, whereas our framework allows for both under- and over-dispersion, an estimate of $\hat{\sigma}^2$ that is substantially larger than 1 suggests potential over-dispersion of the data, thus favoring our affine framework. The OLS estimates are \(\hat\alpha_n=0.962,0.931,0.924\), respectively,
which are close to those reported by \citet{peng2024note}. The resulting $\hat{k}_n$ is 26, 14 and 13, respectively, which is much smaller than $n=344$ for each series, making the mildly stationary framework potentially appropriate. Figure~\ref{fig:peng-crime-pvalue}
plots the plug-in and bootstrap-normal \(p(\alpha)\) functions.

\begin{figure}[H]
    \centering
    \includegraphics[width=0.98\textwidth]{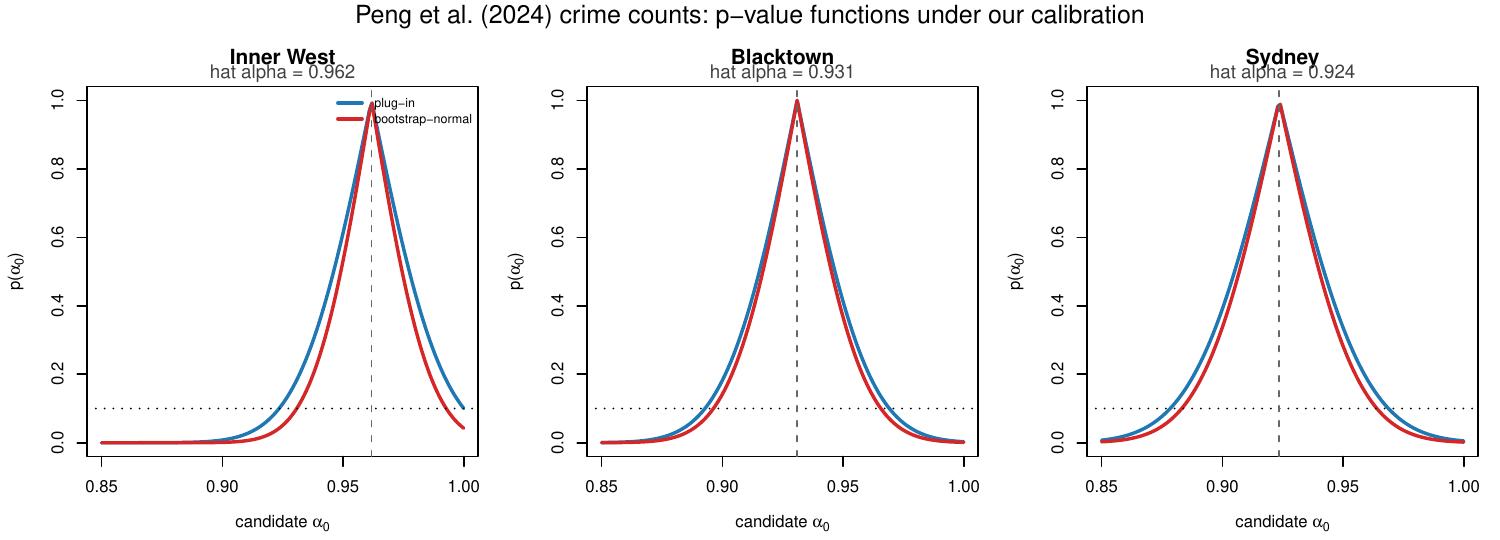}
    \caption{Plug-in and bootstrap-normal \(p(\alpha)\) functions for the
    break-and-enter dwelling counts used by \citet{peng2024note}.  The
    horizontal dotted line is the 10\% significance level, and the vertical
    dashed line marks \(\hat\alpha\).}
    \label{fig:peng-crime-pvalue}
\end{figure}

The resulting non-rejection regions are broadly comparable to those obtained
for the count applications in Section~\ref{sec:7empirical}.  Unlike our main
applications, however, the upper boundary \(\alpha=1\) is rejected for
Blacktown and Sydney under both plug-in and bootstrap-normal inference, and is
borderline for Inner West.  This illustrates that non-rejection near the upper
grid boundary is not a mechanical feature of the \(p(\alpha)\) construction,
but reflects the uncertainty and persistence structure of the particular
sample.

\end{document}